\documentclass[12pt,amstex]{amsart}
\usepackage{}
\usepackage{txfonts}
\usepackage{geometry}
\usepackage{amsmath, amsthm, amssymb,amsfonts}
\usepackage{epsfig}
\usepackage{amsmath}
\usepackage{amssymb}
\usepackage{amscd}
\usepackage{graphicx}
\usepackage{pstricks}
\usepackage{pst-node}
\usepackage[english]{babel}
\usepackage{float}
\theoremstyle{definition}
\newtheorem{thm}{Theorem}[section]
\newtheorem{lem}[thm]{Lemma}
\newtheorem{defn}[thm]{Definition}
\newtheorem{rem}[thm]{Remark}
\newtheorem{prop}[thm]{Proposition}

\newtheorem{cor}[thm]{Corollary}

\newtheorem{conv}[thm]{Convention}

\def \Z {\mathbb{Z}[i]}
\def \P {\mathbb{P}[i]}
\def \N {\mathbb{N}}
\def \A {\alpha}
\def \B {\beta}
\def \C {\gamma}
\def \D {\kappa}
\def \d {\delta}
\def \e {\epsilon}
\def \R {R_{\tilde{N}}}

\def \c {\bold{c}}
\def \a {\bold{a}}
\def \b {\bold{b}}
\def \v {\bold{v}}
\def \w {\bold{w}}
\def \x {\bold{\xi}}

\def \k {\bold{k}}
\def \p {\bold{p}}
\def \l {\bold{\lambda}}

\def \ps {\bold{\psi}}

\title[A structure theorem for multiplicative functions and applications]{A structure theorem for multiplicative functions over the Gaussian integers and applications}
\author{Wenbo Sun}
\address{Department of Mathematics, Northwestern University, 2033 Sheridan Road Evanston, IL 60208-2730, USA}
\email{swenbo@math.northwestern.edu}

\thanks{The author is partially supported by NSF grant 1200971‏} 

\begin{document}

\maketitle

\begin{abstract}
We prove a structure theorem for multiplicative functions on the Gaussian integers, showing that every bounded multiplicative function on the Gaussian integers can be decomposed into a term which is approximately periodic and another which has a small $U^{3}$-Gowers uniformity norm. We apply this to prove partition regularity results over the Gaussian integers for certain equations involving quadratic forms in three variables. For example, we show that for any finite coloring of the Gaussian integers, there exist distinct nonzero elements $x$ and $y$ of the same color such that $x^{2}-y^{2}=n^{2}$ for some Gaussian integer $n$. The analog of this statement over $\mathbb{Z}$ remains open.
\end{abstract}

\section{Introduction}
\subsection{Structure theory in the finite setting}
The structure theorem for functions on $\mathbb{Z}^{d}$ is an important tool in additive combinatorics. It has been studied extensively in ~\cite{FH}, ~\cite{14}, ~\cite{15}, ~\cite{22}, ~\cite{5},~\cite{7},~\cite{6}, ~\cite{Sz} and ~\cite{T}. Roughly speaking, the structure theorem says that every function $f$ can be decomposed into one part with a good uniformity property, meaning it has a small Gowers norm, and another with a good structure, meaning it is a nilsequence with bounded complexity.

A natural question to ask is: can we get a better decomposition for functions $f$ satisfying special conditions? For example, Green, Tao and Ziegler (~\cite{5},~\cite{7},~\cite{6}) gave a refined decomposition result for the von Mangoldt function $\Lambda$. They showed that under some modification, one can take the structured part to be the constant 1.

In this paper, we focus on the class of multiplicative functions:

\begin{defn}[Multiplicative function]
  Let $\mathbb{F}$ be a number field. A \emph{multiplicative function} is a function $\chi\colon\mathbb{F}\rightarrow\mathbb{C}$ that satisfies $\chi(mn)=\chi(m)\chi(n)$ for all $m,n\in\mathbb{F}$. We denote the family of multiplicative functions of modulus 1 by $\mathcal{M}_{\mathbb{F}}$.
\end{defn}

 Recent work of Frantzikinakis and Host ~\cite{FH} provided a decomposition result for
bounded multiplicative functions on $\mathbb{Z}$.
 They showed that for any multiplicative function $\chi$ on $\mathbb{Z}$ and any $d\in\mathbb{N}$, one can decompose $\chi$ into the sum of two functions $\chi_{s}+\chi_{u}$ plus an error term such that $\chi_{s}$ is an ``approximately periodic" function and $\chi_{u}$ has small $U^{d}$-Gowers norm. Moreover, $\chi_{s}$ can be written as the convolution of $\chi$ and a well-behaved function.

 Frantzikinakis and Host asked in ~\cite{FH} whether the structure theorem still holds for multiplicative functions in a general number field $\mathbb{F}$. In this paper, we give an affirmative answer for the case when $d=3$ and $\mathbb{F}=\Z$, where $\Z$ denotes the set of Gaussian integers. We show that any multiplicative function $\chi$ of $\mathbb{Z}[i]$ can be decomposed into the sum of two functions $\chi_{s}+\chi_{u}$ plus an error term such that $\chi_{s}$ is an "approximately periodic" function and $\chi_{u}$ has a small $U^{3}$-Gowers norm (see Section 2 for definitions). Moreover, $\chi_{s}$ can be written as the convolution of $\chi$ and a well-behaved function. The precise statement is Theorem \ref{nU3s}.

\subsection{Partition regularity results for quadratic forms} We use this decomposition result to determine some combinatorial consequences.
Determining whether an algebraic equation (or a system of equations) is partition regular is widely studied in Ramsey theory.
In this article, we restrict our attention to polynomials in three variables. Specifically, we study the following question:
suppose $p(x,y,z)$ is a polynomial of 3 variables over some number field $\mathbb{F}$. For any finite coloring of $\mathbb{F}$, can we find distinct numbers $x,y,z\in\mathbb{F}$ of the same color such that $p(x,y,z)=0$?

The case when the polynomial $p$ is linear and $\mathbb{F}=\mathbb{Z}$  was completely solved by Rado ~\cite{Rado}: for $a,b,c\in\mathbb{Z}$, the equation
$ax+by+cz=0$ is partition regular if and only if either $a+b, b+c, a+c$ or $a+b+c$ is 0 (the original result is stated for $\mathbb{N}$ but a similar result holds for $\mathbb{Z}$).
However, little is known for equations of higher degrees or over number fields other than $\mathbb{Z}$.

An easier and related question is to study the partition regularity of some polynomial equation under the relaxed condition that one of the variables is allowed to vary freely in $\mathbb{F}$ and not necessarily lie in the same piece of the partition. We define:
\begin{defn}[Partition regularity]
  Let $\mathbb{F}$ be a number field, or a subset of some number field. An equation $p(x,y,n)\colon \mathbb{F}^{3}\rightarrow\mathbb{C}$ is \emph{partition regular} in $\mathbb{F}$ if
  for any partition of $\mathbb{F}$ into finitely many disjoint sets, for some $n\in \mathbb{F}$, one of the cells contains distinct $x$ and $y$ that satisfy $p(x,y,n)=0$.
\end{defn}
 It is a classical result of Furstenberg ~\cite{7} and Sark\"{o}zy ~\cite{25} that the equation $x-y=n^{2}$ is partition regular in $\mathbb{N}$. Bergelson and Leibman ~\cite{BL} provided other examples of translation invariant equations by proving a polynomial version of the van der Waerden Theorem. However, little is
known for the case which is not translation invariant. A result of Khalfalah and Szemer\'{e}di ~\cite{KS} is that the equation $x+y=n^{2}$ is partition regular in $\mathbb{N}$.

Recent work of Frantzikinakis and Host ~\cite{FH} showed the connection between the decomposition result for multiplicative functions on $\mathbb{Z}$ and partition regularity problems for certain equations. For example, they proved that the equation $ax^{2}+by^{2}=n^{2}$ is partition regular in $\mathbb{N}$ if $a,b,a+b$ are non-zero square integers (for example, $a=16, b=9$). However, the partition regularity result in $\mathbb{N}$ for other quadratic equations, for example, $x^{2}+y^{2}=n^{2}$ or $x^{2}-y^{2}=n^{2}$, remains open.

Another question is to seek partition regularity results for number fields other than $\mathbb{Z}$. In this paper, we restrict our attention to the field of Gaussian integers $\mathbb{Z}[i]$. Since $\mathbb{N}$ is a subset of $\mathbb{Z}[i]$, every polynomial equation which is partition regular in $\mathbb{N}$ will also be partition regular in $\mathbb{Z}[i]$. In this paper, we show that there are certain quadratic equations for which we do not currently know if they are partition regular in $\mathbb{N}$, but are partition regular in $\mathbb{Z}[i]$. For example, it is not known whether the equation $p(x,y,n)=x^{2}-y^{2}-n^{2}$ is partition regular or not in $\mathbb{N}$, but we show (we remove 0 from $\Z$ to avoid trivial solutions):
\begin{cor}\label{c11}
  The equation $p(x,y,n)=x^{2}-y^{2}-n^{2}$ is partition regular in $\Z\backslash \{0\}$. Equivalently, the system of equations
  \begin{equation} \nonumber
  \begin{split}
          & p_{1}((x_{1},x_{2}),(y_{1},y_{2}),(n_{1},n_{2}))=x_{1}^{2}-x_{2}^{2}-y_{1}^{2}+y_{2}^{2}-n_{1}^{2}+n_{2}^{2};
          \\& p_{2}((x_{1},x_{2}),(y_{1},y_{2}),(n_{1},n_{2}))=x_{1}x_{2}-y_{1}y_{2}-n_{1}n_{2}
   \end{split}
   \end{equation}
is partition regular in $\mathbb{Z}^{2}\backslash \{(0,0)\}$.
\end{cor}
The statement in full generality is given in Theorem \ref{c1}. The crucial property of the equation $p(x,y,n)=x^{2}-y^{2}-n^{2}$ is that its solutions can be parameterized as
  \begin{equation} \nonumber
  \begin{split}
          x=\C\A(\A+2\B), y=\C(\A+(1+i)\B)(\A+(1-i)\B),\A,\B,\C\in\Z.
   \end{split}
   \end{equation}
With the help of the $U^{3}$-structure theorem in this paper, for most choices of $\C_{i},\C'_{i},1\leq i\leq 2$, one can find patterns $\C\prod_{i=1}^{2}(\A+\C_{i}\B)$ and $\C\prod_{i=1}^{2}(\A+\C'_{i}\B)$ in the same cell for any finite partition of $\Z$.

In general, if Theorem \ref{nU3s} holds for $U^{s+1}$ Gowers norm for some $s\geq 3$, then for most choices of $\C_{i},\C'_{i},1\leq i\leq s$, one can find patterns $\C\prod_{i=1}^{s}(\A+\C_{i}\B)$ and $\C\prod_{i=1}^{s}(\A+\C'_{i}\B)$ in the same cell for any finite partition of $\Z$. Therefore it is natural to ask: does Theorem \ref{nU3s} hold for $U^{s}$ Gowers norms for $s\geq 3$?

\subsection{Outline and method of the paper}
The precise statements of the main theorems are given in Section 2.
The method for deriving partition regularity results from the structure theorem in this paper is the same as ~\cite{FH}, and we briefly review this method in Section 3. We prove the main number theory input needed to derive the structure theorem in Section 4.

The remaining sections are devoted to the proof of the $U^{3}$-structure theorem (Theorem \ref{nU3s}). The first step is to prove the $U^{2}$-structure theorem: we show in Section 5 that every bounded multiplicative function $\chi$ on $\Z$ can be decomposed into a term $\chi_{s}$ which is approximately periodic and another term $\chi_{u}$ which has a small $U^{2}$-Gowers uniformity norm. The result in this section extends Theorem 3.3 in ~\cite{FH} to the 2 dimensional case. Their ideas are similar, but an additional technique dealing with the sum of exponential functions on a convex set in $\mathbb{Z}^{2}$ rather than an interval in $\mathbb{Z}$ is needed in our setting.

The second step is to show that if $\chi_{u}$ has sufficiently small $U^{2}$-norm, then it also has a small $U^{3}$-norm. To do so, we first introduce the inverse and factorization theorems in Section 6, which allow us to convert the estimate of Gowers uniformity norms to a problem about the correlation of multiplicative functions with polynomial sequences which is carefully studied in Section 7. In the end, we provide the complete proof of the structure theorem with the help of all these materials in Section 8.

Figure 1 illustrates the various dependence of the results in this paper:

\setlength{\unitlength}{0.5mm}
\begin{picture}(300,260)
\put(70,250){\textbf{Theorem \ref{c1}}}
\put(68,230){\vector(1,1){15}}
\put(35,220){Corollary \ref{c2}}
\put(102,230){\vector(-1,1){15}}
\put(90,220){Lemma \ref{ge}}
\put(50,200){\vector(0,1){15}}
\put(35,190){Corollary \ref{c3}}
\put(33,170){\vector(1,1){15}}
\put(67,170){\vector(-1,1){15}}
\put(0,160){Theorem \ref{FC}}
\put(55,160){Proposition \ref{c4}}
\put(123,162){\vector(-1,0){15}}
\put(125,160){\textbf{Theorem \ref{nU3s}}}
\put(140,140){\vector(0,1){15}}
\put(125,130){Theorem \ref{nU3w}}
\put(55,130){Theorem \ref{nU2}}
\put(108,132){\vector(1,0){15}}
\put(0,90){Corollary \ref{mIU3}}
\put(70,90){Corollary \ref{mF}}
\put(140,90){Proposition \ref{est2}}
\put(210,90){Proposition \ref{est1}}
\put(45,100){\vector(3,1){80}}
\put(115,100){\vector(1,1){27}}
\put(180,100){\vector(-1,1){27}}
\put(225,100){\vector(-2,1){55}}
\put(0,60){Theorem \ref{inv2}}
\put(70,60){Theorem \ref{mFo}}
\put(210,60){Proposition \ref{last}}
\put(25,70){\vector(0,1){15}}
\put(95,70){\vector(0,1){15}}
\put(235,70){\vector(0,1){15}}
\put(210,30){Lemma \ref{appro2}}
\put(235,40){\vector(0,1){15}}
\put(210,250){Frequently used properties:}
\put(210,235){Lemma \ref{katai}}
\put(210,220){Theorem \ref{Lei}}
\put(210,205){Theorem \ref{inv}}
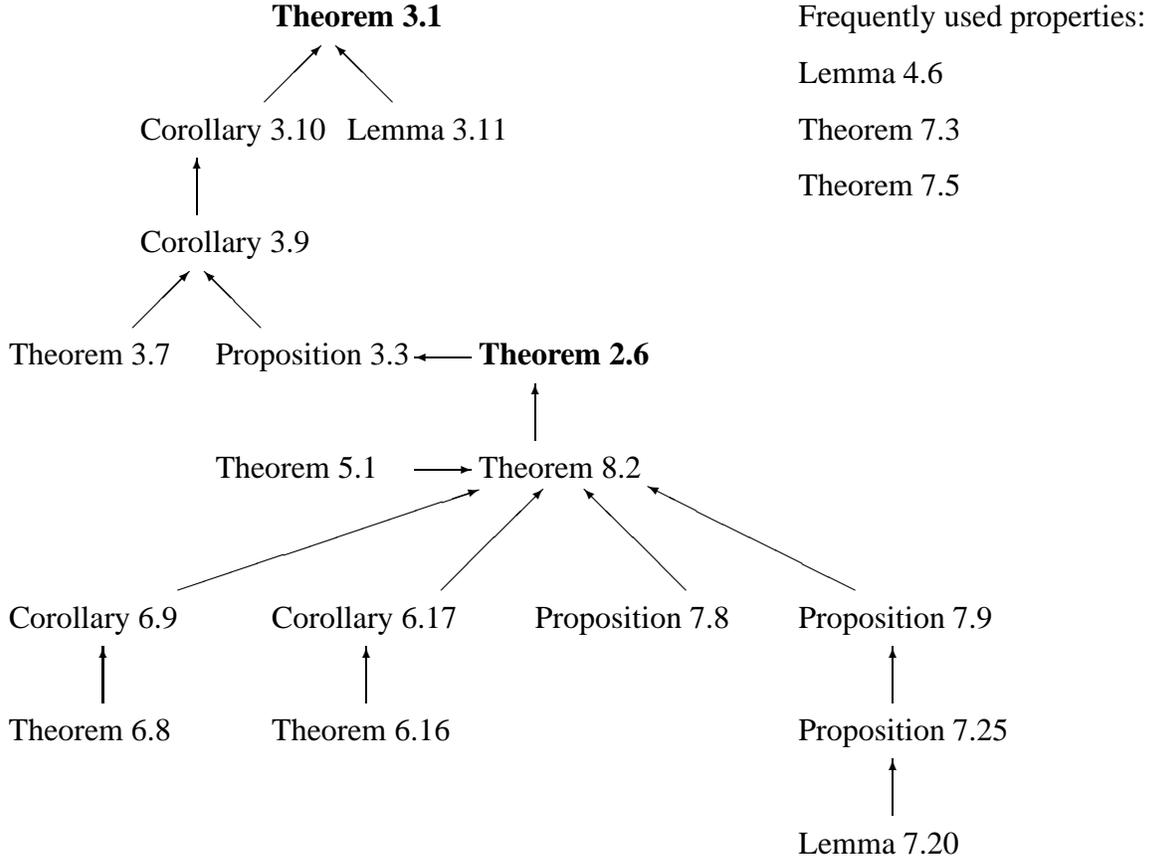
\put(100,20){Figure 1: roadmap of the paper}
\end{picture}

It is worth noting that most of the methods in this paper other than the ones in Section 7, work for $s\geq 3$, but there are certain results, for example Corollary \ref{mF} (see Remark \ref{mF2}), that require more work in order to be applied for the general case.

While the outline of Sections 6, 7 and 8 can be viewed as an extension of the corresponding parts of ~\cite{FHel} restricted to the case of nilmanifolds of order 2, there are two key differences between this paper and ~\cite{FHel}. The first difference is that this paper uses the inverse theorem from ~\cite{Sz} for subsets of $\mathbb{Z}^{2}$, while the paper ~\cite{FHel} used the inverse theorem from ~\cite{6} for subsets of $\mathbb{Z}$. The second and the most crucial difference is that there is an additional difficulty to overcome in this paper, and this is the content of Proposition \ref{last}.
  Roughly speaking, we wish to show that for a nilmanifold $X$ of order 2 (see Section 6.1), if a sequence of elements $(a^{n},b^{n})_{n\in\N}$ is not totally equidistributed on $X\times X$ (see Definition \ref{equid}), then the sequences
  $(a^{n})_{n\in\N}$ and $(b^{n})_{n\in\N}$ are not totally equidistributed on $X$ with only a few exceptions. The case when the Kronecker parts of $a$ and $b$ are parallel to each other was proved implicitly in ~\cite{FH}, but we need new techniques for the general case. This problem occupies the bulk of the paper.

  By the Inverse Leibman Theorem (Theorem \ref{inv}), it suffices to show that if the sequence $(a^{n},b^{n})_{n\in\N}$ is not totally equidistributed on $X\times X$, then
  both the coordinates of $a$ and $b$ are (roughly speaking) sufficiently ``linearly dependent over $\mathbb{Q}$"  with only a few exceptions.
  The main tool of the proof is a modified version of the quantitative Leibman Theorem proved in ~\cite{er}, ~\cite{er2} and ~\cite{GT}, which says that if $(a^{n},b^{n})_{n\in\N}$ is not totally equidistributed on $X\times X$, then the coordinates of $(a,b)$ is sufficiently ``linearly dependent over $\mathbb{Q}$". But this is not enough for our purpose as we need the coordinates of both $a$ and $b$ to be sufficiently linearly dependent.

  To overcome this obstacle, we use Leibman Theorem repeatedly instead of using it only once: if a sequence $(a^{n},b^{n})_{n\in\N}$ is not totally equidistributed on $X\times X$,
  we can find some $(a',b')$ lying in a submanifold $Y$ of $X\times X$ such that the sequence $(a'^{n},b'^{n})_{n\in\N}$ is not totally equidistributed on $Y$, and $(a',b')$ is ``close to'' $(a,b)$. We then continue this process with $X\times X$ replaced by $Y$. Roughly speaking, the ``smaller" a system is, the more ``linearly dependent over $\mathbb{Q}$" a non-equidistributed sequence on this system will be. We show that with only a few exceptions, if we do this process sufficiently many times and reduce the problem from $X\times X$ to a sufficiently ``small" system, we can obtain enough information to deduce that both the coordinates of $a$ and $b$ are sufficiently ``linearly dependent over $\mathbb{Q}$", which proves Proposition \ref{last}.

\textbf{Acknowledgment.} The author thanks Nikos Frantizinakis for comments about this paper, Bryna Kra for the patient guidance and all the useful advice, and the referee for the careful reading and helpful suggestions.

\section{Precise statement of the main theorems}
Before precisely stating the main theorems, we give some notation and review some definitions.
Denote $R_{x}=\{\A=a+bi\in\Z\colon 1\leq a, b\leq x\}$.
We follow the following convention throughout this paper:

\begin{conv}\label{conv}
Whenever we write an element in $\Z$ in the form $a+bi$, we always assume without explicitly writing this that $a,b\in\mathbb{Z}$.

Throughout, we assume that an integer $\ell\in\mathbb{N}$ is given: its precise value depends on the applications we have in mind. We consider $\ell$ as fixed and the dependence on $\ell$ is always left implicit.
For $N\in\N$, let $\tilde{N}$ be the
smallest prime number (in $\N$) greater than $100\ell N$. Therefore, by Bertrand's postulate, $\tilde{N}\leq 200\ell N$.
Throughout this paper, we always assume $\tilde{N}$ is the integer dependent on $N$ defined as above.

For any function $f\colon\R\rightarrow\mathbb{C}$, we use the convention that $f(a+bi)=f((a+\tilde{N})+bi)=f(a+(b+\tilde{N})i)$, meaning that the addition is taken mod $\R$.
\end{conv}

The reason that we work on the set $\R$ rather than $R_{N}$ is that if $\vert\gamma\vert^{2}<\tilde{N}$, the map $\B\to\C\B$ is a bijection from $\R$ to itself (see also the discussion after Proposition \ref{c4} for the reason).
We start with the definitions of the convolution product and the Fourier transformation on $\R$:

\begin{defn}[Convolution product] The convolution product of two functions $f,g\colon\R\rightarrow\mathbb{C}$ is defined by
  \begin{equation}\nonumber
  \begin{split}
   f*g(\A)=\mathbb{E}_{\B\in \R}f(\A-\B)g(\B),
  \end{split}
  \end{equation}
where for any function $f\colon S\rightarrow\mathbb{C}$ on a finite set $S$, we denote $\mathbb{E}_{x\in S}f(x)=\frac{1}{\vert S\vert}\sum_{x\in S}f(x)$.
Here $\vert S\vert$ denotes the cardinality of $S$.
\end{defn}

\begin{defn}[Fourier transformation]  For any $\A=\A_{1}+\A_{2}i,\B=\B_{1}+\B_{2}i\in\R$, write
  \begin{equation}\nonumber
  \begin{split}
   \A\circ_{N}\B=\frac{1}{\tilde{N}}(\A_{1}\B_{1}+\A_{2}\B_{2}).
  \end{split}
  \end{equation}
  For any function $f\colon \R\rightarrow\mathbb{C}$ and $\x=\xi_{1}+\xi_{2}i\in \R$, we let $\widehat{f}\colon\R\rightarrow\mathbb{C}$ denote the \emph{Fourier transformation} of $f$
  \begin{equation}\nonumber
  \begin{split}
   \widehat{f}(\x)=\mathbb{E}_{\A\in \R}f(\A)e(-\A\circ_{N}\x)
  \end{split}
  \end{equation}
for all $\x\in\R$, where $e(x)=\exp(2\pi ix)$ for all $x\in\mathbb{R}$.
\end{defn}

\begin{defn}[Gowers uniformity norms]
For $d\geq 1$, we define the \emph{$d$-th Gowers uniformity norm} of $f$ on $\R$ inductively by
  \begin{equation}\nonumber
    \begin{split}
      \Vert f\Vert_{U^{1}(\R)}=\Bigl\vert\mathbb{E}_{\A\in \R}f(\A)\Bigr\vert
    \end{split}
  \end{equation}
and
  \begin{equation}\nonumber
    \begin{split}
      \Vert f\Vert_{U^{d+1}(\R)}=\Bigl(\mathbb{E}_{\B\in \R}\Vert f_{\B}\cdot\overline{f}\Vert_{_{U^{d}(\R)}}^{2^{d}}\Bigr)^{1/2^{d+1}}
    \end{split}
  \end{equation}
for $d\geq 1$, where $\overline{f}$ denotes the conjugate of $f$ and $f_{\B}(\A)=f(\B+\A)$ for all $\A\in\R$.
\end{defn}
Gowers ~\cite{G} showed that this defines a norm on functions on $\mathbb{Z}_{N}$ for $d>1$. These norms were later used by Green, Tao, Ziegler and others in studying the primes (see, for example, ~\cite{5},~\cite{65} and ~\cite{6}). Analogous semi-norms were defined in the ergodic setting by Host and Kra ~\cite{HK}. It is worth noting that for each of these uses, there is a corresponding decomposition theorem.

A direct computation shows that for any function $f$ on $\R$, we have
\begin{equation}\label{F}
    \begin{split}
      \Vert f\Vert_{\R}^{4}=\sum_{\xi\in \R}\vert\widehat{f}(\xi)\vert^{4}.
    \end{split}
  \end{equation}

\begin{defn}[Kernel]
A function $\phi\colon \R\rightarrow\mathbb{C}$ is a \emph{kernel} if it is non-negative and $\mathbb{E}_{\A\in\R}\phi(\A)=1$. The set $\{\x\in \R\colon \widehat{\phi}(\x)\neq 0\}$ is called the \emph{spectrum} of $\phi$.
\end{defn}

Our main decomposition result is:

\begin{thm}\label{nU3s}($U^{3}$ decomposition theorem).
  For every positive finite measure $\nu$ on the group $\mathcal{M}_{\Z}$, every function $F\colon\mathbb{N}\times\mathbb{N}\times\mathbb{R}^{+}\rightarrow\mathbb{R}^{+}$, every $\e>0$, and every $N\in\mathbb{N}$ sufficiently large depending only on $F$ and $\e$, there exist positive integers $Q=Q(F,N,\e,\nu)$ and $R=R(F,N,\e,\nu)$ bounded by a constant depending only on $F$ and $\e$ such that for every $\chi\in\mathcal{M}_{\Z}$, the function $\chi_{N}\coloneqq\chi\cdot\bold{1}_{R_{N}}$ can be written as
  \begin{equation}\nonumber
   \begin{split}
     \chi_{N}(\A)=\chi_{N,s}(\A)+\chi_{N,u}(\A)+\chi_{N,e}(\A)
   \end{split}
  \end{equation}
for all $\A\in\R$ such that

  (i) $\chi_{N,s}=\chi_{N}*\phi_{N,1}$ and $\chi_{N,s}+\chi_{N,e}=\chi_{N}*\phi_{N,2}$, where $\phi_{N,1}$ and $\phi_{N,2}$ are kernels of $\R$ that are independent of $\chi$, and the convolution product is defined on $\R$;

  (ii) $\vert\chi_{N,s}(\A+Q)-\chi_{N,s}(\A)\vert, \vert\chi_{N,s}(\A+Qi)-\chi_{N,s}(\A)\vert\leq\frac{R}{N}$ for every $\A\in \R$;

  (iii) $\Vert\chi_{N,u}\Vert_{U^{3}(\R)}\leq\frac{1}{F(Q,R,\e)}$;

  (iv) $\mathbb{E}_{\A\in \R}\int_{\mathcal{M}_{\Z}}\vert\chi_{N,e}(\A)\vert d\nu(\chi)\leq\e$.
\end{thm}

\section{Partition regularity on $\Z$}
In this section, we explain how the $U^{3}$-decomposition result of Theorem \ref{nU3s} can be applied to deduce the partition regularity property of Theorem \ref{c1}. We show the following partition regularity result:
\begin{thm}\label{c1}(Partition regularity theorem for a special class of equations on $\Z$).
  Let $p$ be a quadratic equation of the form
    \begin{equation}\nonumber
    \begin{split}
      p(x,y,n)=ax^{2}+by^{2}+cn^{2}+dxy+exn+fyn
    \end{split}
  \end{equation}
  for some $a,b,c,d,e,f\in\Z$. If all of $\sqrt{e^{2}-4ac}, \sqrt{f^{2}-4bc},\sqrt{(e+f)^{2}-4c(a+b+d)}$ belong to $\Z$, then
  $p(x,y,n)$ is partition regular in $\Z$.
\end{thm}

It is worth noting that the same result holds with $\Z$ replaced by $\Z\backslash \{0\}$ by using a similar argument.
Since the equation $x^{2}-y^{2}-n^{2}$ satisfies the hypothesis of this theorem by setting $a=1,b=c=-1,d=e=f=0$, we obtain Corollary \ref{c11}.

\begin{defn}[Admissibility]
 A 4-tuple of Gaussian integers $(\C_{1},\C_{2},\C_{3},\C_{4})\in\Z^{4}$ is \emph{admissible} if $\C_{1}\neq \C_{2}$, $\C_{3}\neq \C_{4}$ and $\{\C_{1},\C_{2}\}\neq \{\C_{3},\C_{4}\}$.
\end{defn}
We begin with the multiple recurrence property for multiplicative functions:

\begin{prop}\label{c4}(Multiple recurrence property for multiplicative functions).
  Let $(\C_{1},\C_{2},\C_{3},\C_{4})\in\Z^{4}$ be an admissible 4-tuple. Let $\nu$ be a positive finite measure on the group $\mathcal{M}_{\Z}$ such that $\nu(\{\bold{1}\})>0$ and
       $$\int_{\mathcal{M}_{\Z}}\chi(x)\overline{\chi}(y)d\nu(\chi)\geq 0$$ 
for all $x,y\in \Z$. 
 Then there exist $\A,\B\in\Z$ such that
  $(\A+\C_{1}\B)(\A+\C_{2}\B)$ and $(\A+\C_{3}\B)(\A+\C_{4}\B)$ are distinct and nonzero, and
    \begin{equation}\nonumber
  \begin{split}
    \int_{\mathcal{M}_{\Z}}{\chi(\A+\C_{1}\B)\chi(\A+\C_{2}\B)\overline{\chi}(\A+\C_{3}\B)
    \overline{\chi}(\A+\C_{4}\B)}d\nu(\chi)>0.
  \end{split}
  \end{equation}
  Moreover, if $\C_{1},\C_{2},\C_{3},\C_{4}$ are pairwise distinct, we have that
  \begin{equation}\label{temp9}
  \begin{split}
    \liminf_{N\rightarrow\infty}\mathbb{E}_{(\A,\B)\in\Theta_{N}}\int_{\mathcal{M}_{\Z}}{\chi(\A+\C_{1}\B)\chi(\A+\C_{2}\B)\overline{\chi}(\A+\C_{3}\B)
    \overline{\chi}(\A+\C_{4}\B)}d\nu(\chi)>0,
  \end{split}
  \end{equation}
  where $\Theta_{N}=\{(\A,\B)\in R_{N}\times  R_{N}\colon \A+\C_{i}\B\in R_{N}, i=1,2,3,4\}$.
\end{prop}

 The proof is similar to Proposition 10.3 and Proposition 10.4 in ~\cite{FH}. We omit the proof but stress the differences. First of all, we need to use Theorem \ref{nU3s} to decompose $\chi_{N}$ into the sum $\chi_{N,s}+\chi_{N,u}+\chi_{N,e}$. Then we can expand the left hand side of (\ref{temp9}) into 81 terms. By a similar argument as Proposition 10.5 in ~\cite{FH}, we have that
   \begin{equation}\nonumber
  \begin{split}
    \liminf_{N\rightarrow\infty}\mathbb{E}_{(\A,\B)\in\Theta_{N}}\int_{\mathcal{M}_{\Z}}{\chi_{N,s}(\A+\C_{1}\B)\chi_{N,s}(\A+\C_{2}\B)\overline{\chi_{N,s}}(\A+\C_{3}\B)
    \overline{\chi_{N,s}}(\A+\C_{4}\B)}d\nu(\chi)
  \end{split}
  \end{equation}
is bounded below by a positive number which is independent of $\e$. So
 it suffices to show that all other terms are negligible. A term is obvious $O(\e)$ if it contents the expression $\chi_{N,e}$. It then suffices to show that all terms containing the expression $\chi_{N,u}$ are negligible, which holds immediately if one can show that
      \begin{equation}\label{temp8}
    \begin{split}
      \Bigl\vert\mathbb{E}_{\A,\B\in \R}\bold{1}_{R_{N}}(\B)\prod_{j=0}^{3}a_{j}(\A+\C_{j}\B)\Bigr\vert
      \leq C\min_{1\leq j\leq 4}(\Vert a_{j}\Vert_{U^{3}(\R)})^{1/3}+\frac{10}{\tilde{N}}
    \end{split}
  \end{equation}
for all functions $a_{1},\dots,a_{4}$ on $\R$ with $\Vert a_{j}\Vert_{L^{\infty}(\R)}\leq 1$ for $j=1,\dots,4$, where $C>0$ depends only on $\C_{1},\dots,\C_{4}$ (the exponent on the right hand side of (\ref{temp8}) is $\frac{1}{2}$ for the 1-dimensional case and is $\frac{1}{3}$ for the 2-dimensional case). The proof of it is a straightforward generalization of Lemma 10.7 in ~\cite{FH} and Theorem 2.1 in ~\cite{FH} in the 2-dimensional case, and so we are done. It is worth noting that this is the only proposition in which we need to use the fact that 
the map $\B\to\C\B$ is a bijection from $\R$ to itself if $\vert\C\vert^{2}<\tilde{N}$.

In order to transform Theorem \ref{c1} to a density regularity property for dilation invariant densities,
we review some definitions adapted to $\Z$:

\begin{defn}[Multiplicative F${\o}$lner sequence]
  The sequence $\{\Phi_{N}\}_{N\in\mathbb{N}}$ of finite subsets of $\Z$ is a \emph{multiplicative F${\o}$lner sequence} if for every $\A\in\Z$,
  \begin{equation}\nonumber
  \begin{split}
    \lim_{N\rightarrow\infty}\frac{\vert \A\Phi_{N}\triangle\Phi_{N}\vert}{\vert \Phi_{N}\vert}=0,
  \end{split}
  \end{equation}
  where $\A\Phi_{N}=\{\A x\colon x\in\Phi_{N}\}$.
\end{defn}

\begin{defn}[Multiplicative density]
  The \emph{multiplicative density} $d_{mult}(E)$ of a subset $E$ of $\Z$ (with respect to the multiplicative F${\o}$lner sequence $\{\Phi_{N}\}_{N\in\mathbb{N}}$) is defined to be
  \begin{equation}\nonumber
  \begin{split}
    d_{mult}(E)=\limsup_{N\rightarrow\infty}\frac{\vert E\cap\Phi_{N}\vert}{\vert \Phi_{N}\vert}.
  \end{split}
  \end{equation}
\end{defn}

\begin{defn}[Action by dilation]
  An \emph{action by dilation} on a probability space $(X,\mathcal{B},\mu)$ is a family $\{T_{\A}\}_{\A\in\Z}$ of invertible measure preserving transformations of $(X,\mathcal{B},\mu)$ that satisfy $T_{1}=id$ and $T_{\A}\cdot T_{\B}=T_{\A\B}$ for all $\A,\B\in\Z$. Note that this can be extended to a measure preserving action $\{T_{z}\}_{z\in\mathbb{Q}[i]}$ by defining $T_{\A/\B}=T_{\A}T_{\B}^{-1}$ for all $\A,\B\in\Z, \B\neq 0$, where $\mathbb{Q}[i]$ is the set of Gaussian rational numbers.
\end{defn}

Since $\Z$ with multiplication is a discrete amenable semi-group, we make use of (see, for example,  Theorem 2.1 in ~\cite{BM2} and Theorem 6.4.17 in ~\cite{B3}):

\begin{thm}\label{FC}(Furstenberg correspondence principle).
  Let $E$ be a subset of $\mathbb{Z}[i]$. Then there exist an action by dilation $\{T_{\A}\}_{\A\in\Z}$ on a probability space $(X,\mathcal{B},\mu)$ and a set $A\in\mathcal{B}$ with $\mu(A)=d_{mult}(E)$ such that for all $k\in\mathbb{N}$ and for all $\A_{1},\dots,\A_{k}\in\Z$, we have
    \begin{equation}\nonumber
  \begin{split}
    d_{mult}(\A_{1}^{-1}E\cap\dots\cap \A_{k}^{-1}E)\geq\mu(T_{\A_{1}}^{-1}A\cap\dots\cap T_{\A_{k}}^{-1}A),
  \end{split}
  \end{equation}
  where $\A^{-1}E=\{x\in\Z\colon \A x\in E\}$.
\end{thm}

For every $f\in L^{2}(\mu)$, by the spectral theorem, there exists a positive finite measure $\nu$ (called the \emph{spectral measure  of $f$}) on the group of
multiplicative functions $\mathcal{M}_{\Z}$ such that for all $\A,\B\in\Z$,
\begin{equation}\nonumber
  \begin{split}
    \int_{X}{T_{\A}f\cdot T_{\B}\overline{f}}d\mu=\int_{X}{T_{\A/\B}f\cdot \overline{f}}d\mu=\int_{\mathcal{M}_{\Z}}{\chi(\A/\B)}d\nu(\chi)
    =\int_{\mathcal{M}_{\Z}}{\chi(\A)\overline{\chi}(\B)}d\nu(\chi).
  \end{split}
  \end{equation}
The following lemma can be deduced by the same argument on page 63 of ~\cite{FH}:
\begin{lem}\label{37}
For every measurable set $A$ with positive measure, the spectral measure $\nu$ of the function $\bold{1}_{A}$ satisfies the condition described in Proposition \ref{c4}, i.e. 
$\nu(\{\bold{1}\})>0$ and $$\int_{\mathcal{M}_{\Z}}\chi(x)\overline{\chi}(y)d\nu(\chi)\geq 0$$ 
for all $x,y\in \Z$.
\end{lem}
 
  Therefore, we can deduce the following corollary from Proposition \ref{c4}:
\begin{cor}\label{c3}
  Let $(\C_{1},\C_{2},\C_{3},\C_{4})\in\Z^{4}$ be an admissible 4-tuple. Let $\{T_{\A}\}_{\A\in\Z}$ be an action by dilation on a probability space $(X,\mathcal{B},\mu)$. Then for every $A\in\mathcal{B}$ with $\mu(A)>0$, there exist $\A,\B\in\Z$ such that $(\A+\C_{1}\B)(\A+\C_{2}\B)$ and $(\A+\C_{3}\B)(\A+\C_{4}\B)$ are distinct and non-zero, and
    \begin{equation}\nonumber
  \begin{split}
    \mu(T_{(\A+\C_{1}\B)(\A+\C_{2}\B)}^{-1}A\cap T_{(\A+\C_{3}\B)(\A+\C_{4}\B)}^{-1}A)>0.
  \end{split}
  \end{equation}
  Moreover, if $\C_{1},\C_{2},\C_{3},\C_{4}$ are pairwise distinct, we have that
  \begin{equation}\nonumber
  \begin{split}
    \liminf_{N\rightarrow\infty}\mathbb{E}_{(\A,\B)\in\Theta_{N}}\mu(T_{(\A+\C_{1}\B)(\A+\C_{2}\B)}^{-1}A\cap T_{(\A+\C_{3}\B)(\A+\C_{4}\B)}^{-1}A)>0,
  \end{split}
  \end{equation}
  where $\Theta_{N}=\{(\A,\B)\in R_{N}\times  R_{N}\colon \A+\C_{i}\B\in R_{N}, i=1,2,3\}$.
\end{cor}

\begin{cor}\label{c2}
  Let $\C_{0}\in\Z$ with $\C_{0}\neq 0$ and let $(\C_{1},\C_{2},\C_{3},\C_{4})\in\Z^{4}$ be an admissible 4-tuple. Then for any partition of $\Z$ into finitely many cells, there exist $\A,\B,\C\in\Z$ such that $\C\C_{0}(\A+\C_{1}\B)(\A+\C_{2}\B)$ and $\C\C_{0}(\A+\C_{3}\B)(\A+\C_{4}\B)$ are distinct and nonzero, and they belong to the same cell.
\end{cor}
\begin{proof} For any partition of $\Z$ into finitely many cells, one of the cells $E$ has positive multiplicative density. Let $A$ be the set in Theorem \ref{FC} corresponding to $E$.
By Corollary \ref{c3},
there exist $\A,\B\in\Z$ such that $(\A+\C_{1}\B)(\A+\C_{2}\B)$ and $(\A+\C_{3}\B)(\A+\C_{4}\B)$ are distinct and nonzero, and
\begin{equation}\nonumber
  \begin{split}
    \mu(T_{\C_{0}(\A+\C_{1}\B)(\A+\C_{2}\B)}^{-1}A\cap T_{\C_{0}(\A+\C_{3}\B)(\A+\C_{4}\B)}^{-1}A)=\mu(T_{(\A+\C_{1}\B)(\A+\C_{2}\B)}^{-1}A\cap T_{(\A+\C_{3}\B)(\A+\C_{4}\B)}^{-1}A)>0.
  \end{split}
  \end{equation}
  So $d_{mult}(\C_{0}^{-1}(\A+\C_{1}\B)^{-1}(\A+\C_{2}\B)^{-1}E\cap\C_{0}^{-1}(\A+\C_{3}\B)^{-1}(\A+\C_{4}\B)^{-1}E)>0$. Thus there exists $\C\in\Z$ such that
  $\C\C_{0}(\A+\C_{1}\B)(\A+\C_{2}\B)$ and $\C\C_{0}(\A+\C_{3}\B)(\A+\C_{4}\B)$ are distinct and nonzero, and they both belong to $E$.
\end{proof}

Thus the proof of Theorem \ref{c1} has now been reduced to the following lemma. Since the proof of it is identical to the one in Appendix C in ~\cite{FH}, we omit it:

\begin{lem}\label{ge}(The general solution for a special class of equations).
  If $p$ is a quadratic equation satisfying the condition in Theorem \ref{c1}, then there exists an admissible 4-tuple $(\C_{1},\C_{2},\C_{3},\C_{4})\in\Z^{4}$ such that for all $\C',\A,\B\in\Z$, the elements $x=\C'\C_{0}(\A+\C_{1}\B)(\A+\C_{2}\B)$ and $y=\C'\C_{0}(\A+\C_{3}\B)(\A+\C_{4}\B)$
  satisfy $p(x,y,n)=0$ for some $n\in\Z$, where $\C_{0}\in\Z,\C_{0}\neq 0$.
\end{lem}

\section{Katai's Lemma}
The rest of the paper is devoted to the proof of Theorem \ref{nU3s}. In this section, we prove the key number theoretic input
that we need in later sections.

Denote $s(x)=\Bigl\vert \{\A\in\Z\colon \mathcal{N}(\A)\leq x\bigr\}\Bigr\vert$, where $\mathcal{N}(a+bi)=\sqrt{a^{2}+b^{2}}$ is the norm of $a+bi$. Let $\P$ be the set of primes in $\Z$.

\begin{defn}
For any finite subset $\mathcal{P}=\{p_{1},\dots,p_{k}\}$ of $\P$, denote $A_{\mathcal{P}}=\sum_{p\in \mathcal{P}}\frac{1}{\mathcal{N}(p)^{2}}$. For all $\A\in\Z$, write $\omega_{\mathcal{P}}(\A)=\sum_{p|\A,p\in\mathcal{P}}1$.
\end{defn}

It is worth noting that $\omega_{\mathcal{P}}$ is an additive function, meaning $\omega_{\mathcal{P}}(\A\B)=\omega_{\mathcal{P}}(\A)+\omega_{\mathcal{P}}(\B)$ for any $\A,\B\in\Z$ coprime to each other.
We need the following Turan-Kubilius Lemma for $\Z$ (see, for example, Lemma 9.3 in ~\cite{FH} for the proof):

\begin{lem}\label{TK} Let $x\in\mathbb{N}$ be sufficiently large with respect to $\mathcal{P}$. Then
\begin{equation}\label{TK2}
  \begin{split}
    \frac{1}{s(x)}\sum_{\mathcal{N}(\A)\leq x}\vert\omega_{\mathcal{P}}(\A)-A_{\mathcal{P}}\vert\leq C A_{\mathcal{P}}^{1/2}
  \end{split}
\end{equation}
for some universal constant $C$.
\end{lem}

If $x\in\mathbb{N}$ is sufficiently large with respect to $\mathcal{P}$. Let $z=\sqrt{2}x$. By Lemma \ref{TK},
  \begin{equation}\nonumber
  \begin{split}
    \qquad\frac{1}{x^{2}}\sum_{\A\in R_{x}}\vert\omega_{\mathcal{P}}(\A)-A_{\mathcal{P}}\vert
    \leq \frac{s(z)}{x^{2}}\cdot\frac{1}{s(z)}\sum_{\mathcal{N}(\A)\leq z}\vert\omega_{\mathcal{P}}(\A)-A_{\mathcal{P}}\vert
    \leq \frac{s(z)}{x^{2}}C A_{\mathcal{P}}^{1/2}\leq 8C A_{\mathcal{P}}^{1/2}.
  \end{split}
\end{equation}
Thus we have:

\begin{cor}\label{TK3} Let $x\in\mathbb{N}$ be sufficiently large with respect to $\mathcal{P}$. Then
\begin{equation}\nonumber
  \begin{split}
    \frac{1}{x^{2}}\sum_{\A\in R_{x}}\vert\omega_{\mathcal{P}}(\A)-A_{\mathcal{P}}\vert\leq 8 C A_{\mathcal{P}}^{1/2}
  \end{split}
\end{equation}
for some universal constant $C$.
\end{cor}

The following is the classification of Gaussian primes, see, for example, ~\cite{GP} for the proof:
\begin{thm}\label{cls}
  The prime elements $p$ of $\Z$, up to a multiple of unit elements, is of one of the following three forms:

  (i) $p=1+i$;

  (ii) $p=a+bi, a^{2}+b^{2}=p_{0}, p_{0}\equiv 1 \mod 4$;

  (iii) $p=p_{0}, p_{0}\equiv 3 \mod 4$.

  Here $p_{0}$ is a prime in $\mathbb{Z}$.
\end{thm}
\begin{prop}\label{ne}
  Let $p=p_{1}+p_{2}i, q=q_{1}+q_{2}i\in\P$. If $\mathcal{N}(p)=\mathcal{N}(q)$, then either (i) $p$ and $q$ conjugate with each other; or (ii) $\vert p_{1}\vert=\vert q_{1}\vert, \vert p_{2}\vert=\vert q_{2}\vert$. In particular, for any $p\in\P$, the cardinality of the set $\mathcal{D}_{p}\coloneqq\{q\in\P\colon \mathcal{N}(p)=\mathcal{N}(q)\}$ is at most 4.
\end{prop}
\begin{proof}
  Since the norm of the three types of primes in Theorem \ref{cls} are different, $\mathcal{N}(p)=\mathcal{N}(q)$ implies that $p$ and $q$ are of the same type.
  If they are of type (i) or (iii), then we are done, and $\vert\mathcal{D}_{p}\vert\leq 4$. If they are of type (ii), then $p_{1}^{2}+p_{2}^{2}=q_{1}^{2}+q_{2}^{2}=p_{0}$ for some prime $p_{0}\equiv 1 \mod 4$. Let $A=p_{1}+q_{1}, B=p_{1}-q_{1}, C=p_{2}+q_{2}, D=p_{2}-q_{2}$, then $AB=CD$. If one of $A,B,C,D$ is 0, then we are done. If $ABCD\neq 0$, let $a>0$ be the greatest common divisor of $A$ and $C$ and suppose $A=ab, C=ac$ for some $b,c\in\mathbb{Z}$. Thus $bB=cD$ and $b|D, c|B$. Assume that $B=d'c, D=bd$. Then $AB=CD$ implies $d=d'$. So $p_{0}=p_{1}^{2}+p_{2}^{2}=(\frac{A+B}{2})^{2}+(\frac{C-D}{2})^{2}=(\frac{ab+cd}{2})^{2}+(\frac{ac-bd}{2})^{2}=\frac{(a^{2}+d^{2})(b^{2}+c^{2})}{4}$. Since $p_{0}$ is a prime, one of $a^{2}+d^{2}, b^{2}+c^{2}$ equals to either 1,2, or 4.
  Since $ABCD\neq 0$, one of $a^{2}+d^{2}, b^{2}+d^{2}$ must equal to 2. Suppose $a^{2}+d^{2}=2$. Then $a,d=\pm 1$. So $A=\pm D, B=\pm C$. In any case, we have $\vert p_{1}\vert=\vert q_{1}\vert, \vert p_{2}\vert=\vert q_{2}\vert$, and $\vert\mathcal{D}_{p}\vert\leq 4$.
\end{proof}

The following theorem is a variation of Lemma 9.4 in ~\cite{FH} which is tailored for our purpose. We include the proof for completion:

\begin{lem}\label{katai}(Katai's Lemma).
  For every $\e>0$ and $K_{0}\in\mathbb{N}$, there exist $\d=\d(\e,K_{0})>0$ and $K=K(\e,K_{0})>K_{0}$ such that the following holds: If $N$ is sufficiently large with respect to $K$ and $f\colon R_{N}\rightarrow\mathbb{C}$ is a function with $\vert f\vert\leq 1$, and
    \begin{equation}\label{kt4}
  \begin{split}
    \max_{p,q\in\P, K_{0}<\mathcal{N}(p)<\mathcal{N}(q)<K}\frac{1}{\vert R_{N}\vert}\Bigl\vert\sum_{\A\in R_{N}/p\cap R_{N}/q}f(p\A)\overline{f}(q\A)\Bigr\vert<\d,
  \end{split}
\end{equation}
then
    \begin{equation}\nonumber
  \begin{split}
    \sup_{\chi\in\mathcal{M}_{\Z}}\Bigl\vert\mathbb{E}_{\A\in R_{N}}\chi(\A)f(\A)\Bigr\vert<\e,
  \end{split}
\end{equation}
where $R_{N}/p=\{\A\in\Z\colon p\A\in R_{N}\}$.
\end{lem}

\begin{proof}
 Let $\mathcal{P}=\{p\in\P\colon K_{0}<\mathcal{N}(p)<K\}$. Fix $\chi\in\mathcal{M}_{\Z}$. Let
\begin{equation}\nonumber
  \begin{split}
    S(N)=\sum_{\A\in R_{N}}\chi(\A)f(\A),
  \end{split}
\end{equation}

\begin{equation}\nonumber
  \begin{split}
    H(N)=\sum_{\A\in R_{N}}\chi(\A)f(\A)\omega_{\mathcal{P}}(\A).
  \end{split}
\end{equation}
 By Corollary \ref{TK3},
\begin{equation}\label{kt1}
  \begin{split}
    \vert H(N)-A_{\mathcal{P}}S(N)\vert\leq 8CN^{2}A_{\mathcal{P}}^{1/2}
  \end{split}
\end{equation}
for some universal constant $C$.
Notice that
\begin{equation}\nonumber
  \begin{split}
    H(N)=\sum_{p\in\mathcal{P}}\sum_{\B\in R_{N}/p}\chi(p\B)f(p\B)=\sum_{\B\in R_{N}}\chi(\B)\Sigma_{\B},
  \end{split}
\end{equation}
where $\Sigma_{\B}=\sum_{p\in\mathcal{P}\cap R_{N}/\B}\chi(p)f(p\B)$. By the Cauchy-Schwartz Inequality,
\begin{equation}\label{kt2}
  \begin{split}
    \vert H(N)\vert^{2}\leq\Bigl(\sum_{\mathcal{N}(\B)\leq 2N}\vert\chi(\B)\vert^{2}\Bigr)\Bigl(\sum_{\mathcal{N}(\B)\leq 2N}\vert\Sigma_{\B}\vert^{2}\Bigr)\leq 4CN^{2}\sum_{\mathcal{N}(\B)\leq 2N}\vert\Sigma_{\B}\vert^{2}
  \end{split}
\end{equation}
since $R_{N}\subset \{\B\colon N(\B)\leq 2N\}$.
It is easy to deduce that $\sum_{p\in\mathcal{P}}\vert R_{N}/p\vert\leq C'RN^{2}A_{\mathcal{P}}$ for some universal constant $C'>0$. Thus

\begin{equation}\label{kt3}
  \begin{split}
    &\qquad\sum_{\mathcal{N}(\B)\leq 2N}\vert\Sigma_{\B}\vert^{2}=\sum_{\mathcal{N}(\B)\leq 2N}\sum_{p,q\in\mathcal{P}\cap R_{N}/\B}\chi(p)\overline{\chi}(q)f(p\B)\overline{f}(q\B)
    \\&\leq C'RN^{2}A_{\mathcal{P}}+\sum_{p,q\in\mathcal{P},\mathcal{N}(p)\neq \mathcal{N}(q)}\Bigl\vert \sum_{\B\in R_{N}/p\cap R_{N}/q}f(p\B)\overline{f}(q\B)\Bigr\vert,
  \end{split}
\end{equation}
where we used Proposition \ref{ne}. Combining (\ref{kt1}),(\ref{kt2}) and (\ref{kt3}), we obtain

\begin{equation}\nonumber
  \begin{split}
    \frac{A_{\mathcal{P}}^{2}\vert S(N)\vert^{2}}{N^{4}}\leq C''\Bigl(1+A_{\mathcal{P}}+\sum_{p,q\in\mathcal{P},\mathcal{N}(p)\neq \mathcal{N}(q)}\frac{1}{N^{2}}\Bigl\vert \sum_{\B\in R_{N}/p\cap R_{N}/q}f(p\B)\overline{f}(q\B)\Bigr\vert\Bigr)
  \end{split}
\end{equation}
for some universal constant $C''>0$. Since $\sum_{p\in\P}\frac{1}{\mathcal{N}(p)^{2}}=\infty$, the lemma follows by picking $K$ first and then $\d$ appropriately.
\end{proof}

\section{$U^{2}$ non-uniformity}\label{u2}

Before proving the $U^{3}$ decomposition result, we start with a $U^{2}$-decomposition result.
We prove the following theorem in this section:

\begin{thm}\label{nU2}($U^{2}$ decomposition theorem).
  For every $\e>0$, there exist positive integers $Q=Q(\e),R=R(\e),N_{0}=N_{0}(\e)$ such that for $N\geq N_{0}$, there exists a kernel $\phi_{N,\e}$ on $\R$ with the following property:
  For every $\chi\in\mathcal{M}_{\Z}$, writing $\chi_{N,s}=\chi_{N}*\phi_{N,\e}$ and $\chi_{N,u}=\chi_{N}-\chi_{N,s}$, we have

  (i)$\vert\chi_{N,s}(\A+Q)-\chi_{N,s}(\A)\vert, \vert\chi_{N,s}(\A+Qi)-\chi_{N,s}(\A)\vert\leq\frac{R}{N}$ for every $\A\in \R$;

  (ii)$\Vert\chi_{N,u}\Vert_{U^{2}(\R)}\leq\e$.

  Moreover, for every $N\in\mathbb{N}, \x\in \R$ and $0<\e'\leq\e$, we have $\widehat{\phi_{N,\e'}}(\x)\geq\widehat{\phi_{N,\e}}(\x)\geq 0$.
\end{thm}

For any $x\in\mathbb{R}$, denote $\Vert x\Vert_{\mathbb{R}/\mathbb{Z}}=\min_{n\in\mathbb{Z}}\vert x-n\vert$. When there is no ambiguity, we write $\Vert x\Vert=\Vert x\Vert_{\mathbb{R}/\mathbb{Z}}$ for short.

We first explain what happens when the Fourier coefficient of $\chi$ is away from 0:

\begin{cor}\label{nF2}
  For every $\e$, there exist $Q=Q(\e),V=V(\e), N_{0}=N_{0}(\e)\in\mathbb{N}$ such that for every $N\geq N_{0}$, every $\chi\in\mathcal{M}_{\Z}$ and every $\x=\xi_{1}+\xi_{2}i\in \R$, if $\vert\widehat{\chi_{N}}(\xi)\vert\geq\e$, then $\Vert Q\xi_{i}/\tilde{N}\Vert\leq\frac{QV}{N}$ for $i=1,2$.
\end{cor}

\begin{proof}
  Let $\delta=\delta(\e,1)$ and $K=K(\e,1)$ be given by Lemma \ref{katai}. Let
    \begin{equation}\nonumber
  \begin{split}
   Q=\prod_{p,q\in\P,\mathcal{N}(p)<\mathcal{N}(q)<K}\mathcal{N}(p-q)^{2}\in\mathbb{N}.
  \end{split}
\end{equation}
Suppose that $N>K$. Let $q,p\in\P$ with $\mathcal{N}(p)<\mathcal{N}(q)<K$ and let $\xi\in \R$. Denote $p=p_{1}+p_{2}i,q=q_{1}+q_{2}i$. Let $\v=(p_{1}-q_{1},p_{2}-q_{2}), \w=(q_{2}-p_{2},p_{1}-q_{1})$. Let $f\colon\R\rightarrow\mathbb{C}, f(\A)=e(-\A\circ_{N}\x)$.

  For $i\in\mathbb{Z}$, denote $U_{j}=\{\A\in R_{N}/p\cap R_{N}/q\colon Im(\A)=j\}$. Then $ R_{N}/p\cap R_{N}/q=\bigcup_{j=-2N}^{2N}U_{j}$. By the convexity of $ R_{N}/p\cap R_{N}/q$, each $U_{j}$ can be written as $U_{j}=\{x+ji\colon a_{j}\leq x\leq b_{j}\}$ for some $a_{j}, b_{j}\in\mathbb{Z}$. We have
  \begin{equation}\nonumber
  \begin{split}
   &\qquad\frac{1}{\tilde{N}^{2}}\Bigl\vert\sum_{\A\in R_{N}/p\cap R_{N}/q}f(p\A)\overline{f}(q\A)\Bigr\vert
   \leq\frac{1}{{\tilde{N}^{2}}}\sum_{j=-2N}^{2N}\Bigl\vert\sum_{\A\in U_{j}}f(p\A)\overline{f}(q\A)\Bigr\vert
   \\&=\frac{1}{{\tilde{N}^{2}}}\sum_{j=-2N}^{2N}\Bigl\vert\sum_{x=a_{j}}^{b_{j}}e(-x\v\circ_{N}\x)e(-j\w\circ_{N}\x) \Bigr\vert
   \leq\frac{1}{{\tilde{N}^{2}}}\sum_{j=-2N}^{2N}
   \frac{2}{\Vert \v\circ_{N}\x\Vert}\leq\frac{C}{N\Vert\v\circ_{N}\x\Vert}
  \end{split}
\end{equation}
for some universal constant $C$.
On the other hand, writing $V_{j}=\{\A\in R_{N}/p\cap R_{N}/q\colon Re(\A)=i\}$, then a similar argument shows that 
$$\frac{1}{\tilde{N}^{2}}\Bigl\vert\sum_{\A\in R_{N}/p\cap R_{N}/q}f(p\A)\overline{f}(q\A)\Bigr\vert\leq\frac{C}{N\Vert\w\circ_{N}\x\Vert}.$$

If $\vert\widehat{\chi_{N}}(\xi)\vert\geq\e$, by Lemma \ref{katai}, there exist $p,q\in\P, \mathcal{N}(p)<\mathcal{N}(q)<K$ such that $\frac{1}{\tilde{N}^{2}}\Bigl\vert\sum_{\A\in R_{N}/p\cap R_{N}/q}f(p\A)\overline{f}(q\A)\Bigr\vert\geq \d$. Thus $\max\{\Vert\v\circ_{N}\x\Vert,\Vert\w\circ_{N}\x\Vert\}\leq \frac{C'}{N\d}$
for some universal constant $C'$.
Suppose
  \begin{equation}\nonumber
  \begin{split}
   &\v\circ_{N}\x=\frac{1}{\tilde{N}}((p_{1}-q_{1})\xi_{1}+(p_{2}-q_{2})\xi_{2})=a_{1}\pm\Vert\v\circ_{N}\x\Vert;
   \\&\w\circ_{N}\x=\frac{1}{\tilde{N}}((q_{2}-p_{2})\xi_{1}+(p_{1}-q_{1})\xi_{2})=a_{2}\pm\Vert\w\circ_{N}\x\Vert
  \end{split}
\end{equation}
for some $a_{1},a_{2}\in\mathbb{Z}$. Then
$\xi_{1}/\tilde{N}=(n_{1}+\d_{1})/\Delta, \xi_{2}/\tilde{N}=(n_{2}+\d_{2})/\Delta$,
where $\Delta=(p_{1}-q_{1})^{2}+(p_{2}-q_{2})^{2}\neq 0, n_{1},n_{2}\in\mathbb{Z}, \vert\d_{1}\vert, \vert\d_{2}\vert\leq\frac{KC''}{N\d}$ for some universal constant $C''$. Since
$\Delta$ divides $Q$, we get $\Vert Q\xi_{i}/\tilde{N}\Vert=\Vert Q\d_{i}/\tilde{N}\Vert\leq\frac{KC''}{N\d}$ for $i=1,2$. The corollary follows by taking $V=\lceil\frac{KC''}{\d}\rceil$ (recall that $\lceil x\rceil$ is the smallest integer which is not smaller than $x$).
\end{proof}

For every $\e>0$, define

  \begin{equation}\nonumber
  \begin{split}
   &\mathcal{A}(N,\e)=\Bigl\{\x\in \R\colon \sup_{\chi\in\mathcal{M}_{\Z}}\vert\widehat{\chi_{N}}(\x)\vert\geq\e^{2}\Bigr\};
   \\&W(N,q,\e)=\max_{\x=\xi_{1}+\xi_{2}i\in\mathcal{A}(N,\e)}\max_{i=1,2}N\Vert q\xi_{i}/\tilde{N}\Vert;
   \\&Q(\e)=\min_{k\in\mathbb{N}}\Bigl\{k!\colon\limsup_{N\rightarrow\infty}W(N,k!,\e)<\infty\Bigr\};
   \\&V(\e)=1+\Bigl\lfloor \frac{1}{Q(\e)}\limsup_{N\rightarrow\infty}W(N,Q(\e),\e)\Bigr\rfloor.
  \end{split}
\end{equation}

It follows from Corollary \ref{nF2} that $Q(\e)$ is well defined. Notice that for all $0<\e'\leq\e, Q(\e')\geq Q(\e)$ and $Q(\e')$ is a multiple of $Q(\e)$. Thus $V(\e')\geq V(\e)$. We fix the above choice of $Q(\e),V(\e)$ for the remaining of this section.

For every $m\geq 1, N>2m$, we define the function $f_{N,m}\colon \R\rightarrow\mathbb{C}$ by

  \begin{equation}\nonumber
  \begin{split}
   f_{N,m}(\A)=\sum_{-m\leq\xi_{1}\leq m}\sum_{-m\leq\xi_{2}\leq m}(1-\frac{\vert\xi_{1}\vert}{m})(1-\frac{\vert\xi_{2}\vert}{m})e(\A\circ_{N}(\x_{1}+\x_{2}i)).
  \end{split}
\end{equation}

Then it is easy to verify that $f_{N,m}$ is a kernel of $\R$. Let $Q_{\tilde{N}}(\e)^{*}$ be the unique integer in $\{1,\dots,\tilde{N}-1\}$ such that $Q(\e)Q_{\tilde{N}}(\e)^{*}\equiv 1\mod \tilde{N}$. For $N>4Q(\e)V(\e)\lceil\e^{-4}\rceil$, we define

  \begin{equation}\label{phi}
  \begin{split}
   \phi_{N,\e}(x+yi)=f_{N,2Q(\e)V(\e)\lceil\e^{-4}\rceil}(Q_{\tilde{N}}(\e)^{*}(x+yi)).
  \end{split}
\end{equation}

In other words, $f_{N,2Q(\e)V(\e)\lceil\e^{-4}\rceil}(x+yi)=\phi_{N,\e}(Q(\e)(x+yi))$. Then $\phi_{N,\e}$ is also a kernel of $\R$, and the spectrum of $\phi_{N,\e}$ is the set

  \begin{equation}\nonumber
  \begin{split}
   \Xi_{N,\e}=\Bigl\{\x=\xi_{1}+\xi_{2}i\in \R\colon\Bigl\Vert\frac{Q(\e)\xi_{i}}{\tilde{N}}\Bigr\Vert<\frac{2Q(\e)V(\e)\lceil\e^{-4}\rceil}{\tilde{N}}, i=1,2\Bigr\},
  \end{split}
\end{equation}
and we have $\widehat{\phi_{N,\e}}(\x)=\Bigl(1-\Bigl\Vert\frac{Q(\e)\xi_{1}}{\tilde{N}}\Bigr\Vert\frac{\tilde{N}}{2Q(\e)V(\e)\lceil\e^{-4}\rceil}\Bigr)
            \Bigl(1-\Bigl\Vert\frac{Q(\e)\xi_{2}}{\tilde{N}}\Bigr\Vert\frac{\tilde{N}}{2Q(\e)V(\e)\lceil\e^{-4}\rceil}\Bigr)$ if $\x\in\Xi_{N,\e}$ and $\widehat{\phi_{N,\e}}(\x)=0$ otherwise.

\begin{proof}[Proof of Theorem \ref{nU2}]
  We show that $\phi_{N,\e}$ defined in (\ref{phi}) satisfies all the requirements. Suppose $0<\e'\leq\e$. Since $Q(\e')\geq Q(\e), V(\e')\geq V(\e)$ and $Q(\e')$ is a multiple of $Q(\e)$, we have $\Xi_{N,\e}\subset\Xi_{N,\e'}$ and $\widehat{\phi_{N,\e'}}(\x)\geq\widehat{\phi_{N,\e}}(\x)$ for every $\x\in \R$.

  For every $\chi\in\mathcal{M}_{\Z},\x=\xi_{1}+\xi_{2}i\in \R$, if $\vert\widehat{\chi_{N}}(\x)\vert\geq\e^{2}$, then by the definition of $Q, \Vert\frac{Q\xi_{i}}{\tilde{N}}\Vert\leq QV/\tilde{N}, i=1,2$. Then $\widehat{\phi_{N,\e}}(\x)\geq(1-\e^{4}/2)^{2}\geq 1-\e^{4}$. So
  $\vert\widehat{\chi_{N}}(\x)-\widehat{\phi_{N,\e}*\chi_{N}}(\x)\vert\leq\e^{4}\leq\e^{2}$. The same estimate also holds if $\vert\widehat{\chi_{N}}(\x)\vert\leq\e^{2}$. Thus by identity (\ref{F}), we have
    \begin{equation}\nonumber
  \begin{split}
  &\qquad \Vert\chi_{N,u}\Vert_{U^{2}(\R)}^{4}=\sum_{\x\in \R}\vert\widehat{\chi_{N}}(\x)-\widehat{\phi_{N,\e}*\chi_{N}}(\x)\vert^{4}\leq\e^{4}\sum_{\x\in \R}\vert\widehat{\chi_{N}}(\x)-\widehat{\phi_{N,\e}*\chi_{N}}(\x)\vert^{2}
  \\&\leq\sum_{\x\in \R}\vert\widehat{\chi_{N}}(\x)\vert^{2}\leq\e^{4},
  \end{split}
\end{equation}
where the last estimate follows from Parseval's identity. This proves (ii).

Lastly, using Fourier inversion formula and the estimate $\vert e(x)-1\vert\leq 2\pi\Vert x\Vert$, we get
    \begin{equation}\nonumber
  \begin{split}
  &\vert\chi_{N,s}(\A+Q)-\chi_{N,s}(\A)\vert\leq\sum_{\x=\xi_{1}+\xi_{2}i\in \R}\vert\widehat{\phi_{N,\e}}(\x)\vert\cdot 2\pi\Bigl\Vert\frac{Q\xi_{1}}{\tilde{N}}\Bigr\Vert
  \leq\vert\Xi_{N,\e}\vert\cdot\frac{4\pi QV\lceil\e^{-4}\rceil}{\tilde{N}};
  \\& \vert\chi_{N,s}(\A+Qi)-\chi_{N,s}(\A)\vert\leq\sum_{\x=\xi_{1}+\xi_{2}i\in \R}\vert\widehat{\phi_{N,\e}}(\x)\vert\cdot 2\pi\Bigl\Vert\frac{Q\xi_{2}}{\tilde{N}}\Bigr\Vert
  \leq\vert\Xi_{N,\e}\vert\cdot\frac{ 4\pi QV\lceil\e^{-4}\rceil}{\tilde{N}}.
  \end{split}
\end{equation}
Since $\vert\Xi_{N,\e}\vert$ depends only on $\e$, the theorem follows by taking $R=\vert\Xi_{N,\e}\vert\cdot 4\pi QV\lceil\e^{-4}\rceil$.
\end{proof}

\section{Inverse and factorization theorems}\label{S4}
In this section, we state and prove some consequences of an inverse theorem by Szegedy ~\cite{Sz} and a factorization theorem by Green and Tao ~\cite{GT} that are particularly tailored for our use. The results in this section generalize Section 4 in ~\cite{FH} to the $\mathbb{Z}^{2}$ case.
Combining these results, we prove that a function that has $U^{3}$-norm bounded away from zero either has
$U^{2}$-norm bounded away from zero, or else correlates in a sub-progression with a totally
equidistributed polynomial sequence of order 2 of a very special form.

Essentially all definitions and results of this section extend without important changes
to arbitrary nilmanifolds. To ease notation, we restrict to the case of nilmanifolds of order 2 as these are the
only ones needed in this article.

\subsection{Filtration and Nilmanifolds}
We review some standard material on nilmanifolds.
\begin{defn}[Filtration]\label{filt}
Let $G$ be a connected, simply connected Lie group with identity element $e_{G}$ and let $d\in\mathbb{N}$. A \emph{filtration} on $G$ is a finite sequence $G_{\bullet}=\{G_{i}\}_{i=0}^{d+1}$ of closed connected subgroups of $G$ such that
\begin{equation}\nonumber
  \begin{split}
    G=G_{0}=G_{1}\supseteq G_{2}\supseteq\cdots\supseteq G_{d+1}=\{e_{G}\}
  \end{split}
\end{equation}
and $[G_{i},G_{j}]\subseteq G_{i+j}$ for all $i,j\geq 0$, where $[H,H']\coloneqq \{hh'h^{-1}h'^{-1}\colon h\in H, h'\in H'\}$ for all $H,H'<G$. The integer $d$ is called the \emph{order} of $G_{\bullet}$.
\end{defn}

\begin{rem}
It is worth noting that $G_{i}$ is not necessarily the $i$-th commutator subgroup of $G$.
\end{rem}

Let $X=G/\Gamma$ be a nilmanifold of order 2 with filtration $G_{\bullet}$, i.e. $G$ is a Lie group with $[G,G]\subset G_{2},[G,G_{2}]=\{e_{G}\}$, and $\Gamma$ is a discrete cocompact subgroup of $G$. From now on, we assume that $G$ is connected and simple connected. The nilmanifold $X$ is endowed with a base point $e_{X}$ which is the projection to $X$ of the unit element of $G$. The action of $G$ on $X$ is denoted by $(g,x)\rightarrow g\cdot x$. The \emph{Haar measure} $m_{X}$ of $X$ is the unique probability measure on $X$ that is invariant under this action.

We denote the dimension of $G$ by $m$ and the dimension of $G_{2}$ by $r$ and set $s=m-r$. We implicitly assume that $G$ is endowed with a \emph{Mal'cev basis} $\mathcal{X}$, meaning $\mathcal{X}$ is a basis $(\xi_{1},\dots,\xi_{m})$ of the Lie algebra $\mathfrak{g}$ of $G$ that has the following properties:

(i) The map $\phi\colon \mathbb{R}^{m}\rightarrow G$ given by
\begin{equation}\nonumber
  \begin{split}
    \phi(t_{1},\dots,\xi_{m})=\exp(t_{1}\xi_{1})\cdot\ldots\cdot\exp(t_{m}\xi_{m})
  \end{split}
\end{equation}
is a homeomorphism from $\mathbb{R}^{m}$ onto $G$;

(ii) $G_{2}=\phi(\{0\}^{s}\times\mathbb{R}^{r})$;

(iii) $\Gamma=\phi(\mathbb{Z}^{m})$.

  We call $\phi$ the \emph{Mal'cev homeomorphism} of $G$ (or $X$) and call $\mathbb{R}^{m}$ the \emph{domain} of $\phi$. Any submanifold $Y$ of $X$ can be realized as $Y=\phi(V)$ for some subspace $V$ of $\mathbb{R}^{m}$. Then $\phi|_{Y}$ naturally induces a Mal'cev homeomorphism of $Y$. We call this induced map \emph{induced Mal'cev homeomorphism} from $X$ to $Y$.

Let $\mathfrak{g}$ be endowed with the Euclidean structure making $\mathcal{X}$ an orthonormal basis. This induces a Riemannian structure on $G$ that is invariant under right translations. The group $G$ is endowed with the associated geodesic distance, which we denote by $d_{G}$. This distance is invariant under right translations.

Let the space $X=G/\Gamma$ be endowed with the quotient metric $d_{X}$. Writing $p\colon G\rightarrow X$ for the quotient map, the metric $d_{X}$ is defined by
    \begin{equation}\nonumber
  \begin{split}
  d_{X}(x,y)=\inf_{g,h\in G}\{d_{G}(g,h)\colon p(g)=x, p(h)=y\}.
  \end{split}
\end{equation}
Since $\Gamma$ is discrete, it follows that the infimum is attained. 
Throughout, we frequently use the fact that 
$\Vert f\Vert_{Lip(X)}\leq \Vert f\Vert_{\mathcal{C}^{1}(X)}$ for all smooth functions $f$ on $X$ (recall that $\mathcal{C}^{n}(X)$ is the space of functions on $X$ with continuous $n$-th derivative).

The following lemma rephrases Lemma 4.1 in ~\cite{FH}:

\begin{lem}\label{cts}(Continuity property).
  For every bounded subset $F$ of $G$, there exists $H>0$ such that

  (i) $d_{G}(g\cdot h,g\cdot h')\leq Hd_{G}(h,h')$ and $d_{X}(g\cdot x,g\cdot x')\leq Hd_{X}(x,x')$ for all $h,h'\in G, x,x'\in X$ and $g\in F$;

  (ii) for any $n\geq 1$, every $f\in\mathcal{C}^{n}(X)$ and every $g\in F$, writing $f_{g}(x)=f(g\cdot x)$, we have $\Vert f_{g}\Vert_{\mathcal{C}^{n}(X)}\leq H\Vert f\Vert_{\mathcal{C}^{n}(X)}$, where $\Vert \cdot\Vert_{\mathcal{C}^{n}(X)}$ is the usual $\mathcal{C}^{n}$-norm on $X$, and $\mathcal{C}^{n}(X)$ is the collection of functions with bounded $\mathcal{C}^{n}$-norm on $X$.
\end{lem}

The following definitions are from ~\cite{GT}:

\begin{defn}[Vertical torus] We keep the same notations as above. The vertical torus is the sub-nilmanifold $G_{2}/(G_{2}\cap\Gamma)$.
The basis induces an isometric identification between $G_{2}$ and $\mathbb{R}^{r}$, and thus of the vertical torus endowed with the quotient metric,
with $\mathbb{T}^{r}$ endowed with its usual metric. Every $\bold{k}\in\mathbb{Z}^{r}$ induces a character $\bold{u}\rightarrow\bold{k}\cdot\bold{u}$ of
the vertical torus. A function $F$ on $X$ is a \emph{nilcharacter with frequency $\bold{k}$} if $F(\bold{u}\cdot x)=e(\bold{k}\cdot\bold{u})F(x)$ for every
$\bold{u}\in\mathbb{T}^{r}=G_{2}/(G_{2}\Gamma)$ and every $x\in X$. The nilcharacter is \emph{non-trivial} if its frequency is non-zero.
\end{defn}
\begin{defn}[Maximal torus and horizontal characters] Let $X=G/\Gamma$ be a nilmanifold of order 2 and let $m$ and $r$ be as above, and let $s=m-r$. The Mal'cev basis induces an isometric identification between the maximal torus $G/([G,G]\Gamma)$, endowed with the quotient metric, and $\mathbb{T}^{s}$, endowed with its usual metric. A \emph{horizontal character} is a continuous group homomorphism $\eta\colon G\rightarrow\mathbb{T}$ with trivial
restriction to $\Gamma$.
\end{defn}
A horizontal character $\eta$
factors through the maximal torus, and we typically abuse notation
 and think of $\eta$ as a
character of the maximal torus, and identify $\eta$ with an element
 $\bold{k}$ of $\mathbb{Z}^{s}$ by the following
rule $\bold{\A}\rightarrow\bold{k}\cdot\bold{\A}=k_{1}\A_{1}+\dots+k_{s}\A_{s}$
for $\bold{\A}=(\A_{1},\dots,\A_{s})\in\mathbb{T}^{s}$ and
$\bold{k}=(k_{1},\dots,k_{s})\in\mathbb{Z}^{s}$.
We define $\Vert\eta\Vert\coloneqq\vert k_{1}\vert+\dots+\vert k_{s}\vert$.

\subsection{Modified inverse theorem}
We first recall the Inverse Theorem proved by Szegedy in ~\cite{Sz}.
We recall the definition of polynomial sequences:
\begin{defn}[Polynomial sequences]
Let $t\in\mathbb{N}$. For $\vec{n}=(n_{1},\dots,n_{t}),\vec{m}=(m_{1},\dots,m_{t})\in\mathbb{N}^{t}$, denote
  $\binom{\vec{n}}{\vec{m}}=\prod_{i=1}^{t}\binom{n_{i}}{m_{i}}$.
For every $t$-tuple $\vec{n}=(n_{1},\dots,n_{t})\in\mathbb{Z}^{t}$, every group $G$ and every function $\phi\colon \mathbb{Z}^{t}\rightarrow G$, we denote $D_{\vec{n}}\phi\colon \mathbb{Z}^{t}\rightarrow G$ by $D_{\vec{n}}\phi(\vec{x})=\phi(\vec{x}+\vec{n})\phi^{-1}(\vec{x})$.

A function $\phi\colon \mathbb{Z}^{t}\rightarrow G$ is called a \emph{polynomial sequence} (or a \emph{polynomial map}) of degree $d$ with respect to the filtration $G_{\bullet}=\{G_{i}\}_{i=0}^{d+1}$ if $D_{\vec{h_{j}}}\dots D_{\vec{h_{1}}}\phi\in G_{j}$ for any $1\leq j\leq d+1$ and any $\vec{h_{1}},\dots,\vec{h_{j}}\in\mathbb{Z}^{t}$.
\end{defn}

For polynomials of degree 2, we have an explicit expression:
\begin{lem}\label{2p}(Corollary of Lemma 6.7 in ~\cite{GT}).
  A map $\phi\colon\mathbb{Z}^{2}\rightarrow G$ is polynomial of degree 2 with respect to the filtration $G_{\bullet}$ of order 2 if and only if it can be written as $\phi(m,n)=g_{0}g_{1,1}^{m}g_{1,2}^{n}g_{2,1}^{\binom{m}{2}}g_{2,2}^{mn}g_{2,3}^{\binom{n}{2}}$, where $g_{0},g_{1,1},g_{1,2}\in G, g_{2,1},g_{2,2},g_{2,3}\in G_{2}$.
\end{lem}

For every $N\in\mathbb{N}$, we write $[N]=\{1,\dots,N\}$.
We will use the following inverse theorem (Theorem 11 of ~\cite{Sz}):

\begin{thm}\label{inv2}(The inverse theorem for $\mathbb{Z}^{2}$ actions). For every $\e>0$, there exists $\d=\d(\e)>0, N_{0}=N_{0}(\e)\in\mathbb{N}$ and a nilmanifold $X=X(\e)$ of order 2 with respect to the filtration $G_{\bullet}$ such that for every $N\geq N_{0}$ and every $f\colon \R\rightarrow\mathbb{C}$ with
$\vert f\vert\leq 1$ and $\Vert f\Vert_{U^{3}(\R)}\geq\e$, there exist a function $\Phi\colon X\rightarrow\mathbb{C}$ with $\Vert\Phi\Vert_{Lip(X)}\leq 1$ and a polynomial sequence $g(m,n)\colon [\tilde{N}]\times[\tilde{N}]\rightarrow G$ of degree 2 on $G$ with respect to $G_{\bullet}$ such that $\vert\mathbb{E}_{m+ni\in\R}f(m+ni)\Phi(g(m,n)\cdot e_{X})\vert\geq\d$.
\end{thm}

In this paper, we need the following modified version of the above theorem.
Its proof is similar to Corollary 4.3 in ~\cite{FHel} and the argument given in Step 1 in the proof of Proposition \ref{est1} in the next section, so we omit it:

\begin{cor}\label{mIU3} (Modified $U^{3}$-inverse theorem).
  For every $\e>0$, there exist $\d=\d(\e)>0, M=M(\e),N_{0}=N_{0}(\e)\in\mathbb{N}$ and a finite family $\mathcal{H}=\mathcal{H}(\e)$ of nilmanifolds of order 2, of dimension at most $M$ and having a vertical torus of dimension 1, such that: for every $N\geq N_{0}$, if $f\colon \R\rightarrow\mathbb{C}$ is a function with $\vert f\vert\leq 1$ and $\Vert f\Vert_{U^{3}(\R)}\geq\e$, then either

  (i) $\Vert f\Vert_{U^{2}(\R)}\geq\d$; or

  (ii) there exists a nilmanifold $X$ belonging to the family $\mathcal{H}$, a polynomial map $g\colon\mathbb{Z}^{2}\rightarrow G$ of degree 2 with respect to the filtration $G_{\bullet}$, and a nilcharacter $\ps$ of $X$ with frequency 1, such that $\Vert\ps\Vert_{\mathcal{C}^{2m}(X)}\leq 1$ and $\Bigl\vert\mathbb{E}_{m+ni\in\R}f(m+ni)\ps(g(m,n)\cdot e_{X})\Bigr\vert\geq\d$.
\end{cor}

\subsection{Modified factorization theorem}
We next review the factorization theorem proved in ~\cite{GT}.
For any $t$-tuple $\vec{N}=(N_{1},\dots,N_{t})\in\mathbb{N}^{t}$, write $[\vec{N}]=[N_{1}]\times\dots\times[N_{t}]$.
We need some definitions before we state the theorem:

\begin{defn}[Totally equidistributed sequences]\label{equid}
  Let $X=G/\Gamma$ be a nilmanifold and $\vec{N}=(\tilde{N},\dots,N_{t})\in\mathbb{N}^{t}$ be a $t$-tuple. A sequence $g\colon[\vec{N}]\rightarrow G$ is called \emph{totally $\e$-equidistributed} if
    \begin{equation}\label{33}
  \begin{split}
  \Bigl\vert\mathbb{E}_{\vec{n}\in[\vec{N}]}\bold{1}_{P_{1}\times\dots\times P_{t}}(\vec{n})F(g(\vec{n})\cdot e_{X})\Bigr\vert\leq\e
  \end{split}
\end{equation}
for all $F\in Lip(X)$ with $\Vert F\Vert_{Lip(X)}\leq 1$ and $\int F dm_{X}=0$ and all arithmetic progressions $P_{i}$ in $[N_{i}]$.
\end{defn}

Modulo a change in the constants, our definition of total equidistribution is equivalent
to the one given in ~\cite{GT}.

\begin{rem}
 We only uses the case $t\leq 2$ and $N_{1}=N_{2}$ if $t=2$. But we state some of the definitions and results in 
full generality in case of further researches.
\end{rem}

\begin{defn}[Smooth sequences]
  Given a nilmanifold $G/\Gamma, M\in\mathbb{N}$, and $\vec{N}=(N_{1},\dots,N_{t})\in\mathbb{N}^{t}$, we say that the sequence $\e\colon[\vec{N}]\rightarrow G$ is \emph{$(M,\vec{N})$-smooth} if for every $\vec{n}\in[\vec{N}]$, we have $d_{G}(\e(\vec{n}),\bold{1}_{G})\leq M$ and $d_{G}(\e(\vec{n}),\e(\vec{n}-\vec{e_{i}}))\leq M/N_{i}$ for all $1\leq i\leq t$, where the $i$-th coordinate of $\vec{e_{i}}$ is 1 and all other coordinates are 0.
\end{defn}

\begin{defn}[Rational sequences]
  We say an element $g\in G$ is \emph{$Q$-rational} for some $Q\in\mathbb{N}$ if there exists $m\leq Q$ such that $g^{m}\in\Gamma$. We say that $g$ is \emph{rational} if it is $Q$-rational for some $Q\in\mathbb{N}$.

  We say that a sequence $\gamma\colon[\vec{N}]\rightarrow G$ is \emph{$Q$-rational} if for every $\vec{n}\in[\vec{N}], \gamma(\vec{n})$ is $Q$-rational.
\end{defn}

\begin{defn}[Rational subgroup]
A \emph{rational subgroup} $G'$ of $G$ is a closed and connected subgroup of $G$ such
that its Lie algebra $\mathfrak{g}'$ admits a base that has rational coordinates in the Mal'cev basis
of $G$.
\end{defn}

\begin{defn}[Filtration of subgroups]
Suppose $G$ is a group with filtration $G_{\bullet}=\{G_{i}\}_{i=0}^{d+1}$ and $G'$ is a subgroup of $G$. We denote $G'_{\bullet}=\{G_{i}\cap G\}_{i=0}^{d+1}$. This is a filtration of $G'$. We call $G'_{\bullet}$ the filtration induced by $G_{\bullet}$.
\end{defn}

In ~\cite{er}, ~\cite{er2} and ~\cite{GT}, the next result is stated only for a function with the form $\omega(M)=M^{-A}$ for some $A>0$, but the same proof works for arbitrary functions $\omega\colon\mathbb{N}\rightarrow\mathbb{R}^{+}$. Recall the fact that if $G'$ is a subgroup of $G$, then $G'/(G'\cap\Gamma)$ is a sub-nilmanifold of $G/\Gamma$.

\begin{thm}\label{mFo}(Factorization of polynomial sequences, Theorem 10.2 in ~\cite{GT}). Suppose that $X=G/\Gamma$ is a nilmanifold of order 2 with respect to the filtration $G_{\bullet}$. For every $M\in\mathbb{N}$, there exists a finite collection $\mathcal{F}(M)$ of sub-nilmanifolds of $X$, each of the form $X'=G'/\Gamma'$, where $G'$ is a rational subgroup of
$G$ and $\Gamma'=G'\cap\Gamma$, such that the following holds:

For every function $\omega\colon\mathbb{N}\rightarrow\mathbb{R}^{+}$ and every $M_{0}\in\mathbb{N}$, there exists $M_{1}=M_{1}(M_{0},X,\omega)\in\mathbb{N}$ such that for every $N\in\mathbb{N}$ and every polynomial sequence of degree 2  $(g(m,n))_{(m,n)\in[\tilde{N}]\times[\tilde{N}]}$ in $G$ with respect to the filtration $G_{\bullet}$, there exist $M\in\mathbb{N}$ with $M_{0}\leq M\leq M_{1}$, a nilmanifold $X'\in\mathcal{F}(M)$, and a decomposition
\begin{equation}\nonumber
  \begin{split}
  g(m,n)=\e(m,n)g'(m,n)\gamma(m,n), (m,n)\in[\tilde{N}]\times[\tilde{N}],
  \end{split}
\end{equation}
where $\e,g',\gamma$ are polynomials of degree 2 with respect to the filtration $G_{\bullet}$ such that

(i) $\e$ is $(M,(\tilde{N},\tilde{N}))$-smooth;

(ii) $(g'(m,n))_{(m,n)\in[\tilde{N}]\times[\tilde{N}]}$ takes values in $G'$, and the finite sequence $(g'(m,n)\cdot e_{X'})_{(m,n)\in[\tilde{N}]\times[\tilde{N}]}$ is totally $\omega(M)$-equidistributed in $X'$ with the metric $d_{X'}$;

(iii) $\gamma\colon[\tilde{N}]\times[\tilde{N}]\rightarrow G$ is $M$-rational, and $(\gamma(m,n)\cdot e_{X})_{(m,n)\in[\tilde{N}]\times[\tilde{N}]}$ is doubly periodic with periods at most $M$.
\end{thm}

We use the following corollary of the previous result that gives a more precise
factorization for a certain explicit class of polynomial sequences. The proof is similar to the discussion in ~\cite{FHel}:

\begin{cor}\label{mF}(Modified factorization theorem).
  Let $X=G/\Gamma$ be a nilmanifold of order 2 with respect to the filtration $G_{\bullet}$ and with vertical torus of dimension 1. For every $M\in\mathbb{N}$, there exists a finite collection $\mathcal{F}(M)$ of sub-nilmanifolds of $X$, each of the form $X'=G'/\Gamma'$ with filtration $G'_{\bullet}$, where $G'$ is a rational subgroup of $G$ and $\Gamma'=G'\cap\Gamma$, and either

  (i) $G'$ is an abelian rational subgroup of $G$; or

  (ii) $G'$ is a non-abelian rational subgroup of $G$ and $G_{2}'/(G_{2}'\cap\Gamma')$ has dimension 1, such that the following holds:

  For every $\omega\colon\mathbb{N}\rightarrow\mathbb{R}^{+}$ and every $M_{0}\in\mathbb{N}$, there exists $M_{1}=M_{1}(M_{0},X,\omega)\in\mathbb{N}$ such that for every $N\in\mathbb{N}$ and every polynomial $g\colon\mathbb{Z}^{2}\rightarrow G$ of degree 2 with respect to the filtration $G_{\bullet}$, there exist $M\in\mathbb{N}$ with $M_{0}\leq M\leq M_{1}$, a nilmanifold $X'\in\mathcal{F}(M)$, and a decomposition
\begin{equation}\nonumber
  \begin{split}
  g(m,n)=\e(m,n)g'(m,n)\gamma(m,n), (m,n)\in[\tilde{N}]\times[\tilde{N}]
  \end{split}
\end{equation}
such that

(iii) $\e$ is $(M,(\tilde{N},\tilde{N}))$-smooth;

(iv) for any $(m,n)\in[\tilde{N}]\times[\tilde{N}]$,
\begin{equation}\nonumber
  \begin{split}
  g'(m,n)=g_{0}g_{1,1}^{m}g_{1,2}^{n}g_{2,1}^{\binom{m}{2}}g_{2,2}^{mn}g_{2,3}^{\binom{n}{2}},
  \end{split}
\end{equation}
where $g_{0}, g_{1,1}, g_{1,2},g_{2,1},g_{2,2},g_{2,3}\in G'$, and moreover $g_{2,1},g_{2,2},g_{2,3}\in G'_{2}$ in case (ii), and $(g'(m,n)\cdot e_{X'})_{(m,n)\in[\tilde{N}]\times[\tilde{N}]}$ is totally $\omega(M)$-equidistributed in $X'$ with the metric $d_{X'}$;

(v) $\gamma\colon[\tilde{N}]\times[\tilde{N}]\rightarrow G$ is $M$-rational, and $(\gamma(m,n))_{(m,n)\in[\tilde{N}]\times[\tilde{N}]}$ is doubly periodic with periods at most $M$.
\end{cor}

\begin{proof}
Let the integers $M_{1}, M$ and the nilmanifold $X'=G'/\Gamma'\in\mathcal{F}
(M)$ be given by Theorem \ref{mFo}. Note that the sequence
$(g(m,n))_{(m,n)\in[\tilde{N}]\times[\tilde{N}]}$ is a degree 2 polynomial sequence in $G$
with respect to the filtration $G_{\bullet}$. Let
$(g'(m,n))_{(m,n)\in[\tilde{N}]\times[\tilde{N}]}$ be the sequence given by the decomposition
of Theorem \ref{mFo}. By Lemma \ref{2p}, we can write
\begin{equation}\nonumber
  \begin{split}
  g'(m,n)=g_{0}g_{1,1}^{m}g_{1,2}^{n}g_{2,1}^{\binom{m}{2}}g_{2,2}^{mn}g_{2,3}^{\binom{n}{2}}
  \end{split}
\end{equation}
for some $g_{0}, g_{1,1}, g_{1,2}\in G, g_{2,1},g_{2,2},g_{2,3}\in G_{2}$.
It remains to show that
$g_{0}, g_{1,1}, g_{1,2}\in G'$, $g_{2,1},g_{2,2},g_{2,3}\in G'_{2}$
for case (ii).

Since $g_{0}=g(0,0)$, we have $g_{0}\in G'$. So the sequence
$g''(m,n)=g_{0}^{-1}g'(m,n)$ also takes values in $G'$.
Denote
$$\partial_{(x,y)}g(m,n)=g(m+x,n+y)g(m,n)^{-1}.$$ Since $G$ is of order 2,
$G_{2}$ is included in the center of $G$. Therefore, we obtain
\begin{equation}\nonumber
  \begin{split}
  \partial_{(1,0)}g''(m,0)=g_{1,1}g_{2,1}^{m} \text{ and }
  \partial_{(1,0)}^{2}g''(m,0)=g_{2,1}.
  \end{split}
\end{equation}
It follows that $g_{1,1},g_{2,1}\in G'$. Similarly,
$g_{1,2},g_{2,3}\in G'$. Thus the sequence
$g'''(m,n)=g_{2,2}^{m,n}$ also takes values in $G'$, which implies
$g_{2,2}\in G'$.

If we are in case (i), then $G'$ is abelian and we are done. If
$G'$ is not abelian, then $G'_{2}$ is a non-trivial subgroup of $G_{2}$.
Moreover, $G'_{2}$ is closed and connected, and by hypothesis $G_{2}$
is isomorphic to the torus $\mathbb{T}$. It follows that $G'_{2}=G_{2}$.
Hence $g_{2,1},g_{2,2},g_{2,3}\in G'_{2}$. This finishes the proof.
\end{proof}

\begin{rem}\label{mF2}
  We remark that Corollary \ref{mF} cannot be generalized to nilmanifolds of order $s>2$ by this method as our proof relies on the fact that $G_{2}$ lies in the center of $G$, which only holds for $s=2$.
\end{rem}

\section{Correlation of multiplicative functions with polynomial sequences}
The main goal of this section is to establish some correlation estimates needed in the next section. We show that multiplicative
functions do not correlate with a class of totally equidistributed polynomial
sequences of order 2. The precise statements appear in Propositions \ref{est2} and \ref{est1}.

\begin{defn}[Smoothness norms]
  Suppose $g\colon\mathbb{Z}^{t}\rightarrow\mathbb{R}/\mathbb{Z}$ is a polynomial map with Taylor expansion
  \begin{equation}\nonumber
  \begin{split}
  g(\vec{n})=\sum_{\vec{j}}\binom{\vec{n}}{\vec{j}}a_{\vec{j}},
  \end{split}
\end{equation}
where $a_{\vec{j}}\in\mathbb{R}/\mathbb{Z}$, and the sum is taken over all $\vec{j}=(j_{1},\dots,j_{t})$ such that $j_{1}+\dots+j_{t}\leq d$ for some $d\in\mathbb{N}$. For any $t$-tuple $\vec{N}=(N_{1},\dots,N_{t})\in\mathbb{N}^{t}$, denote
  \begin{equation}\nonumber
  \begin{split}
  \Vert g\Vert_{C^{\infty}[\vec{N}]}=\sup_{\vec{j}\neq\vec{0}}\vec{N}^{\vec{j}}\Vert a_{\vec{j}}\Vert_{\mathbb{R}/\mathbb{Z}},
  \end{split}
\end{equation}
where $\vec{N}^{\vec{j}}=N_{1}^{j_{1}}\cdot\ldots\cdot N_{t}^{j_{t}}$.
\end{defn}

The following lemma modifies Lemma 8.4 of ~\cite{GT} for our purposes, and its proof is similar to Lemma 5.1 of
~\cite{FH}, so we omit it.
\begin{lem}\label{51}
 Let $\vec{N}=(N_{1},N_{2})\in\mathbb{N}^{2}, d,q,r\in\mathbb{N}$ and $a_{1},a_{2},b_{1},b_{2}$ be
integers with $0<\vert a_{1}\vert, \vert a_{2}\vert\leq q$ and $\vert b_{1}\vert\leq r N_{1}, \vert b_{2}\vert\leq r N_{2}$.
There exist $C=C(d,q,r)>0$ and $\ell=\ell(q,d)\in\mathbb{N}$ such that if
$\phi\colon\mathbb{Z}^{2}\to\mathbb{R}/\mathbb{T}$ is a polynomial map of degree at most $d$ and 
$\psi$ is given by $\psi(n_{1},n_{2})=\psi(a_{1}n_{1}+b_{1},a_{2}n_{2}+b_{2})$, then
$$\Vert\ell\phi\Vert_{C^{\infty}[\vec{N}]}\leq C\Vert\psi\Vert_{C^{\infty}[\vec{N}]}.$$
\end{lem}

The next result is a variation of Theorem 8.6 in ~\cite{GT}. It provides a convenient criterion for establishing equidistribution properties
of polynomial sequences on nilmanifolds:

\begin{thm}\label{Lei}(A variation of the quantitative Leibman Theorem).
  Let $X=G/\Gamma$ be a nilmanifold of order 2 and $t\in\mathbb{N}$. Then for any $\e>0$ small enough, there exists $D=D(X,\e,t)>0$ such that for any $N\in\mathbb{N}$ and any polynomial sequence $g\colon [\tilde{N}]^{t}\rightarrow G$, if $(g(\vec{n})\cdot e_{X})_{\vec{n}\in[\tilde{N}]^{t}}$ is not totally $\e$-equidistributed, then there exists a horizontal character $\eta=\eta(X,\e,t)$ such that $0<\Vert\eta\Vert\leq D$ and $\Vert \eta\circ g\Vert_{C^{\infty}[\tilde{N}]^{t}}\leq D$.
\end{thm}

\begin{rem}
This theorem is stated in ~\cite{er}, ~\cite{er2} and ~\cite{GT} under the stronger hypothesis that the sequence is not ``$\e$-equidistributed in $X$", meaning (\ref{33}) fails for $P_{i}=[N],1\leq i\leq t$.
This stronger result can be obtained by using Theorem 5.2 ~\cite{FH} combined with a similar argument in Lemma 3.1 in ~\cite{er2} (see also this argument in Step 3 of Proposition \ref{est1}). We omit the proof. 
\end{rem}

The following is a partial converse of the above result:

\begin{thm}\label{inv}(Inverse Leibman Theorem).
  Let $X=G/\Gamma$ be a nilmanifold of order 2 and $d\in\mathbb{N}$. There exists $C=C(X,d)>0$ such that for every $D>0$, every 2-tuple $\vec{N}=(N_{1},N_{2})\in\mathbb{N}^{2}$ with both $N_{1}$ and $N_{2}$ sufficiently large depending only on $X$ and $D$, and every polynomial map $g\colon [\vec{N}]\rightarrow G$ of degree at most $d$, if there exists a non-trivial horizontal character $\eta$ of $X$ with $\Vert\eta\Vert\leq D$ and $\Vert \eta\circ g\Vert_{C^{\infty}[\vec{N}]}\leq D$, then the sequence $(g(\vec{n})\cdot e_{X})_{\vec{n}\in[\vec{N}]}$ is not totally $CD^{-3}$-equidistributed in $X$.
\end{thm}
\begin{rem}
  It is worth noting that Theorem \ref{Lei} and \ref{inv} hold for general nilmanifolds of order $s, s\in\mathbb{N}$. But we only need the case $s=2$ in this paper.
\end{rem}
\begin{proof}
  Since $\Vert \eta\circ g\Vert_{C^{\infty}[\vec{N}]}\leq D$, we have that
    \begin{equation}\nonumber
  \begin{split}
  \eta(g(m,n))=\sum_{\vec{j}}\binom{(m,n)}{\vec{j}}a_{\vec{j}},
  \end{split}
\end{equation}
where $\Vert a_{\vec{j}}\Vert\leq\frac{D}{N_{1}^{j_{1}}N_{2}^{j_{2}}}$ for all $0< j_{1}+j_{2}\leq d$. Thus
$\vert e(\eta(g(0,0)))-e(\eta(g(m,n)))\vert\leq 1/2$ for $1\leq m\leq cN_{1}/D, 1\leq n\leq cN_{2}/D$, where $c$ is a constant depending only on $d$. Suppose that $N_{1},N_{2}\geq\frac{8D}{c}$. Then
    \begin{equation}\nonumber
  \begin{split}
  \Bigl\vert\mathbb{E}_{m\leq cN_{1}/D, n\leq cN_{2}/D}e(\eta(g(m,n)))\Bigr\vert\geq\frac{1}{2},
  \end{split}
\end{equation}
which inplies that
    \begin{equation}\nonumber
  \begin{split}
  \Bigl\vert\mathbb{E}_{(m,n)\in[\vec{N}]}\bold{1}_{[cN_{1}/D]\times[cN_{2}/D]}(m,n)
    e(\eta(g(m,n)))\Bigr\vert\geq\frac{c^{2}}{2D^{2}}-\frac{c}{D}(\frac{1}{N_{1}}+\frac{1}{N_{2}})
\geq\frac{c^{2}}{4D^{2}}.
  \end{split}
\end{equation}

 Since $\Vert\eta\Vert\leq D$, the function $x\rightarrow e(\eta(x))$ defined on $X$ is Lipschitz with constant at most $C'D$ for some $C'=C'(X,d)$, and has integral 0 since $\eta$ is non-trivial. Therefore, the sequence $(g(\vec{n})\cdot e_{X})_{\vec{n}\in[\vec{N}]}$ is not totally $CD^{-3}$-equidistributed with $C=c^{2}/4C'$.
\end{proof}

The following lemma is an extension of Corollary 5.5 in ~\cite{FH} and the proof is similar:
\begin{lem}\label{shift}(Shifting the nilmanifold).
  Let $X=G/\Gamma$ be a nilmanifold of order 2, $G'$ be a rational subgroup
 of $G, h\in G$ be a rational element. Denote
 $X'=G'\cdot e_{X}, e_{Y}=h\cdot e_{X}, Y=G'\cdot e_{Y}$.
Then for every $\e>0$, there exists $\d=\d(G',X,h,\e)>0$ such
that for every polynomial $(g'(m,n))_{(m,n)\in[\tilde{N}]\times[\tilde{N}]}$ on
$G'$ of degree 2 with respect to $G'_{\bullet}$, if $(g'(m,n)\cdot e_{X})_{{m,n}\in[\tilde{N}]\times[\tilde{N}]}$ is totally $\d$-equidistributed in $X'$, then $(g'(m,n)\cdot e_{Y})_{(m,n)\in[\tilde{N}]\times[\tilde{N}]}$ is totally $\e$-equidistributed in $Y$.
\end{lem}

By Lemma B.6 in the appendix in ~\cite{FH}, $\Gamma\cap G'$ is co-compact in $G'$ and thus
$X'$ is a closed sub-manifold of $X$. In a similar fashion, $(h\Gamma h^{-1}\cap G')$
is co-compact in $G'$ and $Y$ is a closed sub-nilmanifold of $X$.

We are now going to prove the two main results of this section that give asymptotic orthogonality
of multiplicative functions to some totally equidistributed polynomial sequences. These results
are used later in the proof of Theorem ~\ref{nU3s} to treat each of the two distinct cases arising
from an application of Corollary \ref{mF}. Both proofs are based on Katai's orthogonality
criterion (Lemma \ref{katai}) and the quantitative Leibman Theorem (Theorem \ref{Lei}).

\subsection{Correlation of multiplicative functions with polynomial sequences: abelian case}
The goal of this subsection is to prove the following proposition:
\begin{prop}\label{est2}(Correlation property for abelian case). Let $X=\mathbb{T}^{s}$ for some $s\in\mathbb{N}$.
  For any $\D>0$, there exist $\sigma=\sigma(s,\D)>0,N_{0}=N_{0}(s,\D)\in\mathbb{N}$ such that for every $N\geq N_{0}$ and every totally $\sigma$-equidistributed degree 2 polynomial $(g(m,n))_{(m,n)\in[\tilde{N}]\times[\tilde{N}]}$ in $X$ of the form
  \begin{equation}\nonumber
  \begin{split}
    g(m,n)=\bold{\A}_{0}+\bold{\A}_{1}m+\bold{\A}_{2}n+\bold{\B}_{1}\binom{m}{2}+\bold{\B}_{2}mn+\bold{\B}_{3}\binom{n}{2},
  \end{split}
  \end{equation}
we have that
  \begin{equation}\nonumber
  \begin{split}
    \sup_{\B,\chi,\Phi,P}\Bigl\vert\mathbb{E}_{\A\in \R}\bold{1}_{P}(\A)\chi(\A)\Phi(g'(\B+\A)\cdot e_{X})\Bigr\vert<\D,
  \end{split}
  \end{equation}
  where $g'\colon \R\rightarrow G, g'(a+bi)=g(a,b)$, and the sup is taken over all $\B\in \R, \chi\in\mathcal{M}_{\Z}, \Phi\in Lip(X)$ with $\Vert\Phi\Vert_{Lip(X)}\leq 1$ and $\int\Phi dm_{X}=0$, and all $P=\{a+bi\colon a\in P_{1}, b\in P_{2}\}$, where $P_{i}$ is an arithmetic progression in $[\tilde{N}]$ for $i=1,2$.
\end{prop}

\begin{proof}
  In this proof, $C_{1},C_{2},\dots$ are constants depending only on $s$ and $\D$.

 Suppose that
    \begin{equation}\nonumber
  \begin{split}
    \Bigl\vert\mathbb{E}_{\A\in \R}\bold{1}_{P}(\A)\chi(\A)\Phi(g'(\B+\A)\cdot e_{X})\Bigr\vert\geq\D
  \end{split}
  \end{equation}
  for some $\B=b_{1}+b_{2}i\in \R, \chi\in\mathcal{M}_{\Z}, \Phi\in Lip(X)$ with $\Vert\Phi\Vert_{Lip(X)}\leq 1$ and $\int\Phi dm_{X}=0$, and some $P=\{a+bi\colon a\in P_{1}, b\in P_{2}\}$, where $P_{i}$ is an arithmetic progression in $[\tilde{N}]$ for $i=1,2$.  Without loss of generality, we assume $\Vert\Phi\Vert_{\mathcal{C}^{2s}(X)}\leq 1$.
Indeed,
there exists a function $\Phi'$ with $\Vert\Phi-\Phi'\Vert_{\infty}\leq \D/2$ and 
$\Vert\Phi'\Vert_{\mathcal{C}^{2s}(X)}$ is bounded by a constant depending only on $s$ and $\D$.
 Then
 \begin{equation}\nonumber
  \begin{split}
    &\D\leq\Bigl\vert\mathbb{E}_{\A\in \R}\bold{1}_{P}(\A)\chi(\A)\Phi(g'(\B+\A)\cdot e_{X})\Bigr\vert
    =\Bigl\vert\sum_{\bold{k}\in\mathbb{Z}^{s}}\widehat{\Phi}(\bold{k})\mathbb{E}_{\A\in \R}\bold{1}_{P}(\A)\chi(\A)e(\bold{k}\cdot g'(\B+\A))\Bigr\vert
    \\&\leq\sum_{\bold{k}\in\mathbb{Z}^{s}\backslash{\{\bold{0}\}}}\frac{C_{0}}{1+\Vert\bold{k}\Vert^{2s}}\Bigl\vert\mathbb{E}_{\A\in \R}\bold{1}_{P}(\A)\chi(\A)e(\bold{k}\cdot g'(\B+\A))\Bigr\vert
  \end{split}
  \end{equation}
  for some constant $C_{0}=C_{0}(s)$.
   So there exist $C_{1}>0, \theta=\theta(s,\D)>0,$ and $\bold{k}\in\mathbb{Z}^{s}\backslash{\{\bold{0}\}}$ such that $\Vert \bold{k}\Vert\leq C_{1}$ and
  \begin{equation}\nonumber
  \begin{split}
    \Bigl\vert\mathbb{E}_{\A\in \R}\bold{1}_{P}(\A)\chi(\A)e(\bold{k}\cdot g'(\B+\A))\Bigr\vert\geq\theta.
  \end{split}
  \end{equation}
  Let $\d=\d(\theta,1),K=K(\theta,1)$ be the constants that appear in Lemma \ref{katai}. Then there exist $p=p_{1}+p_{2}i,q=q_{1}+q_{2}i\in\P, \mathcal{N}(p)<\mathcal{N}(q)\leq K$ such that
   \begin{equation}\nonumber
  \begin{split}
    \frac{1}{N^{2}}\Bigl\vert\sum_{\A\in \R/p\cap \R/q}\bold{1}_{P}(p\A)\bold{1}_{P}(q\A)e(\bold{k}\cdot (g'(\B+p\A)-g'(\B+q\A)))\Bigr\vert\geq\d.
  \end{split}
\end{equation}

Let $S=\R/p\cap \R/q$. For $j\in\mathbb{Z}$, set $V_{j}=\{\A\in S\colon Im(\A)=j\}$. Then $S=\bigcup_{j=-2N}^{2N}V_{j}$ and one can verify that $V_{j}=\{a+ji\colon a\in Q_{j}\}$ for some arithmetic progression $Q_{j}$ in $[\tilde{N}]$. So there exists $j\in\mathbb{Z}$ with $-2N\leq j\leq 2N$ such that
 \begin{equation}\nonumber
  \begin{split}
    \frac{1}{N}\Bigl\vert\sum_{\A\in V_{j}}e(\bold{k}\cdot (g'(\B+p\A)-g'(\B+q\A)))\Bigr\vert\geq\d/5.
  \end{split}
\end{equation}
Write
 \begin{equation}\nonumber
  \begin{split}
    h(a)=\bold{k}\cdot\Bigl(g(b_{1}-p_{2}j+p_{1}a,b_{2}+p_{1}j+p_{2}a)-g(b_{1}-q_{2}j+q_{1}a,b_{2}+q_{1}j+q_{2}a)\Bigr).
  \end{split}
\end{equation}
Then the sequence $(h(a))_{a\in[N]}$ is not totally $\d/5$-equidistributed in the circle. By the abelian version of Theorem \ref{Lei}, there exists $0<\ell\leq D=D(\d/5)$ such that 
 \begin{equation}\label{tempr1}
  \begin{split}
    \Vert \ell h\Vert_{C^{\infty}[N]}\leq D.
  \end{split}
\end{equation}

Write $\A'_{i}=\bold{k}\cdot\bold{\A}_{i}, i=0,1,2, \B'_{i}=\bold{k}\cdot\bold{\B}_{i}, i=1,2,3$. Let $A=p_{1}^{2}-q_{1}^{2}, B=2p_{1}p_{2}-2q_{1}q_{2}, C=p_{2}^{2}-q_{2}^{2}$. One can calculate that
 \begin{equation}\nonumber
  \begin{split}
    h(a)=\Bigl(\B'_{1}A+\B'_{2}B+\B'_{3}C\binom{a}{2}+(w+\A'_{1}(p_{1}-q_{1})+\A'_{2}(p_{2}-q_{2})\Bigr) a+r
  \end{split}
\end{equation}
for some $r\in\mathbb{R}/\mathbb{Z}$ independent of $a$, where
 \begin{equation}\nonumber
  \begin{split}
    &w=\B'_{1}\Bigl(\bigl(\binom{p_{1}}{2}-\binom{q_{1}}{2}\bigr)+b_{1}(p_{1}-q_{1})-j(p_{1}p_{2}-q_{1}q_{2})\Bigr)
    \\&\qquad+\B'_{2}\Bigl( j(p_{1}^{2}-p_{2}^{2}-q_{1}^{2}+q_{2}^{2})+b_{2}(p_{1}-q_{1})+b_{1}(p_{2}-q_{2})+(p_{1}p_{2}-q_{1}q_{2})\Bigr)
    \\&\qquad+\B'_{3}\Bigl(\bigl(\binom{p_{2}}{2}-\binom{q_{2}}{2}\bigr)+b_{2}(p_{2}-q_{2})+j(p_{1}p_{2}-q_{1}q_{2})\Bigr).
  \end{split}
\end{equation}
Thus by (\ref{tempr1}), we have
 \begin{equation}\label{c5}
  \begin{split}
    &\Bigl\Vert\ell(\B'_{1}A+\B'_{2}B+\B'_{3}C)\Bigr\Vert\leq\frac{D}{N^{2}}; \\&\Bigl\Vert\ell(w+\A'_{1}(p_{1}-q_{1})+\A'_{2}(p_{2}-q_{2}))\Bigr\Vert\leq\frac{D}{N}.
  \end{split}
\end{equation}
The first equation of (\ref{c5}) shows that $\B'_{1}A+\B'_{2}B+\B'_{3}C$ is at a distance $\leq C_{2}/N^{2}$ from a rational number with denominator $\leq C_{3}$.

If we set $U_{j}=\{\A\in S\colon Re(\A)=j\}$, then $S=\bigcup_{j=-2N}^{2N}U_{j}$ and a similar argument shows that $\B'_{1}C-\B'_{2}B+\B'_{3}A$ is at a distance $\leq C_{2}/N^{2}$ from a rational number with denominator $\leq C_{3}$. Setting $W_{j}=\{\A=\A_{1}+\A_{2}i\in S\colon \A_{1}-\A_{2}=j\}$, then $S=\bigcup_{j=-4N}^{4N}W_{j}$ and a similar argument shows that $-\B'_{1}B+\B'_{2}(A-C)+\B'_{3}B$ is at a distance $\leq C_{2}/N^{2}$ from a rational number with denominator $\leq C_{3}$. Since $\mathcal{N}(p)\neq \mathcal{N}(q)$, we have that $A+C\neq 0$. Thus
 \begin{equation}\nonumber
\begin{vmatrix}

A & B &  C\\

C & -B & A\\

-B & A-C & B\\

\end{vmatrix}=-(A+C)((A-C)^{2}+2B^{2})\neq 0,
\end{equation}
as $(A-C)^{2}+2B^{2}=0$ if and only if $p=\pm q$. This implies that $\B'_{i}$ is at a distance $\leq C_{4}/N^{2}$ of a rational number with denominator $\leq C_{5}$ for $i=1,2,3$. Combining the second equation of (\ref{c5}), we get that $\A'_{1}(p_{1}-q_{1})+\A'_{2}(p_{2}-q_{2})$ is at a distance $\leq C_{8}/N$ from a rational number with denominator $\leq C_{9}$ (here we use the fact that $-4N\leq j,b_{1},b_{2}\leq 4N$). If we consider the set $U_{j}=\{\A\in S\colon Re(\A)=j\}$, a similar argument shows that $-\A'_{1}(p_{2}-q_{2})+\A'_{2}(p_{1}-q_{1})$ is at a distance $\leq C_{8}/N$ from a rational number with denominator $\leq C_{9}$. Since
 \begin{equation}\nonumber
\begin{vmatrix}

p_{1}-q_{1} & p_{2}-q_{2}\\

-p_{2}+q_{2} & p_{1}-q_{1}\\

\end{vmatrix}=(p_{1}-q_{1})^{2}+(p_{2}-q_{2})^{2}\neq 0,
\end{equation}
we deduce that $\A'_{i}$ is at a distance $\leq C_{10}/N$ from a rational number with denominator $\leq C_{11}$ for $i=1,2$.

Thus we can find some non-zero integer $\ell'$ with $\vert\ell'\vert\leq C_{12}/C_{1}$ such that $\Vert\ell'\B'_{i}\Vert\leq C_{12}/N^{2}$, for $i=1,2,3$, and $\Vert\ell'\A'_{i}\Vert\leq C_{12}/N$, for $i=1,2$. Taking $\bold{k}'=\ell'\bold{k}$, we deduce that $\Vert\bold{k}'\cdot g\Vert_{\mathcal{C}^{\infty}[\tilde{N}]\times[\tilde{N}]}\leq C_{12}$. By Theorem \ref{inv}, $(g(m,n))_{(m,n)\in[\tilde{N}]\times[\tilde{N}]}$ is not totally $\sigma$-equidistributed for some $\sigma>0$ depending only on $s$ and $\D$. This finishes the proof.
\end{proof}

\subsection{Correlation of multiplicative functions with polynomial sequences: non-abelian case}

The goal of this subsection is to prove the following proposition:

\begin{prop}\label{est1}(Correlation property for non-abelian case).
  For any nilmanifold $X=G/\Gamma$ of order 2 with filtration $G_{\bullet}$ and any $\D>0$, there exist $\sigma=\sigma(X,\D)>0$ and $N_{0}=N_{0}(X,\D)\in\mathbb{N}$ such that for every $N\geq N_{0}$ and every totally $\sigma$-equidistributed degree 2 polynomial $(g(\vec{n}))_{\vec{n}\in[\tilde{N}]\times[\tilde{N}]}$ in $G$ with respect to $G_{\bullet}$, we have that
  \begin{equation}\nonumber
  \begin{split}
    \sup_{\B,\chi,\Phi,P}\Bigl\vert\mathbb{E}_{\A\in \R}\bold{1}_{P}(\A)\chi(\A)\Phi(g'(\B+\A)\cdot e_{X})\Bigr\vert<\D,
  \end{split}
  \end{equation}
  where $g'\colon \R\rightarrow G, g'(a+bi)=g(a,b)$, and the sup is taken over all $\B\in \R, \chi\in\mathcal{M}_{\Z}, \Phi\in Lip(X)$ with $\Vert\Phi\Vert_{Lip(X)}\leq 1$ and $\int\Phi dm_{X}=0$, and all $P=\{a+bi\colon a\in P_{1}, b\in P_{2}\}$, where $P_{i}$ is an arithmetic progression in $[\tilde{N}]$, for $i=1,2$.
\end{prop}

Proposition \ref{est1} is the key difficulty of the paper, so we briefly explain the general strategy of the proof before we proceed. First of all, using the vertical Fourier transform, we may assume without loss of generality that $\Phi$ is a non-trivial nilcharacter. We may further simplify the argument to the case when $\B=0$ and $g'$ does not have the constant term. 
By Lemma \ref{katai}, it suffices to study the total equidistribution property of the sequence $(g'(p\A),g'(q\A))_{\A\in\R}$ on $G\times G$ with respect to the function $\Phi\otimes\overline{\Phi}$. This is studied carefully in Proposition \ref{last}, where we provide a necessary condition for the sequence $(g'(p\A),g'(q\A))_{\A\in\R}$. Finally, this condition leads to a contradiction by a linear algebraic argument.

It is worth making a few remarks about Proposition \ref{last}, which is the main technical part of the proof. Roughly speaking, we wish to give a necessary condition for a sequence $(g_{1}(\A),g_{2}(\A))_{\A\in\R}$ on $G\times G$ of degree 2 not being totally equidistributed. If $(g_{1}(\A),g_{2}(\A))_{\A\in\R}$ is not totally equidistributed, the Approximation Lemma (Lemma \ref{appro2}) allows us to replace any polynomial sequence on some sub-nilmanifold $H$ of $X\times X$ with another one lying on a sub-nilmanifold $H'$ of $H$ such that they have a similar total equidistribution behavior, unless $H$ satisfies some degeneracy property. This enables us to reduce the problem to the case where $(g_{1}(\A),g_{2}(\A))_{\A\in\R}$ lies in a submanifold of $X\times X$ with additional algebriac structures, which leads to the necessary condition we desire.

\begin{conv}\label{conv2}
Throughout this subsection, $X=G/\Gamma$ is a nilsystem of order 2 with filtration $G_{\bullet}$. We assume without generality that $\dim(G_{2})=1$ (we then show that this assumption can be dropped in Step 1 of the proof of Proposition \ref{est1}). We consider $G_{\bullet}$ as fixed in this subsection. We assume that $G=\mathbb{R}^{s}\times\mathbb{R},\Gamma=\mathbb{Z}^{s}\times\mathbb{Z}$, and under the Mal'cev basis, the group action on $G$ is given by
\begin{equation}\nonumber
    \begin{split}
      (x_{1},\dots,x_{s}; y)\cdot(x'_{1},\dots,x'_{s}; y')=(x_{1}+x'_{1},\dots,x_{s}+x'_{s}; y+y'+\sum_{1\leq j<i\leq s}B_{i,j}x_{i}x'_{j})
    \end{split}
  \end{equation}
  for some $B_{i,j}\in\mathbb{R}, 1\leq j<i\leq s$. So $G_{2}=[G,G]=\{0\}^{s}\times\mathbb{R}$. Let $B^{G}$ denote the $s\times s$ matrix given by $B^{G}_{i,j}=B_{i,j}$
for $i>j$, $B^{G}_{i,j}=-B_{j,i}$ for $i<j$, and $B^{G}_{i,i}=0$ for all $1\leq i\leq s$.
Since $B^{G}$ is skew symmetric, by changing coordinates if necessary, we can always assume that
\begin{equation}\nonumber
B^{G}=\begin{vmatrix}
 B_{0}^{G} & \bold{0} \\
 \bold{0}  & \bold{0} \\
\end{vmatrix}
\end{equation}
for some invertible $s'\times s'$ integer screw symmetric matrix $B_{0}^{G}$, for some $0<s'\leq s$.

For any
 $\bold{x}=(x_{1},\dots,x_{s};x_{s+1})\in G$, we denote $\widehat{\bold{x}}=(x_{1},\dots,x_{s})$
 and $\bold{x}^{\bullet}=(x_{1},\dots,x_{s'})$.

Set $H_{0}=G\times G=\mathbb{R}^{s+1}\times\mathbb{R}^{s+1}, \Gamma_{0}=\Gamma\times\Gamma=\mathbb{Z}^{s+1}\times\mathbb{Z}^{s+1}$. Then $Y_{0}=H_{0}/\Gamma_{0}$ is a nilmanifold of order 2 endowed with the filtration $(H_{0})_{\bullet}=\{G_{i}\times G_{i}\}_{i=0}^{3}$. So $[H_{0},H_{0}]\cong G_{2}\times G_{2}$. Each subgroup $H$ of $H_{0}$ is endowed with the filtration induced by $(H_{0})_{\bullet}$,
   and every subgroup $H$ of $H_{0}$ is identified with a subspace $V(H)$ of $\mathbb{R}^{2s+2}$ under the Mal'cev basis.
\end{conv}

 The following lemma follows via direct computation:

\begin{lem}\label{dege}
  For any $H<H_{0}$, we have that
\begin{equation}\nonumber
    \begin{split}
      [H,H]=\Bigl\{\bigl((\bold{0};\widehat{\bold{x}}B^{G}\widehat{\bold{y}}^{T}),(\bold{0};\widehat{\bold{z}}B^{G}\widehat{\bold{w}}^{T})\bigl)\colon (\bold{x},\bold{z}),
  (\bold{y},\bold{w})\in H\Bigr\}.
    \end{split}
  \end{equation}
\end{lem}

By Lemma \ref{dege}, each subgroup $H$ of $H_{0}$ must be one of the following three types:
\begin{enumerate}
    \item $[H,H]=(\{0\}^{s}\times\mathbb{R})\times(\{0\}^{s}\times\mathbb{R})$. In this case, we say that $H$ is of \emph{type 1};
    \item $[H,H]=\Bigl\{\bigl((\bold{0};\lambda_{1}t),(\bold{0};\lambda_{2}t)\bigr)\colon t\in\mathbb{R}\Bigr\}$ for some $\lambda_{1},\lambda_{2}\in\mathbb{R}$
        not all equal to 0. In this case, we say that $H$ is of \emph{type 2} for the pair $(\lambda_{1},\lambda_{2})$;
     \item $[H,H]=\{\bold{0}\}$. In this case, we say that $H$ is of \emph{type 3}, which is equivalent to saying that $H$ is abelian.
   \end{enumerate}

\begin{defn}[Types of vectors]
 An element $\k=(k_{1},\dots,k_{2s+2})\in\mathbb{Z}^{2s+2}$ is said to be a vector of \emph{type 1} if $k_{2s+1}=k_{2s+2}=0$ and of \emph{type 2} if $k_{2s+2}=0$. Every
$\k\in\mathbb{Z}^{2s+2}$ is said to be a vector of \emph{type 3}. For any subgroup $H$ of $H_{0}$, $\k\in\mathbb{Z}^{2s+2}$ is said to be \emph{of the same type as} $H$
if both $\k$ and $H$ are of type $j$ for some $j=1,2,3$.
\end{defn}

For any $\l=(\lambda_{1},\lambda_{2})\in\mathbb{R}^{2}\backslash \{(0,0)\}$, we say that a subgroup $H$ of $H_{0}$ is \emph{$\l$-shaped} if $H$ is of type 1 or 3, or $H$ is of type 2 for the pair $(\lambda_{1},\lambda_{2})$.
 For any $\l\in\mathbb{R}^{2},\l\neq (0,0)$, if $\lambda_{1}\neq 0$, we define $\theta_{\l}\colon \mathbb{R}^{2s+2}\rightarrow H_{0}$ by
   \begin{equation}\nonumber
    \begin{split}
\theta_{\l}(\bold{x},\bold{y},w,z)=\Bigl((\bold{x};\lambda_{1}z),(\bold{y};\lambda_{2}z+w)\Bigr);
    \end{split}
  \end{equation}
  if $\lambda_{1}=0, \lambda_{2}\neq 0$, we define
     \begin{equation}\nonumber
    \begin{split}
\theta_{\l}(\bold{x},\bold{y},w,z)=\Bigl((\bold{x};w),(\bold{y};\lambda_{2}z)\Bigr).
    \end{split}
  \end{equation}
  This defines a homeomorphism $\theta_{\l}$ for all $\l\neq (0,0)$.
  It is easy to see that for $i=1,2,3$, if $H$ is a $\lambda$-shaped subgroup of $H_{0}$ of type $i$,
then $\theta_{\l}^{-1}([H,H])=\{0\}^{2s+i-1}\times\mathbb{R}^{3-i}$. From now on we always use coordinates on this new basis to denote elements
in $(\lambda_{1},\lambda_{2})$-shaped subgroups of $H_{0}$.

It is easy to see that if $\k\in\mathbb{Z}^{2s+2}$ is of the same type as $H$, then
$\c \cdot e_{Y}\rightarrow \k\cdot\theta_{\l}^{-1}\c$
is a well-defined horizontal character on $Y=H/(H\cap(\Gamma\times\Gamma))$.
For any element $\c\in H$, we write
     \begin{equation}\nonumber
    \begin{split}
     &\tilde{\c}=\theta_{\l}^{-1}\c=(\a,\b,w,z)\in \theta_{\l}^{-1}H;
     \\&\c'=(\a,\b)\in\mathbb{R}^{2s}.
    \end{split}
  \end{equation}
In other words, $\tilde{\c}$ is $\c$ written in the new coordinate system, and $\c'$ denotes the first $2s$ coordinates of $\c$.
  The definition of $\tilde{\c}$ depends on the choice of $\l$, and we clarify the dependency on $\l$ when we use this notation.
Notice that viewed as vectors in $\mathbb{R}^{2s+2}$, the first $2s$ coordinates of $\tilde{\c}$ are independent of the choice of $\lambda$ and are the same as that of $\c$, so this change of
variable does not change the Mal'cev coordinates when $H$ is of type 1.

For convenience, we define a special family of frequency and the equidistribution property among these frequencies:
\begin{defn}[$S$-totally equidistribution]
Let $Y=H/(H\cap\Gamma_{0})$ be a submanifold of $H_{0}$.
If $H$ is of type 1, an integer vector $\bold{k}\in\mathbb{Z}^{2}$ is called an \emph{$S$-frequency} of $H$ if $\bold{k}=(k,k)$ for some $k\neq 0$.
If $H$ is of type 2, an integer $k\in\mathbb{Z}$ is called an \emph{$S$-frequency} of $H$ if $k\neq 0$.
An \emph{$S$-frequency nilcharacter} of $H$ is a nilcharacter whose frequency is an $S$-frequency of $H$.

   Let $N\in\mathbb{N}$ and $\e>0$. A sequence $g\colon [N]\rightarrow Y$ is  \emph{$S$-totally $\e$-equidistrbuted} if for any $S$-frequency nilcharacter (of $H$) $f$ with $f\in Lip(Y)$, $\Vert f\Vert_{Lip(Y)}\leq 1$, and any arithmetic progression $P\subset[N]$, we have that
  \begin{equation}\nonumber
    \begin{split}
      \Bigl\vert\mathbb{E}_{n\in[N]}\bold{1}_{P}(n)f(g(n)\cdot e_{Y})\Bigr\vert<\e.
    \end{split}
  \end{equation}
\end{defn}

\begin{defn} [Good pair]
Given two subgroups $H',H$ of $H_{0}$. Suppose that $H'$ is a subgroup of $H$. We say that the pair $(H,H')$ is \emph{good} if the restriction to $H'$ of any $S$-frequency nilcharacter of $H$ is an $S$-frequency nilcharacter of $H'$.
\end{defn}

\begin{lem}\label{good}(Criteria for good pairs).
(1) If $\Phi$ is a non-trivial nilcharacter of $G$, then $\Psi=\Phi\otimes\overline{\Phi}$ is an $S$-frequency nilcharacter of $H_{0}=G\times G$;

(2) Let $H'<H$ be two subgroups of $H_{0}$. If $H$ and $H'$ are both of type 1 or both of type 2, then $(H,H')$ is good;

(3) Let $H$ be a type 1 subgroup of $H_{0}$, and $H'$ be a $(\lambda_{1},\lambda_{2})$-shaped type 2 subgroup of $H$. Then $(H,H')$ is good if and only if $\lambda_{1}\neq\lambda_{2}$.
\end{lem}
\begin{proof}(1) Suppose $\Phi$ is a nilcharacter of frequency $k\neq 0$ of $G$. Then $\Phi(g\cdot x)=e(k\cdot g)\Phi(x)$ for all $g\in G_{2}$ and $x\in G/(G\cap\Gamma)$. So
  \begin{equation}\nonumber
    \begin{split}
      &\qquad\Phi\otimes\overline{\Phi}((g,g')\cdot (x,x'))=\Phi(g\cdot x)\overline{\Phi}(g'\cdot x')
      \\&=e(k\cdot g-k\cdot g')\Phi(x)\overline{\Phi}(x')
=e\Bigl((k,-k)\cdot(g,g')\Bigr)\Phi\otimes\overline{\Phi}(x,x')
    \end{split}
  \end{equation}
for all $g,g'\in G_{2}$ and $x,x'\in G/(G\cap\Gamma)$. So $\Psi=\Phi\otimes\overline{\Phi}$ is a nilcharacter of frequency $(k,-k)$ of $G\times G$.

  (2) If $H$ and $H'$ are both of type 1 or both of type 2, then $[H,H]=[H',H']$. Suppose $f$ is a nilcharacter of frequency $\bold{k}$ on $H$. Then $f(g\cdot x)=e(\bold{k}\cdot g)f(x)$ for all $g\in[H,H]$ and $x\in H/(H\cap\Gamma)$, where the product $\bold{k}\cdot g$ is taken as the inner product on $[H,H]$. So $f(g\cdot x)=e(\bold{k}\cdot g)f(x)$ for all $x\in H'/(H'\cap\Gamma)\subset H/(H\cap\Gamma)$ and $g\in [H,H]=[H',H']$. So $f|_{H'}$ is also a nilcharacter of frequency $\bold{k}$, which finishes the proof.

  (3) Suppose $f$ is a nilcharacter of frequency $\bold{k}=(k,-k)$ on $H$ ($k\neq 0$). Then $f(g\cdot x)=e(\bold{k}\cdot g)f(x)=e(k(a-b))f(x)$ for all $g=(\bold{0},a;\bold{0},b)\in[H,H]$ and  $x\in H/(H\cap\Gamma)$. Since $H'$ is $(\lambda_{1},\lambda_{2})$-shaped, every $h\in [H',H']$ can be written as $h=(\bold{0},\lambda_{1} a;\bold{0},\lambda_{2}a)$. So
    \begin{equation}\nonumber
    \begin{split}
   f(h\cdot x)=e(\bold{k}\cdot h)f(x)=e(k(\lambda_{1}-\lambda_{2})a)f(x)
    \end{split}
  \end{equation}
  for all $x\in Y'=H'/(H'\cap\Gamma)$. So $f|_{Y'}$ is an $S$-frequency character on $Y'$ if and only if $\lambda_{1}\neq\lambda_{2}$, which finishes the proof.
\end{proof}

We need some definitions before we state the next lemma:

\begin{defn}[$d$-polynomial]
  Let $d\in\mathbb{N}$. We say that a polynomial on a nilpotent group $G$ is a \emph{$d$-polynomial} if it can be written as
$g(n)=\prod_{i=0}^{d}g_{i}^{\binom{n}{i}}$ for some $g_{0},\dots,g_{d-1}\in G$ and $g_{d}\in G_{2}$.
\end{defn}
\begin{rem}
 A $d$-polynomial is a polynomial of degree $d$ with respect to the filtration $G'_{\bullet}=\{G'_{i}\}_{i=0}^{d+1}$ where $G'_{0}=G'_{1}=\dots G'_{d-1}=G, G'_{d}=G_{2}, G'_{d+1}=\{e_{G}\}$.
 A 2-polynomial is just a polynomial of degree 2 with respect to the filtration $G_{\bullet}$.
\end{rem}

\begin{defn}\label{subgroup}\label{sg}
  Let $\l\in\mathbb{R}^{2}, \l\neq(0,0)$. Set $\tilde{\A}=\theta^{-1}_{\l}\A$.
  For any $\k\in\mathbb{Z}^{2s+2}$, denote $\underline{\k}=\frac{1}{p}\k$, where $p$ is the greatest positive common divisor of entries of $\k$.
  For any subgroup $H<H_{0}$, denote $H^{\underline{\k}}=\{g\in H: \underline{\k}\cdot \tilde{g}=0\}$.
\end{defn}

\begin{rem}
 Note that $H^{\underline{\k}}$ depends on the choice of $\l$. We clarify the dependency on $\l$ when we use this notation.
\end{rem}

The following lemma allows us to approximate a polynomial sequence lacking equidistribution with a polynomial sequence on a subgroup:

\begin{lem}\label{appro2}(Approximation property).
Let $\l\in\mathbb{R}^{2}, \l\neq(0,0)$. Set $\tilde{\A}=\theta^{-1}_{\l}\A$.
Let $H$ be a subgroup of $H_{0}$. For every $C,D,\D>0, d\in\mathbb{N}$, there exist $C'=C'(X,C,\D,d), \D'=\D'(X,C,\D,d), D'=D'(X,C,D,\D,d)>0, N_{0}=N_{0}(X,C,\D,d)\in\mathbb{N}$ with the following property: Let $N\geq N_{0}$. Suppose that $g_{1}(n)=\prod_{i=0}^{d}\A_{i}^{\binom{n}{i}}, \A_{0},\dots,\A_{d-1}\in H, \A_{d}\in H_{2}$ is a $d$-polynomial which is not $S$-totally $\D$-equidistributed on $H$. Suppose also that
  $\k=p\underline{\k}\in\mathbb{Z}^{2s+2}$ is a vector of the same type as $H$ such that
   $H$ and $H^{\underline{\k}}$ are $\l$-shaped,
   and that
   $0<\Vert\k\Vert\leq C, \Vert\k\cdot\tilde{\A_{i}}\Vert_{\mathbb{R}/\mathbb{Z}}\leq\frac{C}{N^{i}}$ for all $1\leq i\leq d-1$.

  Suppose that both $H$ and $H^{\underline{\k}}$ are $\l$-shaped.
  Then there exists a polynomial sequence of the form $g_{2}(n)=\prod_{i=0}^{d}\B_{i}^{\binom{n}{i}}$, $\B_{0},\dots,\B_{d}\in H^{\underline{\k}}$ such that

  1) If the pair $(H,H^{\underline{\k}})$ is good, then $g_{2}$ is not $S$-totally $\D'$-equidistributed on $H^{\underline{\k}}$;

2) If there exists $\p\in\mathbb{Z}^{2s+2}$ of the same type as $H$ (and thus of $H^{\underline{\k}}$) such that $0<\Vert\p\Vert\leq D$ and $\Vert\p\cdot\tilde{\B_{1}}\Vert_{\mathbb{R}/\mathbb{Z}}\leq\frac{D}{N}$,
then there exists some $w\in\mathbb{Z}, 0<w\leq C^{2}$ such that $\Vert w\p\cdot\tilde{\A_{1}}\Vert_{\mathbb{R}/\mathbb{Z}}\leq\frac{D'}{N}$.

3) If $H_{2}=H_{2}^{\underline{\k}}$, then $\B_{d}\in H_{2}^{\underline{\k}}$.
\end{lem}
Roughly speaking, this lemma says that if all the coefficients of a polynomial sequence on $H$ are ``close to" a sub-manifold $H'$, then this polynomial sequence can be approximated by a polynomial sequence whose coefficients are exactly contained in $H'$, and both sequences have a similar equidistribution property if the pair $(H,H')$ is good.

\begin{rem}
 It is worth noting that 3) implies that $g_{2}$ is also a $d$-polynomial when $H_{2}=H_{2}^{\underline{\k}}$, but it could be a $(d+1)$-polynomial if this condition does not hold.
\end{rem}
\begin{proof} Since $\k$ is of the same type as $H$, it is easy to see that $H_{2}\subset H^{\underline{\k}}$.

In the proof, $C_{1},C_{2}$ are constants depending only on $X,C,\D$ and $d$. 
  We assume without loss of generality that $g_{1}(n)$ has no constant term, i.e. we assume $g_{1}(n)=\prod_{i=1}^{d}\A_{i}^{\binom{n}{i}}$.
  Let $F_{0}\subset H$ be a bounded fundamental domain of the projection $H\rightarrow Y$ (we assume $F_{0}$ is fixed given $Y$). Then by Lemma \ref{cts}, there exists $C_{1}>0$ such that
   \begin{equation}\nonumber
   \begin{split}
     d_{Y}(h\cdot y,h\cdot z)\leq C_{1}d_{Y}(y,z)
   \end{split}
 \end{equation}
for all $h\in F_{2}$ and $y,z\in Y$.

Since $\k=p\underline{\k}$,
 we have that $\Vert\underline{\k}\cdot\tilde{\A_{i}^{p}}\Vert_{\mathbb{R}/\mathbb{Z}}
 =\Vert\k\cdot\tilde{\A_{i}}\Vert_{\mathbb{R}/\mathbb{Z}}\leq\frac{C}{N^{i}}$ for $1\leq i\leq d-1$. So there exist $C_{2}>0$ and $\omega_{1},\dots,\omega_{d-1}\in H$ such that for all $1\leq i\leq d-1$, $\underline{\k}\cdot\tilde{(\omega_{i}^{-1}\A_{i}^{p})}\in\mathbb{Z}$ and
 \begin{equation}\nonumber
   \begin{split}
     d_{H}(\omega_{i},id_{H})\leq\frac{C_{2}}{N^{i}}.
   \end{split}
 \end{equation}
 Since the entries of $\underline{\k}$ are relatively prime integers, we deduce that there exist $\gamma_{1},\dots,\gamma_{d-1}\in\Gamma\times\Gamma$ such that
 $\underline{\k}\cdot\tilde{(\omega_{i}^{-1}\A_{i}^{p}\gamma_{i}^{-1})}=0$ for all $1\leq i\leq d-1$. In other words,
 \begin{equation}\nonumber
    \begin{split}
     \A'_{i}\coloneqq \omega_{i}^{-1}\A_{i}^{p}\gamma_{i}^{-1}\in H^{\underline{\k}}, \text{ for all } 1\leq i\leq d-1.
    \end{split}
  \end{equation}
  By assumption, there exists an $S$-frequency nilcharacter $f$ with $\Vert f\Vert_{Lip(H)}\leq 1$ such that
   \begin{equation}\label{temp41}
    \begin{split}
     \Bigl\vert\mathbb{E}_{n\in[N]}\bold{1}_{P}(n)f(g_{1}(n)\cdot e_{Y})\Bigr\vert>\D
    \end{split}
  \end{equation}
  for some arithmetic progression $P\subset[N]$.
We define
 \begin{equation}\nonumber
    \begin{split}
     L\coloneqq \lfloor\frac{N}{10K\cdot d!}\min\{\frac{\D}{C_{2}},\frac{1}{C}\}\rfloor
    \end{split}
  \end{equation}
  and assume that $N$ is chosen to be sufficiently large such that $L\geq 1$. Here $K>1/5$ is some universal constant depending only on $d$ to be specified latter. Since $K>1/5>p/5C$, we have that $N\geq 2pL\cdot d!$, we can make a partition of the interval $[N]$ into arithmetic
 progressions of step $p\cdot d!$ and length between $L$ and $2L$. Then we deduce from (\ref{temp41}) that there exist $k_{0}\in[N]$ and an arithmetic
progression $P'\subset[N]$ such that
   \begin{equation}\label{temp71}
    \begin{split}
     \Bigl\vert\mathbb{E}_{n\in[N]}\bold{1}_{P'}(n)f(g_{1}(p\cdot d!n+k_{0})\cdot e_{Y})\Bigr\vert>\D.
    \end{split}
  \end{equation}
  Since $g_{1}(n)=\prod_{i=0}^{d}\A_{i}^{\binom{n}{i}}$, it is easy to verify that there exist elements $h_{0}\in H, v_{1},\dots,v_{d}\in H_{2}\subset H^{\underline{\k}}$, a universal constant $K>0$ depending only on $d$, and polynomials $p_{1},\dots,p_{d-1}$, $q_{1},\dots,q_{d-1}$
  satisfying $\vert p_{i}(n)\vert\leq Kn^{i}$ for all $1\leq i\leq d-1$
such that
     \begin{equation}\label{temp51}
    \begin{split}
     g_{1}(p\cdot d!n+k_{0})=\Bigl(\prod_{i=1}^{d-1}\omega_{i}^{p_{i}(n)}\Bigr)\cdot h_{0}\cdot\Bigl(\prod_{i=1}^{d-1}(\A'_{i}v_{i})^{\binom{n}{i}}\Bigr)\cdot v_{d}^{\binom{n}{d}}\cdot\Bigl(\prod_{i=1}^{d-1}\gamma_{i}^{q_{i}(n)}\Bigr)
    \end{split}
  \end{equation}
  for every $n\in\mathbb{N}$.
  We now pick the constant $K$ that satisfies the above condition and the condition $K>\frac{1}{5}$.
   Choose $h'_{0}\in F_{0}$ and $\lambda\in\Gamma\times\Gamma$ such that $h_{0}=h'_{0}\lambda$. Define
    \begin{equation}\label{temp61}
    \begin{split}
     & \B_{i}=\lambda\A'_{i} v_{i}\lambda^{-1}, 1\leq i\leq d-1;
     \\& \B_{d}=\lambda v_{d}\lambda^{-1}, \B_{d+1}=id_{H^{\underline{\k}}};
     \\& g_{2}(n)=\prod_{i=0}^{d}\B_{i}^{\binom{n}{i}}, \forall n\in[N];
     \\& F(y)=f(h_{0}'\cdot y), \forall y\in Y.
    \end{split}
  \end{equation}
  We remark that for $1\leq i\leq d-1$, $\B_{i}$ belong to $H^{\underline{\k}}$ because
  $\B_{i}=[\lambda,\A'_{i}]v_{i}\A'_{i}$ and  $[\lambda,\A'_{i}]v_{i}\in H_{2}\subset H^{\underline{\k}}$. Also $\B_{d}=[\lambda,v_{d}]v_{d}\in H_{2}\subset H^{\underline{\k}}$.
Therefore, $g_{2}(n)$ is a polynomial sequence in $H^{\underline{\k}}$ having the form as stated.
Note that in the special case when $H_{2}=H_{2}^{\underline{\k}}$, we have that $\B_{i}\in H_{2}^{\underline{\k}}$ for all $1\leq i\leq d$. So Property 3) is satisfied.

By (\ref{temp51}) and (\ref{temp61}), we have that
       \begin{equation}\nonumber
    \begin{split}
     g_{1}(p\cdot d!n+k_{0})\cdot e_{Y}=\Bigl(\prod_{i=1}^{d-1}\omega_{i}^{p_{i}(n)}\Bigr)h_{0}'g_{2}(n)\cdot e_{Y}.
    \end{split}
  \end{equation}
  For $n\leq 2L$, by the right invariance of the metric $d_{H}$, we have that
       \begin{equation}\nonumber
    \begin{split}
     d_{H}(\prod_{i=1}^{d-1}\omega_{i}^{p_{i}(n)},id_{H})\leq \sum_{i=1}^{d-1}KC_{2}L^{i}/N^{i}\leq dKC_{2}L/N\leq\D/2.
    \end{split}
  \end{equation}
So
  \begin{equation}\nonumber
    \begin{split}
     d_{Y}\Bigl(g_{1}(p \cdot d!n+k_{0})\cdot e_{Y},h_{0}'g_{2}(n)\cdot e_{Y}\Bigr)\leq\D/2
    \end{split}
  \end{equation}
  Since $\Vert f\Vert_{Lip(X)}\leq 1$, it follows from (\ref{temp71}) that
     \begin{equation}\nonumber
    \begin{split}
     &\quad\Bigl\vert\mathbb{E}_{n\in[N]}\bold{1}_{P'}(n)F(g_{2}(n)\cdot e_{Y})\Bigr\vert
=\Bigl\vert\mathbb{E}_{n\in[N]}\bold{1}_{P'}(n)f(h_{0}'g_{2}(n)\cdot e_{Y})\Bigr\vert
\\&\geq \Bigl\vert\mathbb{E}_{n\in[N]}\bold{1}_{P'}(n)f(g_{1}(p \cdot d!n+k_{0})\cdot e_{Y})\Bigr\vert-\D/2>\D/2,
    \end{split}
  \end{equation}
  and $F$ has a bounded Lipschitz constant since $h_{0}'\in F_{0}$. Moreover, by assumption, $f|_{H^{\underline{\k}}}$ is of $S$-frequency and thus so is $F|_{H^{\underline{\k}}}$, which proves
  Property 1).

   We are left with proving Property 2). Note that $\B_{1}=v'\A'_{1}$ for some $v'\in H_{2}$, which implies that $\bold{p}\cdot\tilde{v'}=0$ since $\bold{p}$ is of the same type as $H$. Therefore,
  \begin{equation}\nonumber
    \begin{split}
    \Vert p\bold{p}\cdot\tilde{\A_{1}}\Vert_{\mathbb{R}/\mathbb{Z}}=
  \Vert\bold{p}\cdot\tilde{(\omega_{1}\B_{1})}\Vert_{\mathbb{R}/\mathbb{Z}}
  \leq
  \Vert\bold{p}\cdot\tilde{\omega_{1}}\Vert_{\mathbb{R}/\mathbb{Z}}
  +\Vert\bold{p}\cdot\tilde{\B_{1}}\Vert_{\mathbb{R}/\mathbb{Z}}
  \leq\frac{(C_{2}+1)D}{N}.
      \end{split}
  \end{equation}
\end{proof}

For convenience, we also need the following definitions:
\begin{defn}[Height]
  The \emph{height} of a rational number $\frac{p}{q} (p,q\in\mathbb{Z},(p,q)=1)$ is $\max\{\vert p\vert,\vert q\vert\}$. We denote the height of an irrational number to be $\infty$.

  The \emph{height} of a matrix is the maximum of the heights of entries of $M$.

  If $A$ is a subspace of $\mathbb{R}^{n}$ with dimension $r$, let $\mathcal{C}(A)$ denote the collection of all $n\times r$ matrices $B$ such that $A$ can
  be written as $A=\{B\vec{t}\in\mathbb{R}^{n}\colon\vec{t}\in\mathbb{R}^{r}\}$. The \emph{height} of $A$ is the minimum of the heights of $B\in\mathcal{C}(A)$. Since the integers are discrete, the height of $A$ is always well-defined.
 \end{defn}

  \begin{defn}[Non-trivial vector]
 Let $A$ be a subspace and $\bold{v}$ be a vector of $\mathbb{R}^{n}$. We say that $\bold{v}$ is \emph{non-trivial} with respect to $A$ if $A$ is not contained in the orthogonal complement of $\bold{v}$. In particular, $\bold{v}$ is non-trivial with respect to $\mathbb{R}^{n}$ if and only if $\bold{v}\neq \bold{0}$.
\end{defn}

\begin{defn}[Group of automorphisms]
    Let $n\in\mathbb{N}$ and $M_{n\times n}$ be the collection of all $n\times n$ real matrices.
For any $A\in M_{n\times n}$, write $\Lambda(A)=\{P\in M_{n\times n}\colon P^{T}AP=A\}$.
\end{defn}

  The key to the proof of Proposition \ref{est1} is to extract as much information as possible from the non-equidistribution of a sequence on $H_{0}$. We start with the following proposition:

\begin{prop}\label{last}(Non-equidistribution on $H_{0}$ implies non-equidistribution on $G$). Suppose $G$ is not abelian.
  Then for any $\D>0$, there exist $D=D(X,\D)>0$ and $N_{0}=N_{0}(X,\D)\in\mathbb{N}$ such that for any $N\geq N_{0}$, any $\a=(a_{1},\dots,a_{s};a_{s+1}),
\b=(b_{1},\dots,b_{s};b_{s+1})\in G,$ and any $a',b'\in G_{2}$,
 if there exist a non-trivial nilcharacter $\Phi$ such that $\Vert\Phi\Vert_{Lip(X)}\leq 1$ and an arithmetic progression $P\subset[N]$ such that
  \begin{equation}\nonumber
    \begin{split}
      \Bigl\vert\mathbb{E}_{n\in[N]}\bold{1}_{P}(n)\Phi(\a^{n}a'^{\binom{n}{2}}\cdot e_{X})\overline{\Phi(\b^{n}b'^{\binom{n}{2}}\cdot e_{X})}\Bigr\vert>\D,
    \end{split}
  \end{equation}
  then one of the following is true:

%  (i) there exists a horizontal character $\eta$ of $X$ such that $0<\Vert\eta\Vert\leq D$ and %$\Vert\eta\circ g\Vert_{C^{\infty}[N]}\leq D$, where $g(n)=(\a\b^{-1})^{n}$;

  (i) there exist horizontal characters $\eta_{1}$ and $\eta_{2}$ of $X$ such that $0<\Vert\eta_{1}\Vert, \Vert\eta_{2}\Vert\leq D$ and $\Vert\eta_{1}\circ g_{1}\Vert_{C^{\infty}[N]},\Vert\eta_{2}\circ g_{2}\Vert_{C^{\infty}[N]}\leq D$,
where $g_{1}(n)=\a^{n}, g_{2}(n)=\b^{n}$;

  (ii) there exists an $s'\times s'$ matrix $M\in \Lambda(B_{0}^{G})$ of height at most $D$, such that $\a^{\bullet}B_{0}^{G}M$ is at most $\frac{D}{N}$-away from $\b^{\bullet}B_{0}^{G}$ (recall that $B_{0}^{G}$,$s'$, $\a^{\bullet}$ and $\b^{\bullet}$ are defined in Convention \ref{conv2}).
\end{prop}
\begin{rem}
  It is worth noting that every matrix $M\in \Lambda(B_{0}^{G})$ induces a natural automorphism $\sigma_{M}\colon G\rightarrow G$ given by
    \begin{equation}\nonumber
    \begin{split}
      \sigma_{M}(a_{1},\dots,a_{s};a_{s+1})=\bigl((a_{1},\dots,a_{s'})M,a_{s'+1},\dots,a_{s};a_{s+1}\bigr), \forall (a_{1},\dots,a_{s};a_{s+1})\in G
    \end{split}
  \end{equation}
  such that $[\a,\b]=[\sigma_{M}\a,\sigma_{M}\b]$ for all $\a,\b\in G$.
\end{rem}
\begin{proof}
Throughout the proof, all the numbers $\D_{1},\D_{2},\dots,D_{1},D_{2},\dots,F_{1},F_{2},\dots$ depend only on $X$ and $\D$, and $D_{1}',D_{2}',\dots\colon
\mathbb{R}\rightarrow\mathbb{R}$ are functions depending only on $X$ and $\D$. $N$ is always assumed to be large enough depending only on the above constants, $X$, 
and $\D$ (and therefore depending only on $X$ and $\D$).  

%If the frequency of $\Phi$ is 0, then the $\Phi$ can be viewed as a function on the abelian group %$G/G_{2}$. So the problem is reduced to the abelian case and thus claim (i) holds. Now we %assume that $\Phi$ is a nilcharacter of non-zero frequency.
By Lemma \ref{good}, $\Psi=\Phi\otimes\overline{\Phi}$ is an $S$-frequency nilcharacter on $H_{0}=G\times G$. Since the boundedness of the Lipschitz norm
of $\Phi$ implies the boundedness of the Lipschitz norm of $\Psi$, by ignoring a scale depending only on $X$, we may assume without loss of generality that the Lipschitz norm of $\Psi$ is less than 1.
 Then the sequence $h_{0}(n)=\Bigl(\a^{n}a'^{\binom{n}{2}},\b^{n}b'^{\binom{n}{2}}\Bigr)$ is not $S$-totally $\D$-equidistributed on $H_{0}$.

Denote $\a_{0}=\a, \b_{0}=\b, a'_{0}=a', b'_{0}=b',a_{0}''=b_{0}''=e_{G},v_{0}=e_{H}$.
For convenience, throughout the proof, we write $\c_{j}=(\a_{j},\b_{j}), c'_{j}=(a'_{j},b'_{j}),c''_{j}=(a''_{j},b''_{j})$ for all $j\in\mathbb{N}$ whenever the notation $\a_{j},\b_{j},a'_{j},b'_{j},a''_{j},b''_{j}$ is introduced.

We prove the proposition by induction. For any $i\geq 0$, we say that \emph{hypothesis $P(i)$} is satisfied if:
 for all $0\leq j\leq i$, there exist a subgroup $H_{j}$ of $H_{0}$, $\Gamma_{j}=\Gamma_{0}\cap H_{j}=(\Gamma\times\Gamma)\cap H_{j}, Y_{j}=H_{j}/\Gamma_{j}$, and a 3-polynomial
$h_{j}(n)=v_{j}\Bigl(\a_{j}^{n}a_{j}'^{\binom{n}{2}}a_{j}''^{\binom{n}{3}},\b_{j}^{n}b_{j}'^{\binom{n}{2}}b_{j}''^{\binom{n}{3}}\Bigr)$ on $H_{j}$ ($a_{j}'',b_{j}''\in G_{2}$)
such that for all $0\leq j\leq i$, we have that

$(I)$ $h_{j}(n)$ is not $S$-totally $\D_{j}$-equidistributed on $H_{j}$;

$(II)$ If $H_{j}$ is of type 1, then $h_{j}$ is a 2-polynomial;

$(III)$ If $j\geq 1$, then $H_{j}=H_{j-1}^{\underline{\v_{j}}}$ for some $\v_{j}$ non-trivial with respect of $H_{j-1}$ (recall Definition \ref{sg}) and of the same type as $H_{j-1}$. Therefore, for all $0\leq j\leq i$, if we let $V_{j}$ denote the subspace of $\mathbb{R}^{2s+2}$ of $H_{j}$ under the Mal'cev basis, then $V_{j}$ is of co-dimension $j$. Moreover, there exists $\l\in\mathbb{R}^{2}, \l\neq (0,0)$ such that each $H_{j}$ is $\l$-shaped;

$(IV)$ The coefficients of the pair $(h_{j-1},h_{j})$ satisfy Property 2) of Lemma \ref{appro2} , for some function $D'_{j}$ and integer $w_{j}$ bounded by
some constant depending only on $X$ and $\D$.

\

We first note that hypothesis $P(0)$ is satisfied. Suppose that hypothesis $P(i)$ is satisfied for some $i\in\mathbb{N}$, we wish to prove that hypothesis $P(i+1)$ is satisfied. Notice that the element $\l$ appearing in $(III)$ is fixed throughout the induction step,
and so we always use the notation $\tilde{\A}$ to denote $\theta^{-1}_{\l}\A$ for this fixed $\l=(\lambda_{1},\lambda_{2})$ in the proof.

Suppose that $H_{i}$ is not abelian. We distinguish cases depending on the type of $H_{i}$:

\textbf{Case that $H_{i}$ is of type 1.} By hypothesis $(II)$, $h_{i}$ is a 2-polynomial. Then
using Theorem \ref{Lei} on $H_{i}$,
there exist $D_{i+1}>0$ and a vector $\v_{i+1}\in\mathbb{Z}^{2s+2}$ non-trivial with respect to $V_{i}$ of the same type as $H_{i}$ with length no larger than $D_{i+1}$ such that
\begin{equation}\nonumber
    \begin{split}
     \Vert \v_{i+1}\cdot\tilde{\c_{i}}\Vert_{\mathbb{R}/\mathbb{Z}}\leq\frac{D_{i+1}}{N}.
    \end{split}
  \end{equation}
  Let $H_{i+1}=H^{\underline{\v_{i+1}}}$ (recall Definition \ref{subgroup}). Then $V(H_{i+1})=V_{i+1}$ is of co-dimension $i+1$ since $\v_{i+1}$ is non-trivial with respect to $V_{i}$. Let
 $\Gamma_{i+1}=\Gamma_{0}\cap H_{i+1}, Y_{i+1}=H_{i+1}/\Gamma_{i+1}$.

 If $H_{i+1}$ is of type 1, then $(H_{i})_{2}=(H_{i+1})_{2}$. By Lemma \ref{appro2}, there exists
 a polynomial of the form $h_{i+1}(n)=v_{i+1}\Bigl(\a_{i+1}^{n}a_{i+1}'^{\binom{n}{2}},\b_{i+1}^{n}b_{i+1}'^{\binom{n}{2}}\Bigr)$
on $H_{i+1}$ ($a_{i+1}'(n),b_{i+1}'(n)\in G_{2}$) such that it is not $S$-totally $\D_{i+1}$-equidistributed on $H_{i+1}$ and the coefficients of the pair
$(h_{i},h_{i+1})$ satisfy Property 2) of Lemma \ref{appro2}
for some function $D'_{i+1}$ and integer $w_{i+1}$ bounded by some constant depending only on $X$ and $\D$. Thus $(IV)$ is
satisfied for the system $Y_{i+1}$ and the sequence $h_{i+1}$. By Property 1) of Lemma \ref{appro2}, $(I)$ is true since the pair $(H_{i},H_{i+1})$ is good by
 Lemma \ref{good}. $(II)$ also holds since $h_{i+1}$ is a 2-polynomial.

If $H_{i+1}$ is of type 2 and the pair $(H_{i},H_{i+1})$ is good, by Lemma \ref{appro2}, there exists
 a 3-polynomial $h_{i+1}(n)=v_{i+1}\Bigl(\a_{i+1}^{n}a_{i+1}'^{\binom{n}{2}}a_{i+1}''^{\binom{n}{3}},\b_{i+1}^{n}b_{i+1}'^{\binom{n}{2}}b_{i+1}''^{\binom{n}{3}}\Bigr)$
on $H_{i+1}$ ($a_{i+1}''(n),b_{i+1}''(n)\in G_{2}$) such that it is not $S$-totally $\D_{i+1}$-equidistributed on $H_{i+1}$ and the coefficients of the pair $(h_{i},h_{i+1})$ satisfy Property 2) of Lemma \ref{appro2}
for some function $D'_{i+1}$ and integer $w_{i+1}$ bounded by some constant depending only on $X$ and $\D$.  Thus $(IV)$ is
satisfied for the system $Y_{i+1}$ and the sequence $h_{i+1}$. $(I)$ is true by assumption. Condition $(II)$ is trivial since $H_{i+1}$ is of type 2.

If $H_{i+1}$ is of type 2 and the pair $(H_{i},H_{i+1})$ is not good, then we stop the induction procedure.

If $H_{i+1}$ is abelian, then we stop the induction procedure.

\textbf{Case that $H_{i}$ is of type 2.} This time $h_{i}$ is a 3-polynomial and $H_{i}$ is $\l$-shaped. Using Theorem \ref{Lei} on $H_{i}$,
there exist $D_{i+1}>0$ and a vector $\v_{i+1}\in\mathbb{Z}^{2s+2}$ non-trivial with respect to $V_{i}$ of the same type as $H_{i}$ with length no larger than $D_{i+1}$ such that
\begin{equation}\nonumber
    \begin{split}
     \Vert \v_{i+1}\cdot\tilde{\c_{i}}\Vert_{\mathbb{R}/\mathbb{Z}}\leq\frac{D_{i+1}}{N}, \Vert \v_{i+1}\cdot \tilde{c'_{i}}\Vert_{\mathbb{R}/\mathbb{Z}}\leq\frac{D_{i+1}}{N^{2}}.
    \end{split}
  \end{equation}
  Let $H_{i+1}=H^{\underline{V_{i+1}}}$. Then $V_{i+1}$ is of codimension $i+1$ since $\v_{i+1}$ is non-trivial with respect to $V_{i}$. Let
 $\Gamma_{i+1}=\Gamma_{0}\cap H_{i+1}, Y_{i+1}=H_{i+1}/\Gamma_{i+1}$.

 Since $H_{i}$ is of type 2, we have that $H_{i+1}$ is either of type 2 or abelian. If $H_{i+1}$ is of type 2, then $(H_{i})_{2}=(H_{i+1})_{2}$. By
 Lemma \ref{appro2}, there exists
 a 3-polynomial of the form $h_{i+1}(n)=v_{i+1}\Bigl(\a_{i+1}^{n}a_{i+1}'^{\binom{n}{2}}a_{i+1}''^{\binom{n}{3}},\b_{i+1}^{n}b_{i+1}'^{\binom{n}{2}}b_{i+1}''^{\binom{n}{3}}\Bigr)$
on $H_{i+1}$ ($a_{i+1}''(n),b_{i+1}''(n)\in G_{2}$)  such that it is not $S$-totally $\D_{i+1}$-equidistributed on $H_{i+1}$ and the coefficients of the pair
$(h_{i},h_{i+1})$ satisfy Property 2) of Lemma \ref{appro2}
for some function $D'_{i+1}$ and integer $w_{i+1}$ bounded by some constant depending only on $X$ and $\D$. Thus $(IV)$ is
satisfied for the system $Y_{i+1}$ and the sequence $h_{i+1}$. By Property 1) of Lemma \ref{appro2}, $(I)$ is true since the pair $(H_{i},H_{i+1})$ is good by
Lemma \ref{good}. Condition $(II)$ is trivial since $H_{i+1}$ is of type 2.

If $H_{i+1}$ is abelian, then we stop the induction procedure.

\

In conclusion, the above procedure can be continued and thus hypothesis $P(i)$ is satisfied unless either $H_{i+1}$ appearing in the above construction is abelian or the pair $(H_{i},H_{i+1})$ is not good. Since the dimension of $H_{0}$ is finite, the above procedure must stop within $2s+2$ steps.

\

If $H_{i+1}$ is abelian, by a similar discussion,
we can still construct the system $Y_{i+1}$ and the sequence $h_{i+1}$ (possibly with different degree) such that $(II),(III)$ and $(IV)$ hold (but $(I)$ may not hold). Denoting $r=2s-i+1, H_{i+1}$ can be written as
\begin{equation}\nonumber
 \begin{split}
  \theta_{\l}^{-1}H_{i+1}=\Bigl\{ (\bold{t}A,\bold{t}A';\bold{t}\xi,\bold{t}\xi')\colon\bold{t}\in\mathbb{R}^{r}\Bigr\},
 \end{split}
\end{equation}
where $A,A',\xi,\xi'$ are respectively $r\times s, r\times s, r\times 1, r\times 1$ matrices with height at most $F_{1}$.
By Lemma \ref{dege}, we have that (recall that $B^{G}$ is defined in Convention \ref{conv2})
\begin{equation}\nonumber
 \begin{split}
  [H_{i+1},H_{i+1}]=\Bigl\{ (\bold{0},\bold{s}AB^{G}A^{T}\bold{t}^{T};\bold{0},\bold{s}A'B^{G}A'^{T}\bold{t}^{T})\colon\bold{s},\bold{t}\in\mathbb{R}^{r}\Bigr\}.
 \end{split}
\end{equation}
So we deduce that
\begin{equation}\label{ABA2}
 \begin{split}
  AB^{G}A^{T}=A'B^{G}A'^{T}=0.
 \end{split}
\end{equation}

First suppose that $rank(A)<s$. Since $\tilde{\c_{i+1}}=\tilde{(\a_{i+1},\b_{i+1})}=(\bold{t}A,\bold{t}A';\bold{t}\xi,\bold{t}\xi')$ for some $\bold{t}\in\mathbb{R}^{r}$, we deduce that there exist $F_{3}>0$ and $\p=(\p_{0},\bold{0})\in\mathbb{Z}^{s}\times\{0\}^{s+2}$ with $0<\Vert\p\Vert\leq F_{3}$ such that
$\Vert\p\cdot\tilde{\c_{i+1}}\Vert_{\mathbb{R}/\mathbb{Z}}=\Vert\p_{0}\cdot\widehat{\a_{i+1}}\Vert_{\mathbb{R}/\mathbb{Z}}\leq\frac{F_{3}}{N}$ (recall that
$\widehat{\a}$ is the first $s$ coordinates of $\a$ by Convention \ref{conv2}).
Since the coefficients of each pair $(h_{j},h_{j+1}) (0\leq j\leq i)$ satisfy Property 2) of Lemma \ref{appro2} and $\p$ is of type 1 (and thus is of the same type
as all subgroups of $H_{0}$), there exists $w\in\mathbb{Z}$ with $0<\vert w\vert\leq F_{4}$ such that
$\Vert w\p\cdot\tilde{\c_{0}}\Vert_{\mathbb{R}/\mathbb{Z}}=\Vert w\p_{0}\cdot\widehat{\a_{0}}\Vert_{\mathbb{R}/\mathbb{Z}}\leq\frac{F_{4}}{N}$,
where $F_{4}=D'_{1}(F_{1}\cdot D'_{2}(\dots F_{1}\cdot D'_{i}(0)\dots))$. So the existence of $\eta_{1}$ in (i) is proved.
The existence of $\eta_{2}$ in (i) holds similarly and therefore conclusion (i) holds.

If $rank(A)= s$, by a simple computation using linear algebra, (\ref{ABA2}) implies that $B^{G}=0$. So $G$ is abelian, a contradiction.

\

If the pair $(H_{i},H_{i+1})$ is not good, by Lemma \ref{good}, we must have that $H_{i}$ is of type 1 and $H_{i+1}$ is of type 2,
and $H_{i+1}$ must be $(1,1)$-shaped. By a similar discussion,
we can still construct the system $Y_{i+1}$ and the sequence $h_{i+1}$ such that $(II),(III)$ and $(IV)$ hold (but $(I)$ may not hold).
Set $r=2s-i+1$. Since in this case $\v_{j}$ is of type 1 for all $0\leq j\leq i$, $H_{i+1}$ can be written as
\begin{equation}\nonumber
 \begin{split}
  H_{i+1}=\Bigl\{ (\bold{t}A,w;\bold{t}A',w')\colon\bold{t}\in\mathbb{R}^{r}, w,w'\in\mathbb{R}\Bigr\},
 \end{split}
\end{equation}
where $A$ and $A'$ are some full rank $r\times s$ matrices with height at most $F_{1}$. By Lemma \ref{dege}, we have that
\begin{equation}\nonumber
 \begin{split}
  [H_{i+1},H_{i+1}]=\Bigl\{ (\bold{0},\bold{s}AB^{G}A^{T}\bold{t}^{T};\bold{0},\bold{s}A'B^{G}A'^{T}\bold{t}^{T})\colon\bold{s},\bold{t}\in\mathbb{R}^{r}\Bigr\}.
 \end{split}
\end{equation}
So we deduce that
\begin{equation}\label{ABA}
 \begin{split}
  AB^{G}A^{T}=A'B^{G}A'^{T}.
 \end{split}
\end{equation}

If $rank(A)<s$, similar to the previous case, there exist $F_{3}>0$ and $\p=(\p_{0},\bold{0})\in\mathbb{Z}^{s}\times\{0\}^{s+2}$ with $0<\Vert\p\Vert\leq F_{3}$ such that
$\Vert\p\cdot\tilde{\c_{i+1}}\Vert_{\mathbb{R}/\mathbb{Z}}=\Vert\p_{0}\cdot\widehat{\a_{i+1}}\Vert_{\mathbb{R}/\mathbb{Z}}\leq\frac{F_{3}}{N}$.
Since the coefficients of each pair $(h_{j},h_{j+1}) (0\leq j\leq i)$ satisfy Property 2) of Lemma \ref{appro2} and $\p$ is of type 1 (and thus is of the same type
as all subgroups of $H_{0}$), there exists
$w\in\mathbb{Z}$ with $0<\vert w\vert\leq F_{4}$ such that
$\Vert w\p\cdot\tilde{\c_{0}}\Vert_{\mathbb{R}/\mathbb{Z}}=\Vert w\p_{0}\cdot\widehat{\a_{0}}\Vert_{\mathbb{R}/\mathbb{Z}}\leq\frac{F_{4}}{N}$,
where $F_{4}=D'_{1}(F_{1}\cdot D'_{2}(\dots F_{1}\cdot D'_{i}(0)\dots))$. So the existence of $\eta_{1}$ in (i) is proved.
The existence of $\eta_{2}$ in (i) holds similarly and therefore conclustion (i) holds.

Now we assume that $rank(A)=s$. In this case $r\geq s$. Suppose that $A'=Y\binom{I_{s\times s}}{0_{(r-s)\times s}}$ for some
invertible $r\times r$ matrix $Y$. Denote $A=Y\binom{A_{1}}{A_{2}}$, where $A_{1}$ and $A_{2}$ are respectively $s\times s$ and $(r-s)\times s$ matrices.
Then (\ref{ABA}) implies that
\begin{equation}\nonumber
 \begin{split}
  & A_{1}B^{G}A_{1}^{T}=B^{G};
  \\& A_{2}B^{G}(A_{1}^{T},A_{2}^{T})=0.
 \end{split}
\end{equation}
Note that $A_{1}B^{G}A_{1}^{T}=B^{G}$ implies that $A_{1}^{T}\in\Lambda(B^{G})$. Since $rank(A)=s$, $A_{2}B^{G}(A_{1}^{T},A_{2}^{T})=0$ implies that $A_{2}B^{G}=0$. Let $A''$ be the upper left $s'\times s'$ block of the matrix $A_{1}^{T}$. Since $A_{1}^{T}\in\Lambda(B^{G})$, we deduce that $A''\in\Lambda(B_{0}^{G})$.

Suppose that $\widehat{\a_{i+1}}=\bold{t}A, \widehat{\b_{i+1}}=\bold{t}A'$ for some $\bold{t}\in\mathbb{R}^{r}$. Then
$\widehat{\a_{i+1}}=\bold{s}\binom{A_{1}}{A_{2}}, \widehat{\b_{i+1}}=\bold{s}\binom{I_{s\times s}}{0_{s\times (r-s)}}$, where $\bold{s}=\bold{t}Y$. Denote $\bold{s}=(\bold{s}_{1},\bold{s}_{2}), \bold{s}_{1}\in\mathbb{R}^{s}, \bold{s}_{2}\in\mathbb{R}^{r-s}$.
Then $\widehat{\b_{i+1}}=\bold{s}_{1}, \widehat{\a_{i+1}}=\bold{s}_{1}A_{1}+\bold{s}_{2}A_{2}=\widehat{\b_{i+1}}A_{1}+\bold{s}_{2}A_{2}$. So $\widehat{\a_{i+1}}B^{G}=\widehat{\b_{i+1}}A_{1}B^{G}$.
Thus
$\widehat{\a_{i+1}}B^{G}A_{1}^{T}=\widehat{\b_{i+1}}A_{1}B^{G}A_{1}^{T}= \widehat{\b_{i+1}}B^{G}$.
Suppose that $\bold{s}_{3}\in\mathbb{R}^{s'}$ is the vector consisting of the first $s'$ entries of $\bold{s}_{1}$. Then
considering the first $s'$ entries of $\widehat{\a_{i+1}}B^{G}A_{1}^{T}$ and $\widehat{\b_{i+1}}B^{G}$, we get that
$\a_{i+1}^{\bullet}B_{0}^{G}A''=\b_{i+1}^{\bullet}B_{0}^{G}$.
Since the coefficients of each pair $(h_{j},h_{j+1}) (0\leq j\leq i)$ satisfy Property 2) of Lemma \ref{appro2}, we
deduce that $\a^{\bullet}B_{0}^{G}A''$ is at most $\frac{F_{5}}{N}$-away from $\b^{\bullet}B_{0}^{G}$ for some $F_{3}>0$ depending only on $X$ and $\D$, and $A''\in\Lambda(B_{0}^{G})$ is of height at most $F_{5}$. Therefore, conclusion
(ii) holds.
\end{proof}

\begin{proof}[Proof of Proposition \ref{est1}.]
In this proof $K,C_{0},C_{1},C_{2},\dots,\D_{1},\D_{2},\dots,\sigma_{1},\sigma_{2},\dots$ are constants depending only on $X$ and $\D$.

Suppose that
    \begin{equation}\label{B1}
  \begin{split}
    \Bigl\vert\mathbb{E}_{\A\in \R}\bold{1}_{P}(\A)\chi(\A)\Phi(g'(\B+\A)\cdot e_{X})\Bigr\vert\geq\D
  \end{split}
  \end{equation}
  for some $\B\in \R, \chi\in\mathcal{M}_{\Z}, \Phi\in Lip(X)$ with $\Vert\Phi\Vert_{Lip(X)}\leq 1$ and $\int\Phi dm_{X}=0$, and some $P=\{a+bi\colon a\in P_{1}, b\in P_{2}\}$, where $P_{i}$ is an arithmetic progression in $[\tilde{N}], i=1,2$. 
Our goal is to show there exist $\sigma=\sigma(X,\D), N_{0}=N_{0}(X,\D)$ such that if (\ref{B1}) holds for some $N\geq N_{0}$, then
  \begin{equation}\label{B2}
  \begin{split}
    (g(m,n))_{(m,n)\in[\tilde{N}]\times[\tilde{N}]} \text{ is not totally $\sigma$-equidistributed in $X$}.
  \end{split}
  \end{equation}

\textbf{Step 1: Reduction to some particular nilmanifold.}
 Write $r=\dim(G_{2}), m=s+r=\dim(G)$ and identify the vertical torus $G_{2}/(G_{2}\cap\Gamma)$ with $\mathbb{T}^{r}$.
We may assume without loss of generality that $\Vert\Phi\Vert_{\mathcal{C}^{2m}(X)}\leq 1$. Indeed,
there exists a function $\Phi'$ with $\Vert\Phi-\Phi'\Vert_{\infty}\leq \D/2$ and 
$\Vert\Phi'\Vert_{\mathcal{C}^{2m}(X)}$ is bounded by a constant depending only on $X$ and $\D$.

   We start with some definitions. For $\k\in\mathbb{Z}^{r}$, the character $\k$ of $G_{2}/(G_{2}\cap\Gamma)$
induces a character of $G_{2}$ given by a linear function $\phi_{\k}\colon G_{2}=\mathbb{R}^{r}\to\mathbb{R}$.
Let $G^{\k}$ denote the quotient of $G$ by $ker(\phi_{\k})\subset G_{2}$ and let $\Gamma^{\k}$ be the 
image of $\Gamma$ under this quotient. Then $\Gamma^{\k}$ is a discrete co-compact subgroup of $G^{\k}$.
Denote $X_{\k}=G^{\k}/\Gamma^{\k}$ and let $\pi_{\k}\colon X\to X_{\k}$ be the natural projection. 

If $\k$ is non-zero, then $X_{\k}$ is a non-abelian nilmanifold of order 2, and the vertical torus of $X_{\k}$ 
has dimension 1. If $\k$ is zero, then $X_{\k}$ is the maximum torus of $X$ and so is a compact abelian Lie group.

   We recall the definition of the vertical Fourier transform. The restriction to $G_{2}\cap\Gamma$ of the action by translation of $G$ on $X$ is trivial, and thus this action induces an action of the vertical torus on $X$ by
$(\bold{u},x)\mapsto\bold{u}\cdot x$ for $\bold{u}\in\mathbb{T}^{r}$ and $x\in X$. The vertical Fourier series of the function $\Phi$ is
\begin{equation}\nonumber
  \begin{split}
    \Phi=\sum_{\k\in\mathbb{Z}^{r}}\Phi_{\k}, \text{where }
    \Phi_{\k}(x)=\int_{\mathbb{T}^{r}}\Phi(\bold{u}\cdot x)e(-\k\cdot\bold{u}) dm_{\mathbb{T}^{r}}(\bold{u}), \k\in\mathbb{Z}^{r}.
  \end{split}
\end{equation}
The function $\Phi_{\k}$ is a nilcharacter with frequency $\k$ and
thus can be written as
\begin{equation}\nonumber
  \begin{split}
    \Phi_{\k}=\Psi_{\k}\circ\pi_{\k}
  \end{split}
\end{equation}
for some function $\Psi_{\k}$ on $X_{\k}$.
If $\k\neq 0$, then $\Phi_{\k}$ is a nilcharacter of $X_{\k}$ with frequency equal to 1. Moreover, since $\Vert\Phi\Vert_{\mathcal{C}^{2m}(X)}\leq 1$, we have that $\Vert\Phi_{\k}\Vert_{\mathcal{C}^{2m}(X)}\leq 1$, and there exists $C_{0}>0$ such that and
$\vert\Phi_{\k}(x)\vert\leq C_{0}(1+\Vert\k\Vert)^{-2m}$ for every $\k\in\mathbb{Z}^{r}$ and every $x\in X$. Since $m>r$, there exists a constant $C_{1}$ such that
\begin{equation}\nonumber
  \begin{split}
    \sum_{\k\colon\Vert\k\Vert>C_{1}}\vert\Phi_{\k}(x)\vert<\D/2
    \text{ for every } x\in X.
  \end{split}
\end{equation}
 Replacing $\Phi$ in (\ref{B1}) by its vertical Fourier series, this last bound implies that there exists $\k\in\mathbb{Z}^{r}=\widehat{\mathbb{T}^{r}}$
   such that
    \begin{equation}\nonumber
  \begin{split}
    \Vert\k\Vert\leq C_{1}, \Vert\Phi_{\k}\Vert\leq 1, \Bigl\vert\mathbb{E}_{\A\in \R}\bold{1}_{P}(\A)\chi(\A)\Phi_{\k}(g'(\B+\A)\cdot e_{X})\Bigr\vert\geq\D_{1}
  \end{split}
  \end{equation}
  for some constants $C_{1},\D_{1}>0$.
   If $\k=\bold{0}$, then the conclusion follows from Proposition \ref{est2}. So we may assume $\k\neq\bold{0}$ and continue our proof with the assumptions that
\begin{equation}\label{B3}
  \begin{split}
    & r=\dim(G_{2})=1;
    \\& \Phi \text{ is a nilcharacter of frequency } 1;
    \\& \Vert\Phi\Vert_{Lip(X)}\leq 1;
    \\& \Bigl\vert\mathbb{E}_{\A\in \R}\bold{1}_{P}(\A)\chi(\A)\Phi(g'(\B+\A)\cdot e_{X})\Bigr\vert\geq\D_{2}
  \end{split}
  \end{equation}
   for some constants $\D_{2}>0$.

\textbf{Step 2: Reduction to some particular polynomial.} We make some further reductions. Suppose that the conclusion (\ref{B2}) holds for some $\sigma$ and $N_{0}$ under the stronger assumption that (\ref{B1}) holds for $\B=0$ and a sequence given by $g'(m+ni)=g(m,n)=g_{1,1}^{m}g_{1,2}^{n}g_{2,1}^{\binom{m}{2}}g_{2,2}^{mn}g_{2,3}^{\binom{n}{2}}$ for all $m,n\in\mathbb{N}$, where $g_{1,1},g_{1,2}\in G, g_{2,1},g_{2,2},g_{2,3}\in G_{2}$.

  Let $\D>0$ and $N\geq N_{0}$. Let $F_{1}\subset G$ be a bounded fundamental domain of the projection $G\rightarrow X$ (we assume that $F_{1}$ is fixed given $X$). By the first statement of Lemma
  \ref{cts}, there exists a constant $C_{2}>0$ such that
  \begin{equation}\label{B4}
  \begin{split}
   d_{X}(g\cdot x,g\cdot x')\leq C_{2}d_{X}(x,x')
  \end{split}
  \end{equation}
for all $g\in F_{1}$ and $x,x'\in X$. Given $g_{0},g_{1,1},g_{1,2}\in G, g_{2,1},g_{2,2},g_{2,3}\in G_{2}$ and $\B=m'+n'i\in \R$, write
  \begin{equation}\nonumber
  \begin{split}
   & g_{0}g_{1,1}^{m'}g_{1,2}^{n'}g_{2,1}^{\binom{m'}{2}}g_{2,2}^{m'n'}g_{2,3}^{\binom{n'}{2}}=a_{\B}t_{\B}, a_{\B}\in F, t_{\B}\in\Gamma;
   \\& g_{\B,1,1}=t_{\B}g_{1,1}g_{2,1}^{m'}g_{2,2}^{n'}t_{\B}^{-1};
   \\& g_{\B,1,2}=t_{\B}g_{1,1}g_{2,2}^{m'}g_{2,3}^{n'}t_{\B}^{-1}.
  \end{split}
  \end{equation}
Then for $\A=m+ni\in \R$, we have that
  \begin{equation}\nonumber
  \begin{split}
  g'(\A+\B)\cdot e_{X}=g(m+m',n+n')\cdot e_{X}=a_{\B}g_{\B,1,1}^{m}g_{\B,1,2}^{n}g_{2,1}^{\binom{m}{2}}g_{2,2}^{mn}g_{2,3}^{\binom{n}{2}}\cdot e_{X}.
  \end{split}
  \end{equation}
Set
  \begin{equation}\nonumber
  \begin{split}
 \Phi_{\B}(x)=\Phi(a_{\B}\cdot x).
  \end{split}
  \end{equation}
Then for every $\B\in \R, \Phi_{\B}$ is a nilcharacter of frequency 1. Since $a_{\B}\in F_{1}$ and $\Vert\Phi\Vert_{Lip(X)}\leq 1$, we get by (\ref{B4}) that  $\Vert\Phi_{\B}\Vert_{Lip(X)}\leq C_{2}$. Estimate (\ref{B3}) can be rewritten as
     \begin{equation}\nonumber
  \begin{split}
    \Bigl\vert\mathbb{E}_{\A\in \R}\bold{1}_{P}(\A)\chi(\A)\Phi_{\B}(g'_{\B}(\A)\cdot e_{X})\Bigr\vert\geq\D_{2},
  \end{split}
  \end{equation}
where $g'_{\B}(m+ni)=g_{\B,1,1}^{m}g_{\B,1,2}^{n}g_{2,1}^{\binom{m}{2}}g_{2,2}^{mn}g_{2,3}^{\binom{n}{2}}$. 
By assumption, we deduce that the sequence $(g'_{\B}(m+ni))_{(m,n)\in[\tilde{N}]\times[\tilde{N}]}$ is not totally
 $\sigma_{1}$-equidistributed in $X$ for some $\sigma_{1}>0$. 
Let $\eta$ be the horizontal character provided by Theorem \ref{Lei}. Then 
$\eta(g'_{\B}(\A))=\eta(a_{\B})^{-1}\eta(g'(\A+\B))$.
Applying Lemma \ref{51} with $\phi(\A)=\eta(g'(\A+\B))$ and $\psi(\A)=\eta(g'(\A))$ and then applying
Theorem \ref{inv}, we deduce that there exist an integer $N_{0}'$ and a constant 
$\sigma_{2}>0$ such that if $N\geq N_{0}'$, then the sequence $(g'(m+ni))_{(m,n)\in[\tilde{N}]\times[\tilde{N}]}$ 
is not totally $\sigma_{2}$-equidistributed in $X$. Hence in the rest of the proof, we can take $\B=0$ 
and $g_{0}=id_{G}$.

\textbf{Step 3: Non-equidistribution on $X\times X$.}
Combining (\ref{B3}) with Lemma \ref{katai}, there exist a positive integer $K$, primes $p=p_{1}+p_{2}i,q=q_{1}+q_{2}i\in\P$ with $\mathcal{N}(p)<\mathcal{N}(q)<K$, and a positive constant $\D_{3}$ such that
\begin{equation}\label{B00}
  \begin{split}
   \frac{1}{N^{2}}\Bigl\vert\sum_{\A\in \R/p\cap \R/q}\bold{1}_{P}(p\A)\bold{1}_{P}(q\A)\Phi(g'(p\A)\cdot e_{X})\overline{\Phi}
   (g'(q\A)\cdot e_{X})\Bigr\vert>\D_{3}.
  \end{split}
  \end{equation}
Since $p,q$ belong to a finite set whose cardinality depends only on $\D$, we may consider these numbers as fixed. Let $L$ be a sufficiently large integer depending only on $X$ and $\D$ to be chosen later. Let $N>L^{2}$. By (\ref{B00}) and a similar argument to that of Lemma 3.1 in ~\cite{er2}, for each $\C=r_{1}+r_{2}i$ with $r_{1},r_{2}\in[L]$, there exists $\B_{\C}\in \R$ such that
\begin{equation}\label{B001}
  \begin{split}
   \frac{L^{2}}{N}\Bigl\vert\sum_{n\in[N/L^{2}],\B_{\C}+n\C\in \R}\bold{1}_{P_{\C}}(n)\Phi(g'(p(\B_{\C}+n\C))\cdot e_{X})\overline{\Phi}
   (g'(q(\B_{\C}+n\C))\cdot e_{X})\Bigr\vert>\D_{3}/2
  \end{split}
  \end{equation}
for some arithmetic progression $P_{\C}$. Let $f_{1}(\C)=g_{1,1}^{r_{1}p_{1}-r_{2}p_{2}}g_{1,2}^{r_{1}p_{2}+r_{2}p_{1}},
f_{2}(\C)=g_{1,1}^{r_{1}q_{1}-r_{2}q_{2}}g_{1,2}^{r_{1}q_{2}+r_{2}q_{1}}$ be maps from $[L]^{2}$ to $G$. Then by Proposition \ref{last} (replacing $\a$ and $\b$ with $f_{1}(\C)$ and $f_{2}(\C)$), $(f_{1}(\C),f_{2}(\C))$ satisfies one of the two claims (with $N$ replaced by $N/L^{2}$). So there exist $D=D(X,\D)>0, C(L)>1$ and a subset $W$ of $[L]^{2}$ such that $\vert W\vert\geq L^{2}/4$, and one of the following holds:

%  (i) For all $\C\in W$, there exists a horizontal character $\eta_{\C}$ such that %$0<\Vert\eta_{\C}\Vert\leq D$ and $\Vert\eta\circ h_{\{\C\}}\Vert_{C^{\infty}[N]}\leq C(L)D$, %where $h_{\{\C\}}(n)=(f_{1}({\C})f_{2}({\C})^{-1})^{n}$;

  (i) For all $\C\in W$, there exist horizontal characters $\eta_{1,\C},\eta_{2,\C}$ such that
$0<\Vert\eta_{1,\C}\Vert,\Vert\eta_{2,\C}\Vert\leq D$ and $\Vert\eta_{1,\C}\circ h_{1,\{\C\}}\Vert_{C^{\infty}[N]},\Vert\eta_{1,\C}\circ h_{2,\{\C\}}\Vert_{C^{\infty}[N]}\leq C(L)D$, where $h_{1,\{\C\}}(n)=f_{1}({\C})^{n},h_{2,\{\C\}}(n)=f_{2}({\C})^{n}$;

  (ii) For all $\C\in W$, there exists a matrix $M(\C)\in \Lambda(B_{0}^{G})$ of height at most $D$ such that $f_{1}({\C})^{\bullet}B_{0}^{G}M(\C)=_{C(L)D}f_{2}({\C})^{\bullet}B_{0}^{G}$. 

Here the notation $A=_{D}B$ means that $A$ is at most $\frac{D}{N}$-away from $B$, and we use this notation throughout the proof.

Since the number of choices of $\eta_{1},\eta_{2}$ and $M$ is bounded by a constant depending only on $X$ and $\D$,
if $L$ is sufficiently large, there exist $\e=\e(X,\D)>0$ and a subset $V$ of $W$ such that $\vert V\vert>\e L^{2}$ and for all
 $\C\in W$, the corresponding $\eta_{1,\C},\eta_{2,\C},M(\C)$, depending on which of the above 
situations occurs, are the the same. We denote this vector or matrix by $\eta_{1},\eta_{2}$ or $M$ depending on which of the above two situations occurs.
By Lemma 3.2 in ~\cite{er2}, the vectors in $V$ spans $\mathbb{Q}^{2}$.

Write $D_{1}=C(L)D$, which depends only $X$ and $\D$ (since $L$ depends only on $X$ and $\D$). Suppose $\widehat{g_{1,1}}=\a, \widehat{g_{1,2}}=\b$, and
$\a^{\bullet},\b^{\bullet}$ are the first $s'$ coordinates of $g_{1,1}$ and $g_{1,2}$, respectively.

\textbf{Case (i).} Now we have that $\Vert\eta_{1}\circ h_{1,\{1\}}\Vert,\Vert\eta_{1}\circ h_{1,\{i\}}\Vert\leq D_{1}$. Suppose that $\eta_{1}(\bold{x})=\bold{k}\cdot\widehat{\bold{x}}, 0<\Vert \bold{k}\Vert\leq D_{1}$. Then
\begin{equation}\nonumber
  \begin{split}
  &\Vert p_{1}\bold{k}\cdot\a+p_{2}\bold{k}\cdot\b\Vert_{\mathbb{R}/\mathbb{Z}}\leq D_{1}/N;
    \\& \Vert -p_{2}\bold{k}\cdot\a+p_{1}\bold{k}\cdot\b\Vert_{\mathbb{R}/\mathbb{Z}}\leq D_{1}/N.
  \end{split}
\end{equation}
This implies $\Vert \Delta\bold{k}\cdot\a\Vert_{\mathbb{R}/\mathbb{Z}}, \Vert \Delta\bold{k}\cdot\b\Vert_{\mathbb{R}/\mathbb{Z}}\leq D_{2}/N$, where $\Delta=p_{1}^{2}+p_{2}^{2}\neq 0$. Then by Theorem \ref{inv}, $(g(m,n))_{(m,n)\in[\tilde{N}]\times[\tilde{N}]}$ is not totally $\D_{4}$-equidistributed, a contradiction.

\textbf{Case (ii).} Now we have that $f_{1}(1)^{\bullet}B_{0}^{G}M=_{D_{1}}f_{2}(1)^{\bullet}B_{0}^{G},f_{1}(i)^{\bullet}B_{0}^{G}M=_{D_{1}}f_{2}(i)^{\bullet}B_{0}^{G}$ (recall that $A=_{D}B$ means that $A$ is at most $\frac{D}{N}$-away from $B$). This is equivalent to saying that
\begin{equation}\label{rep}
\begin{split}
  WM=_{D_{1}}RW,
\end{split}
\end{equation}
where
\begin{equation}\nonumber
W=\binom{\a^{\bullet}}{\b^{\bullet}}B_{0}^{G}, R=P^{-1}Q,
P=\begin{vmatrix}

p_{1} &  p_{2}\\

-p_{2} & p_{1}\\
\end{vmatrix},
Q=\begin{vmatrix}

q_{1} &  q_{2}\\

-q_{2} & q_{1}\\
\end{vmatrix}.
\end{equation}
Let $f$ and $g$ be the minimal polynomials of the matrices $M$ and $R$, respectively, i.e.
the integer polynomials with leading coefficient 1 and with the smallest
degree such that $f(M)=0$ and $g(R)=0$. Since $M$ and $R$ are integer matrices of height at most $D_{1}$, $f$ and $g$ are polynomials
with coefficients of height at most $D_{2}$. Since
\begin{equation}\nonumber
R=\begin{vmatrix}

x &  y\\

-y & x\\
\end{vmatrix}
\end{equation}
for some rational numbers $x$ and $y$ not both 0, we deduce that $\deg g=2$.

If $g(M)\neq 0$, then (\ref{rep}) implies that $0=g(R)W=_{D_{3}}Wg(M)$. Since $g(M)\neq 0$, there exists at least one column of $g(M)$ which is not 0.
So there exists
$\bold{v}\in\mathbb{Z}^{s'},0<\Vert\bold{v}\Vert\leq D_{4}$ such that $\Vert\bold{v}\cdot\a^{\bullet} B_{0}^{G}\Vert,\Vert\bold{v}\cdot\b^{\bullet} B_{0}^{G}\Vert\leq\frac{D_{4}}{N}$. By Theorem \ref{inv} and the fact that $B_{0}^{G}$ is invertible, $(g(m,n))_{(m,n)\in[\tilde{N}]\times[\tilde{N}]}$ is not totally $\D_{4}$-equidistributed, a contradiction.

If $g(M)=0$ but $f(R)\neq 0$, then (\ref{rep}) implies that $0=Wf(M)=_{D_{3}}f(R)W$. By the minimality of $f$, we have that $f$ divides $g$. Since $\deg g=2$, we deduce that $\deg f\leq 1$. This means that $M$ is diagonal. Thus $\lambda f_{1}(1)^{\bullet}B_{0}^{G}=_{D_{1}}f_{2}(1)^{\bullet}B_{0}^{G},\lambda f_{1}(i)^{\bullet}B_{0}^{G}=_{D_{1}}f_{2}(i)^{\bullet}B_{0}^{G}$ for some $\lambda\in\mathbb{Z}, \vert\lambda\vert\leq D$. This implies that
\begin{equation}\nonumber
  \begin{split}
  &\lambda p_{1}\a^{\bullet} B_{0}^{G}+\lambda p_{2}\b^{\bullet} B_{0}^{G}=_{D_{1}}q_{1}\a^{\bullet} B_{0}^{G}+q_{2}\b^{\bullet} B_{0}^{G};
    \\& -\lambda p_{2}\a^{\bullet} B_{0}^{G}+\lambda p_{1}\b^{\bullet} B_{0}^{G}=_{D_{1}}-q_{2}\a^{\bullet} B_{0}^{G}+q_{1}\b^{\bullet} B_{0}^{G}.
  \end{split}
\end{equation}
Thus $\a^{\bullet} B_{0}^{G}=_{D_{2}}\b^{\bullet} B_{0}^{G}=_{D_{2}}\bold{0}$. Since $G$ is not abelian, there exists at least one row $\bold{v}$ of $B_{0}^{G}$ such that $\bold{v}\neq\bold{0}$.
Then $\bold{v}\cdot\a^{\bullet}=_{D_{2}}\bold{v}\cdot\b^{\bullet}=_{D_{2}}0$. So $(g(m,n))_{(m,n)\in[\tilde{N}]\times[\tilde{N}]}$ is not totally $\D_{4}$-equidistributed, a contradiction.

If $f(R)=0$ and $g(M)=0$, then $f=g$. Since
$\deg g=2,$ we may assume that $f(x)=g(x)=x^{2}+c_{1}x+c_{0}, c_{0},c_{1}\in\mathbb{Z}$ and let $\mu_{1},\mu_{2}$
be its roots. Since $R$ has no real eigenvalue, $\mu_{1}$ and $\mu_{2}$ are conjugate complex numbers. Since $M^{2}+c_{1}M+c_{0}I=0$, the eigenvalues of $M$ are $\mu_{1}$ and $\mu_{2}$, and from the fact that $c_{0},c_{1}\in\mathbb{Z}$ and $\mu_{1}, \mu_{2}$ are not real numbers, we can easily
deduce that the Jordan normal form $J$ of $M$ must be a diagonal matrix with each entry on the diagonal either $\mu_{1}$ or $\mu_{2}$. So
there exists an $m\times m$ invertible complex-valued matrix $S$ such that
\begin{equation}\nonumber
\begin{split}
M=SJS^{-1}.
\end{split}
\end{equation}
Since $M\in\Lambda(B_{0}^{G})$, we deduce that
\begin{equation}\label{zxc}
\begin{split}
J(S^{T}B_{0}^{G}S)J=S^{T}B_{0}^{G}S.
\end{split}
\end{equation}
Since $\mu_{1}, \mu_{2}$ are not real numbers, $\mu_{1}^{2}\neq 1$ and $\mu_{2}^{2}\neq 1$. Since we also have that $\mu_{1}\mu_{2}=\det R=\det Q/\det P\neq 1$, we can deduce from (\ref{zxc}) and the fact that $J$ is diagonal that $S^{T}B_{0}^{G}S=0$, i.e., $B_{0}^{G}=0$. In other words, $G$ is abelian, a contradiction.
\end{proof}

\section{Proof of the decomposition result} The purpose of this section is to prove Theorem \ref{nU3s}.
The following lemma generalizes Lemma A.6 in ~\cite{FH}:

\begin{lem}\label{1p}
  Let $N\in\mathbb{N}$. For every function $a\colon \R\rightarrow\mathbb{C}$ and all arithmetic progressions $P_{1}, P_{2}\subset[\tilde{N}]$, we have that
  \begin{equation}\nonumber
    \begin{split}
      \Bigl\vert\mathbb{E}_{\A\in \R}\bold{1}_{P}(\A)a(\A)\Bigr\vert\leq c_{1}\Vert a\Vert_{U^{2}(\R)}
    \end{split}
  \end{equation}
  for some universal constant $c_{1}$, where $P=\{a+bi\in \R\colon a\in P_{1}, b\in P_{2}\}$.
\end{lem}
\begin{proof}
  Since $\tilde{N}$ is a prime, the norm $\Vert a\Vert_{U^{2}(\A)}$ is invariant under any change of variables of the form $m+ni\rightarrow (am+c)+(bn+d)i$, where $a,b,c,d\in\mathbb{N}$ and $\tilde{N}\nmid a, \tilde{N}\nmid b$. So we may assume without loss of generality that $P_{i}$ is an interval $\{1,\dots,d_{i}\}$, for $i=1,2$. A direct computation shows that
    \begin{equation}\nonumber
    \begin{split}
     \vert\widehat{\bold{1}_{P}}(x+yi)\vert\leq\frac{4}{\tilde{N}^{2}}\Vert x/\tilde{N}\Vert\cdot\Vert y/\tilde{N}\Vert=\frac{4}{\min\{x,\tilde{N}-x\}\cdot\min\{y,\tilde{N}-y\}}
    \end{split}
  \end{equation}
  for all $x+yi\in \R$. Thus
      \begin{equation}\nonumber
    \begin{split}
     \Bigl\Vert\widehat{\bold{1}_{P}}(x+yi)\Bigr\Vert_{\ell^{4/3}(\R)}\leq c_{1}
    \end{split}
  \end{equation}
  for some universal constant $c_{1}$. Then by Parseval's identity, H\"{o}lder's inequality, and identity (\ref{F}), we deduce that
    \begin{equation}\nonumber
    \begin{split}
      \Bigl\vert\mathbb{E}_{\A\in \R}\bold{1}_{P}(\A)a(\A)\Bigr\vert
      =\Bigl\vert\sum_{\A\in \R}\widehat{\bold{1}_{P}}(\A)\widehat{a}(\A)\Bigr\vert
      \leq c_{1}\Bigl(\sum_{\xi\in \R}\vert \widehat{a}(\xi)\vert^{4}\Bigr)^{1/4}
      \leq c_{1}\Vert a\Vert_{U^{2}(\R)}.
    \end{split}
  \end{equation}
\end{proof}

Before the proof of Theorem \ref{nU3s}, we show:
\begin{thm}\label{nU3w}(Weak $U^{3}$ decomposition theorem).
  For every $\theta_{0},\e>0$, there exist positive integers $Q=Q(\e,\theta_{0}),R=R(\e,\theta_{0}), N_{0}=(\e,\theta_{0})$ and $0<\theta<\theta_{0}$ such that for every $N\geq N_{0}$ and every $\chi\in\mathcal{M}_{\Z}$, the function $\chi_{N}$ can be written as
  \begin{equation}\nonumber
   \begin{split}
     \chi_{N}(\A)=\chi_{N,s}(\A)+\chi_{N,u}(\A),
   \end{split}
  \end{equation}
where

  (i) $\chi_{N,s}=\chi_{N}*\phi_{N,\theta}$, where $\phi_{N,\theta}$ is the kernel of $\R$ defined in Theorem \ref{nU2}
  which is independent of $\chi$, and the convolution product is defined on $\R$;

  (ii)$\vert\chi_{N,s}(\A+Q)-\chi_{N,s}(\A)\vert, \vert\chi_{N,s}(\A+Qi)-\chi_{N,s}(\A)\vert\leq\frac{R}{N}$ for every $\A\in\R$;

  (iii)$\Vert\chi_{N,u}\Vert_{U^{3}(\R)}\leq\e$.
\end{thm}

By using an iterative argument of energy increment, we can deduce that Theorem \ref{nU3w} implies Theorem \ref{nU3s}. As the method is identical to Section 8.10 in ~\cite{FH}, we omit the proof. Therefore, it suffices to prove Theorem \ref{nU3w}. The method of Theorem \ref{nU3w} is similar to the discussion of Section 8 in ~\cite{FH}. We include the proof in this section for completeness.

\subsection{Set up}
Throughout this section, let $\e$ be fixed. Let
  \begin{equation}\nonumber
   \begin{split}
     \mathcal{H}=\mathcal{H}(\e), \d=\d(\e), m=m(\e)
   \end{split}
  \end{equation}
be defined by Corollary \ref{mIU3}. In the sequel we implicitly assume that $N$ is sufficiently large (thus so is $\tilde{N}$) depending only on $\e$. So Corollary \ref{mIU3} always holds.

Let $M\in\mathbb{N}$ and $X=G/\Gamma$ be a nilmanifold in $\mathcal{H}$. By Corollary B.3 in ~\cite{FH}, for every $M\in\mathbb{N}$, there exists a finite subset $\Sigma=\Sigma(M,X)\subset G$ of $M$-rational elements such that for every $M$-rational element $g\in G$, there exists $h\in\Sigma$ with $h^{-1}g\in\Gamma$, i.e. $g\cdot e_{X}=h\cdot e_{X}$. We assume that $\bold{1}_{G}\in\Sigma$.

Let $\mathcal{F}=\mathcal{F}(M,X)$ be the family of submanifolds of $X$ defined by Corollary \ref{mF}. We define a larger family of nilmanifolds
  \begin{equation}\nonumber
   \begin{split}
     \mathcal{F}'=\mathcal{F}'(M,X)=\{Y=G'\cdot e_{Y}\cong G'/(h\Gamma h^{-1}\cap G')\colon X'=G'/\Gamma'\in\mathcal{F}, h\in\Sigma, e_{Y}=h\cdot e_{X}\}.
   \end{split}
  \end{equation}
By Lemma \ref{cts}, there exists a positive real number $H=H(M,X)$ such that

\

(i) $d_{X}(gh\cdot x,gh\cdot x')\leq Hd_{X}(x,x')$ for all $x,x'\in X, h\in\Sigma$ and all $g\in G$ with $d_{G}(g,\bold{1}_{G})\leq M$;

(ii) for every $f\in\mathcal{C}^{2m}(X), h\in\Sigma$ and every $g\in G$ with $d_{G}(g,\bold{1}_{G})\leq M$, writing $f_{gh}(x)=f(gh\cdot x)$, we have that $\Vert f_{gh}\Vert_{\mathcal{C}^{2m}(X)}\leq H\Vert f\Vert_{\mathcal{C}^{2m}(X)}$.

\

The distance on a nilmanifold $Y\in\mathcal{F}'$ is not the one induced by inclusion in $X$. However, the inclusion map $i\colon Y\rightarrow X$ is smooth and thus we can assume that

\

(iii) for every nilmanifold $Y\in\mathcal{F}'$ and every $x,x'\in Y$, we have that $d_{X}(x,x')\leq H d_{Y}(x,x')$;

(iv) for every nilmanifold $Y\in\mathcal{F}'$ and every function $f$ on $X$, we have that $\Vert f|_{Y}\Vert_{\mathcal{C}^{2m}(Y)}\leq H\Vert f\Vert_{\mathcal{C}^{2m}(X)}$. Here the Mal'cev homeomorphism of $Y$ is taken to be the induced Mal'cev homeomorphism from $X$ to $Y$.

\

By Lemma \ref{shift}, for every $X'\in\mathcal{F},\zeta>0$ and $h\in\Sigma$, there exists $\rho=\rho(M,X,X',h,\zeta)$ such that

\

(v) Let $X'=G'/\Gamma'\in\mathcal{F},h\in\Sigma, e_{Y}=h\cdot e_{X}$, and $(g'(m,n))_{(m,n)\in[\tilde{N}]\times[\tilde{N}]}$ be a polynomial sequence in $G'$ of degree at most 2. If $(g'(m,n)e_{X})_{(m,n)\in[\tilde{N}]\times[\tilde{N}]}$ is totally $\rho$-equidistributed in $X'$, then $(g'(m,n)e_{Y})_{(m,n)\in[\tilde{N}]\times[\tilde{N}]}$ is totally $\zeta$-equidistributed in $Y'=G'\cdot e_{Y}$.

\

Define

  \begin{equation}\nonumber
   \begin{split}
     &\mathcal{F}'(M)=\bigcup_{X\in\mathcal{H}}\mathcal{F}'(M,X);
     \\& H(M)=\max_{X\in\mathcal{H}}H(M,X);
     \\& \rho(M,\zeta)=\min_{X\in\mathcal{H},X'\in\mathcal{F}(M,X),h\in\Sigma(M,X)}\rho(M,X,X',h,\zeta);
     \\& \d_{1}(M)=\frac{\d^{3}}{257M^{4}};
     \\& \theta(M)=\min\Bigl\{\frac{\d}{2},\frac{\d_{1}(M)}{2c_{1}}\Bigr\},
   \end{split}
  \end{equation}
where $c_{1}$ is the universal constant in Lemma \ref{1p} (recall $\d=\d(\e)$). To every $\D>0$ and every nilmanifold $Y$ in the finite collection $\mathcal{F}'(M)$, either Proposition \ref{est2} or Proposition \ref{est1} (applied with $Y$ instead of $X$) associates a positive number $\sigma(Y,\D)$ (depending on whether $Y$ is abelian or not). Let
  \begin{equation}\nonumber
   \begin{split}
   &\tilde{\sigma}(M)=\min_{Y\in\mathcal{F}'(M)}\sigma\Bigl(Y,\frac{\d_{1}(M)}{17H(M)^{2}}\Bigr);
   \\& \omega(M)=\rho(M,\tilde{\sigma}(M));
   \\& M_{0}=\lceil 2/\e\rceil;
   M_{1}=\max_{X\in\mathcal{H}}M_{1}(M_{0},X,\omega),
   \end{split}
  \end{equation}
where $M_{1}(M_{0},X,\omega)$ is defined in Corollary \ref{mF} ($\omega$ is viewed as a function on $\mathbb{N}$). Note that $M_{1}$ depends only on $\e$.
Write
  \begin{equation}\nonumber
   \begin{split}
   \theta_{1}=\theta(M_{1}).
   \end{split}
  \end{equation}
\subsection{Weak $U^{3}$ decomposition}
Replacing $\e$ by $\theta_{1}$ in Theorem \ref{nU2}, we deduce that there exist $Q=Q(M_{1}),R=R(M_{1})$ (depending only on $\e$) such that for sufficiently large $N$ and every $\chi\in\mathcal{M}_{\Z}$, the decomposition $\chi_{N}=\chi_{N,s}+\chi_{N,u}$ satisfies the conclusions of Theorem \ref{nU2}. In particular, $\Vert\chi_{N,u}\Vert_{U^{2}(\R)}\leq \theta_{1}$. We claim that this decomposition also satisfies the conclusions of Theorem \ref{nU3w}. Note that (i) and (ii) follow from the conclusion of Theorem \ref{nU2}, so we are left with checking (iii).

Suppose on the contrary that (iii) does not hold, then
  \begin{equation}\nonumber
   \begin{split}
   \Vert\chi_{N,u}\Vert_{U^{3}(\R)}>\e.
   \end{split}
  \end{equation}
By the choice of $\theta_{1}$ and Corollary \ref{mIU3}, there exist a nilmanifold $X=G/\Gamma$ of order 2 belonging to the family $\mathcal{H}$, a nilcharacter $\Psi$ on $X$ with frequency 1, and a polynomial map $g\colon\mathbb{Z}^{2}\rightarrow G$ of degree at most 2, such that $\Vert\Psi\Vert_{\mathcal{C}^{2m}(X)}\leq 1$ and
  \begin{equation}\nonumber
   \begin{split}
   \Bigl\vert\mathbb{E}_{\A=a+bi\in \R}\chi_{N,u}(\A)\Psi(g(a,b)\cdot e_{X})\Bigr\vert\geq\d,
   \end{split}
  \end{equation}
  where $m$ is the dimension of $X$. Applying Corollary \ref{mF} to $(g(m,n))_{(m,n)\in [\tilde{N}]\times[\tilde{N}]}$ and $\omega, M_{0}, M_{1}$ defined in the previous subsection, we get an integer $M_{0}\leq M\leq M_{1}$, a nilmanifold $X'=G'/\Gamma'$ belonging to the family $\mathcal{F}(M,X)$, and a factorization
    \begin{equation}\nonumber
   \begin{split}
   g(m,n)=\e(m,n)g'(m,n)\gamma(m,n)
   \end{split}
  \end{equation}
  into sequences satisfying Properties (i)-(v) of Corollary \ref{mF}. From now on, we work with this specifically chosen $M$. Note that $M$ is bounded below and above by constants depending only on $\e$.

  Denote $L=\lfloor\frac{\d\tilde{N}}{8M^{2}}\rfloor$. By Property (v) of Corollary \ref{mF}, we can partition $[\tilde{N}]\times[\tilde{N}]$ into products of arithmetic progressions of step $p_{1}$ and $p_{2}$ and of size $L\times L$ for some $p_{1},p_{2}\leq M$, and a leftover set of size $C\d/M$ which can be ignored (upon replacing $\d$ with $\d/2$ below) such that $\gamma(m,n)$ is a constant on each piece. So there exists a product of arithmetic progressions $P$ of step $p_{1}$ and $p_{2}$ and of size $L\times L$ such that
  \begin{equation}\nonumber
   \begin{split}
   \Bigl\vert\mathbb{E}_{(m,n)\in [\tilde{N}]\times[\tilde{N}]}\bold{1}_{P}(m,n)\chi_{N,u}(m+ni)\Psi(g(m,n)\cdot e_{X})\Bigr\vert\geq\frac{\d L^{2}}{2\tilde{N}^{2}}\geq
   \frac{\d^{3}}{128M^{4}}-\frac{\d^{2}}{8M^{2}\tilde{N}}.
   \end{split}
  \end{equation}

Pick $(m,n),(m_{0},n_{0})\in P$. By the right invariance of $d_{G}$ and the $(M,(\tilde{N},\tilde{N}))$-smoothness of $(\e(m,n))_{(m,n)\in \R}$, we have that
  \begin{equation}\nonumber
   \begin{split}
   &\qquad d_{G}(g(m,n),\e(m_{0},n_{0})g'(m,n)\gamma(m_{0},n_{0}))\leq d_{G}(\e(m,n),\e(m_{0},n_{0}))
   \\&\leq (p_{1}+p_{2})L\frac{\sqrt{2}M}{\tilde{N}}\leq \frac{2\sqrt{2}M^{2}L}{\tilde{N}}.
   \end{split}
  \end{equation}
  Then
  \begin{equation}\nonumber
   \begin{split}
   \Bigl\vert\Psi(g(m,n)\cdot e_{X})-\Psi(\e(m_{0},n_{0})g'(m,n)\gamma(m_{0},n_{0})\cdot e_{X})\Bigr\vert\leq \frac{2\sqrt{2}M^{2}L}{\tilde{N}}.
   \end{split}
  \end{equation}
Thus
  \begin{equation}\nonumber
   \begin{split}
   &\qquad\mathbb{E}_{(m,n)\in [\tilde{N}]\times[\tilde{N}]}\bold{1}_{P}(m,n)\Bigl\vert\chi_{N,u}(m+ni)\Bigr\vert\cdot\Bigl\vert\Psi(g(m,n)\cdot e_{X})
   -\Psi(\e(m_{0},n_{0})g'(m,n)\gamma(m_{0},n_{0}))\cdot e_{X})\Bigr\vert
   \\&\leq
   \frac{L^{2}}{\tilde{N}^{2}}\frac{2\sqrt{2}M^{2}L}{\tilde{N}}\leq\frac{\d^{3}}{256M^{4}}.
   \end{split}
  \end{equation}
  So we deduce that
  \begin{equation}\label{temp1}
   \begin{split}
   \Bigl\vert\mathbb{E}_{(m,n)\in [\tilde{N}]\times[\tilde{N}]}\bold{1}_{P}(m,n)\chi_{N,u}(m+ni)\Psi(\e(m_{0},n_{0})g'(m,n)\gamma(m_{0},n_{0})\cdot e_{X})\Bigr\vert\geq\frac{\d^{3}}{257M^{4}}=\d_{1}(M)
  \end{split}
  \end{equation}
  provided that $N$ is sufficiently large depending on $\e$.

Since $\gamma(m_{0},n_{0})$ is $M$-rational, there exists $h_{0}\in\Sigma(M,X)$ such that $\gamma(m_{0},n_{0})\cdot e_{X}=h_{0}\cdot e_{X}$. Let
$e_{Y}=h_{0}\cdot e_{X}, Y=G'\cdot e_{Y}\cong G'/(h_{0}\Gamma h_{0}^{-1}\cap G'), \Psi'(x)=\Psi(\e(m_{0},n_{0})\cdot x)$. Note that $Y$ belongs to the family $\mathcal{F}'$. For every $m,n\in\mathbb{N}$, we have that
\begin{equation}\nonumber
   \begin{split}
   \Psi(\e(m_{0},n_{0})g'(m,n)\gamma(m_{0},n_{0})\cdot e_{X})=\Psi'(g'(m,n)\cdot e_{Y}),
  \end{split}
  \end{equation}
  and by (\ref{temp1}), we deduce that
  \begin{equation}\label{temp2}
   \begin{split}
   \Bigl\vert\mathbb{E}_{(m,n)\in [\tilde{N}]\times[\tilde{N}]}\bold{1}_{P}(m,n)\chi_{N,u}(m+ni)\Psi'(g'(m,n)\cdot e_{Y})\Bigr\vert\geq\d_{1}(M).
  \end{split}
  \end{equation}
Since $(\e(m,n))_{(m,n)\in\R}$ is $(M,(\tilde{N},\tilde{N}))$-smooth, we have that $d_{G}(\e(m_{0},n_{0}),\bold{1}_{G})\leq M$. Furthermore, since $\Vert\Psi\Vert_{\mathcal{C}^{2m}(X)}\leq 1$, by the choice of $H(M)$ (property (ii)), $\Vert\Psi'\Vert_{\mathcal{C}^{2m}(X)}\leq H(M)$.
Thus (property (iv))
\begin{equation}\nonumber
   \begin{split}
   \Vert\Psi'|_{Y}\Vert_{\mathcal{C}^{2m}(Y)}\leq H(M)^{2}.
  \end{split}
  \end{equation}
Recall that $(g'(m,n)\cdot e_{X})_{(m,n)\in[\tilde{N}]\times[\tilde{N}]}$ is totally $\omega(M)$-equidistributed in $X$, and by the definition of $\omega$ and $\rho$, we deduce that $(g'(m,n)\cdot e_{Y})_{(m,n)\in[\tilde{N}]\times[\tilde{N}]}$ is totally $\tilde{\sigma}(M)$-equidistributed in $Y$.

Let $z=\int_{\mathbb{T}^{m}}\Psi'\circ\psi dm$ and $\Psi''=\Psi'-z$. By Lemma \ref{1p}, we get that
\begin{equation}\nonumber
   \begin{split}
   \Bigl\vert\mathbb{E}_{(m,n)\in [\tilde{N}]\times[\tilde{N}]}\bold{1}_{P}(m,n)z\chi_{N,u}(m+ni)\Bigr\vert
   \leq c_{1}\Vert\chi_{N,u}\Vert_{U^{2}(\R)}\leq c_{1}\theta_{1}=c_{1}\theta(M_{1})\leq \d_{1}(M_{1})/2\leq \d_{1}(M)/2.
  \end{split}
  \end{equation}
  Thus by (\ref{temp2}), we have that
  \begin{equation}\label{temp3}
   \begin{split}
   \Bigl\vert\mathbb{E}_{(m,n)\in [\tilde{N}]\times[\tilde{N}]}\bold{1}_{P}(m,n)\chi_{N,u}(m+ni)\Psi''(g'(m,n)\cdot e_{Y})\Bigr\vert\geq\d_{1}(M)/2,
  \end{split}
  \end{equation}
where
  \begin{equation}\nonumber
   \begin{split}
   \Vert\Psi''|_{Y}\Vert_{\mathcal{C}^{2m}(Y)}\leq H(M)^{2}, \int_{\mathbb{T}^{m}}\Psi''\circ\phi dm =0.
  \end{split}
  \end{equation}

Recall that $\chi_{N,s}=\chi_{N}*\psi'$, where $\psi'$ is a kernel of $\R$. So we may write $\chi_{N,u}=\chi_{N}*\psi$
for some function $\psi$ on $\R$. Since $\mathbb{E}_{\A\in \R}\psi'(\A)=1$, we have that $\mathbb{E}_{\A\in \R}\vert\psi'(\A)\vert\leq 2$.
We deduce from (\ref{temp3}) that there exists $(m',n')\in [\tilde{N}]\times[\tilde{N}]$ such that
  \begin{equation}\nonumber
   \begin{split}
   &\Bigl\vert\mathbb{E}_{(m,n)\in [\tilde{N}]\times[\tilde{N}]}\bold{1}_{P}((m+m',n+n')\mod [\tilde{N}]\times[\tilde{N}])
   \\&\cdot\chi_{N}(m+ni)\Psi''(g'((m+m',n+n')\mod [\tilde{N}]\times[\tilde{N}])\cdot e_{Y})\Bigr\vert\geq\d_{1}(M)/4,
  \end{split}
  \end{equation}
  where the residue class $(m+m',n+n')\mod [\tilde{N}]\times[\tilde{N}]$ is taken in $[1,\tilde{N}]\times[1,\tilde{N}]$ instead of the more usual set $[0,\tilde{N}-1]\times[0,\tilde{N}-1]$. Then the average
  \begin{equation}\nonumber
   \begin{split}
   \Bigl\vert\mathbb{E}_{(m,n)\in [\tilde{N}]\times[\tilde{N}]}\bold{1}_{P}(m+m'',n+n'')\bold{1}_{J}(m,n)\bold{1}_{[N]\times[N]}(m,n)\chi(m+ni)\Psi''(g'(m+m'',n+n'')\cdot e_{Y})\Bigr\vert
  \end{split}
  \end{equation}
    is at least $\d_{1}(M)'/16$, where the pair $(J,(m'',n''))$ is one of the following 4 combinations:

  (i) $J=[\tilde{N}-m']\times[\tilde{N}-n'], m''=m', n''=n'$;

  (ii) $J=[\tilde{N}-m']\times(\tilde{N}-n',\tilde{N}], m''=m', n''=n'-\tilde{N}$;

  (iii) $J=(\tilde{N}-m',\tilde{N}]\times[\tilde{N}-n'], m''=m'-\tilde{N}, n''=n'$;

  (iv) $J=(\tilde{N}-m',\tilde{N}]\times(\tilde{N}-n',\tilde{N}], m''=m'-\tilde{N}, n''=n'-\tilde{N}$.

  Note that $\bold{1}_{P}(m+m'',n+n'')\bold{1}_{J}(m,n)\bold{1}_{[N]\times[N]}(m,n)=\bold{1}_{P'}(m,n)$ for some $P'$ which is a product of arithmetic progressions. So
    \begin{equation}\nonumber
   \begin{split}
   \Bigl\vert\mathbb{E}_{(m,n)\in [\tilde{N}]\times[\tilde{N}]}\bold{1}_{P'}(m,n)\chi(m+ni)\Psi''(g'(m+m'',n+n'')\cdot e_{Y})\Bigr\vert\geq\d_{1}(M)/16.
  \end{split}
  \end{equation}

Since $\Vert\Psi''|_{Y}\Vert_{Lip(Y)}\leq\Vert\Psi''|_{Y}\Vert_{\mathcal{C}^{2m}(Y)}\leq H(M)^{2}$,
$(g'(m,n)\cdot e_{Y})_{(m,n)\in[\tilde{N}]\times[\tilde{N}]}$ is not totally $\frac{\d_{1}(M)}{16H(M)^{2}}$-equidistributed in $Y$.
This contradicts the fact that $(g'(m,n)\cdot e_{Y})_{(m,n)\in[\tilde{N}]\times[\tilde{N}]}$ is totally $\tilde{\sigma}(M)$-equidistributed in $Y$. Therefore, $\Vert\chi_{N,u}\Vert_{U^{3}(\R)}\leq\e$, completing the proof of Theorem \ref{nU3w}.


\begin{thebibliography}{99}

\bibitem{B3}\textsc{V. Bergelson}, \emph{Ergodic Theory and diophantine problems: Topics in symbolic dynamics and applications}. London Math. Soc. Lecture Note Ser. 279, Cambridge Univ. Press, Cambridge (1996), 167-205.

\bibitem{BL}\textsc{V. Bergelson, A. Leibman}, \emph{Polynomial extensions of van der Waerden's and Szemer\'{e}di theorems}. J. Amer. Math. Soc. $\bold{9}$ (1996), no. 3, 725-753.

\bibitem{BM2}\textsc{V. Bergelson, R. McCutcheon}, \emph{Recurrence for semigroup actions and a non-commutative Schur theorem}. Contemporary Mathematics, $\bold{65}$ (1998), 205-222.

\bibitem{FH}\textsc{N. Frantzikinakis, B. Host}, \emph{Higher order Fourier analysis of multiplicative functions and applications}. arXiv: 1403.0945.

\bibitem{FHel}\textsc{N. Frantzikinakis, B. Host}, \emph{Uniformity of multiplicative functions and partition regularity of some quadratic equations}. arXiv: 1303.4329.

\bibitem{7}\textsc{H. Furstenberg}, \emph{Ergodic behavior of diagonal measures and a theorem of Szemer\'{e}di on arithmetic
progressions}. J. Analyse Math. $\bold{31}$ (1977), 204-256.

\bibitem{G}\textsc{T. Gowers}, \emph{A new proof of Szemer\'{e}di theorem}. Geom. Funct. Anal. $\bold{11}$ (2001), 465-588.

\bibitem{14}\textsc{T. Gowers}, \emph{Decompositions, approximate structure, transference, and the Hahn-Banach theorem}. Bulletin London Math. Soc. $\bold{42}$ (2010), no. 4, 573-606.

\bibitem{15}\textsc{T. Gowers, J. Wolf}, \emph{Linear forms and quadratic uniformity for functions on $\mathbb{Z}_{N}$}. J. Anal. Math.
$\bold{115}$ (2011), 121-186.

\bibitem{22}\textsc{B. Green, T. Tao}, \emph{An arithmetic regularity lemma, associated counting lemma, and applications}. An irregular mind, Bolyai, Soc. Math. Stud. $\bold{21}$, Janos Bolyai Math. Soc., Budapest, (2010), 261-334.

\bibitem{5}\textsc{B. Green, T. Tao}, \emph{Linear equations in primes}. Ann. of Math. (2) $\bold{171}$ (2010), no. 3, 1753-1850.

\bibitem{er}\textsc{B. Green, T. Tao}, \emph{On the quantitative distribution of polynomial nilsequences - erratum}. Ann. of Math, $\bold{179}$ (2014), no. 3, 1175-1183.

\bibitem{er2}\textsc{B. Green, T. Tao}, \emph{On the quantitative distribution of polynomial nilsequences - erratum}. arXiv:1311.6170.

\bibitem{65}\textsc{B. Green, T. Tao}, \emph{The primes contain arbitrarily long arithmetic progressions}. Ann. of Math. $\bold{167}$ (2008),
    481-547.

\bibitem{7}\textsc{B. Green, T. Tao}, \emph{The Mobius function is strongly orthogonal to nilsequences}. Ann. of Math. (2) $\bold{175}$ (2012), no. 2,
    541-566.

\bibitem{GT}\textsc{B. Green, T. Tao}, \emph{The quantitative behaviour of polynomial orbits on nilmanifolds}. Ann. of. Math. (2) $\bold{175}$ (2012), no. 2, 465-540.

\bibitem{6}\textsc{B. Green, T. Tao, T. Ziegler}, \emph{An inverse theorem for the Gowers $U^{s+1}$-norm}. Ann. of Math. (2) $\bold{176}$ (2012), no. 2,
    1231-1372.

\bibitem{Gau}\textsc{G.H. Hardy, Ramanujan}, \emph{Twelve lectures on subjects suggested by his life and work}. 3rd ed. New York: Chelsea, (1999), p.67.

\bibitem{HK}\textsc{B. Host, B. Kra}, \emph{Nonconventional ergodic averages and nilmanifolds}. Ann. of Math. (2) $\bold{161}$ (2005), no. 1, 397-488.

\bibitem{KS}\textsc{A. Khalfalah, E. Szemeredi}, \emph{On the number of monochromatic solutions of $x+y=z^{2}$}. Combin.
Probab. Comput. $\bold{15}$ (2006), no. 1-2, 213-227.

\bibitem{GP}\textsc{J. Neukirch}, \emph{Algebraic number theory}. New York: Springer, 1999. Print.

\bibitem{Rado}\textsc{R. Rado}, \emph{Studien zur Kombinatorik}. Math. Z. $\bold{36}$ (1933), no. 1, 424-470.

\bibitem{25}\textsc{A. Sark\"{o}zy}, \emph{On difference sets of integers. III}. Acta Math. Acad. Sci. Hungar. $\bold{31}$ (1978), no. 3-4, 355-386.

\bibitem{Sz}\textsc{B. Szegedy}, \emph{On higher order Fourier analysis}.arXiv: 1203.2260.

\bibitem{T}\textsc{T. Tao}, \emph{A quantitative ergodic theory proof of Szemeredi's theorem}. Electron. J. Combin. (2) $\bold{13}$ (2006), no. 1, Research Paper 99, 49 pp.

\end{thebibliography}
\end{document}